\documentclass[11pt,reqno]{article}
\usepackage{amsfonts}
\usepackage{mathtools}
\usepackage{amsmath}
\usepackage{amsthm}
\usepackage{amssymb}
\usepackage{comment}
\usepackage{amscd}
\usepackage{cite}
\usepackage[ngerman,english]{babel}
\usepackage[mathscr]{eucal}
\usepackage{indentfirst}
\usepackage{graphicx}
\usepackage{graphics}
\usepackage{pict2e}
\usepackage{epic}
\numberwithin{equation}{section}
\usepackage[margin = 3.5cm, top = 2.5cm, bottom = 2.5cm]{geometry}
\usepackage{epstopdf} 
\usepackage{bbm}
\usepackage{xcolor}
\usepackage{comment}
\usepackage{tikz}
\usepackage{mathtools}
\usepackage{enumitem}

\usepackage{authblk}

\let\OLDthebibliography\thebibliography
\renewcommand\thebibliography[1]{
	\OLDthebibliography{#1}
	\setlength{\parskip}{0pt}
	\setlength{\itemsep}{3pt}
}
\definecolor{myblue}{HTML}{26649b}
\definecolor{mygreen}{HTML}{46C755}
\definecolor{myorange}{HTML}{FF8C00}

\mathtoolsset{showonlyrefs=true}

\usepackage{hyperref}
\hypersetup{ colorlinks=true,
    allcolors= [RGB]{0,0,128} 
}

\numberwithin{equation}{section}

\theoremstyle{plain}
\newtheorem{Th}{Theorem}[section]
\newtheorem{Lemma}[Th]{Lemma}
\newtheorem{Cor}[Th]{Corollary}
\newtheorem{Prop}[Th]{Proposition}
\DeclareMathOperator{\R}{\mathbb{R}}
\DeclareMathOperator{\Z}{\mathbb{Z}}
\DeclareMathOperator{\C}{\mathbb{C}}
\DeclareMathOperator{\N}{\mathbb{N}}

\DeclareMathOperator{\cN}{\mathcal{N}}

\DeclareMathOperator{\pc}{\text{pc}}

\newcommand{\f}[2]{\frac{#1}{#2}}

\newcommand{\el}{\ell}
\newcommand{\sphere}{\mathbb{S}}
\newcommand{\conv}{\ast}

\def \re{\mathrm{Re}}
\def \imm{\mathrm{Im}}
\def \d{\mathrm{d}}
\def \dt{\mathrm{d}t}
\def \G{\mathcal{G}}
\def \H{\mathcal{H}}

\def \J{\mathcal{J}}

\theoremstyle{definition}
\newtheorem{Def}[Th]{Definition}

\newtheorem{Rem}[Th]{Remark}
\newtheorem{?}[Th]{Problem}

\newcommand{\abs}[1]{\left \lvert #1 \right \rvert}

\setlist[enumerate]{leftmargin=10mm}

\renewcommand{\thefootnote}{\fnsymbol{footnote}}
\title{Multisoliton solutions and blow up for the\\ $L^2$-critical Hartree equation }
\author[1]{Jaime G\'omez}
            \author[2]{Tobias Schmid}
            \author[3]{Yutong Wu}
            
\affil[1]{EPFL SB MATH PDE, Bâtiment MA, Station 8, CH-1015 Lausanne, \linebreak Switzerland\; \href{mailto:jaime.gomezramirez@epfl.ch}{jaime.gomezramirez@epfl.ch}\;\\ \vspace{.2cm} }
\affil[2]{University of Vienna, Faculty of Mathematics, Oskar-Morgenstern-Platz 1, A-1090 Vienna, \linebreak Austria\; \href{mailto:tobias.schmid@univie.ac.at}{tobias.schmid@univie.ac.at}\;\\ \vspace{.2cm} }
                    
\affil[3]{Department of Mathematics, Yale University, New Haven, CT 06511, USA \linebreak
					 \href{mailto:yutong.wu.yw894@yale.edu}{yutong.wu.yw894@yale.edu}}

\date{}                     

\begin{document}
    
	\begin{titlepage}

\maketitle
\thispagestyle{empty}

\;\\
\;\\
\renewcommand{\abstractname}{Abstract}
\begin{abstract}
\noindent
We construct multisoliton solutions for the $L^2$-critical Hartree equation with trajectories asymptotically obeying a many-body law for an inverse square potential. Precisely, we  consider  the  $m$-body hyperbolic and parabolic non-trapped dynamics.  The  pseudo-conformal symmetry then implies finite-time collision blow-up in the latter case and a solution blowing up at $m$ distinct points in the former case. The approach we take is based on the ideas of Krieger-Martel-Rapha\"el \cite{KRM} and the third author's recent extension \cite{Wu}. The approximation scheme requires new aspects in order to deal with a certain degeneracy for generalized root space elements. 
\end{abstract}
 \;\\   
 \;\\
 \;\\
 \renewcommand{\abstractname}{Acknowledgements}
\begin{abstract}
\noindent
\hspace{.8cm}We would like to thank Joachim Krieger for helpful comments on the article.
\end{abstract}
\;\\
\renewcommand{\abstractname}{Statements and Declarations}
\begin{abstract}
\noindent
 The authors have no competing interests to declare that are relevant to the content of this article.
\end{abstract}
\renewcommand{\abstractname}{Data Availability Statement}
\begin{abstract}
\noindent
Data sharing is not applicable to this article as no datasets were
generated or analyzed during the current study. 
\end{abstract}

\renewcommand{\abstractname}{Abstract}

\end{titlepage}
\renewcommand{\thefootnote}{\arabic{footnote}}

\section{Introduction}

\subsection{Setting of the problem} We consider the nonlinear Hartree equation in $ d = 4$ dimensions, i.e. the following cubic Schr\"odinger equation with Hartree nonlinearity:
	\begin{align}\label{H} 
		\begin{cases}
			\;\; i \partial_t u + \Delta u - \phi_{|u|^2} u = 0,\; \;\; &\\[2pt]
			\;\; u(0, x) = u_0(x),\; u_0 : \R^{4} \to \C,\;\; (t,x) \in [0,T)\times\R^{ 4}, &
		\end{cases}
	\end{align}
	where the function $\phi_{|u|^2}$ is given by 
	\begin{align} \label{def-conv}
		\phi_{|u|^2} = - \frac{1}{|x|^2} \ast |u|^2,\;\; x \in \R^4,
	\end{align}
	and  $\ast$ denotes the usual convolution. Note that for $ \Phi = (4\pi^2)^{-1}\phi_{|u|^2}$ we have $ \Delta  \Phi = |u|^2 $ and we use \eqref{def-conv} in \eqref{H} for convenience.
    It is well known, see \cite{Caz,Ginibre-Velo}, that \eqref{H} is locally well-posed in $ H^1(\R^4)$.  In particular \eqref{H} is a Hamiltonian equation for which $H^1$-solutions conserve \emph{energy, mass} and \emph{momentum}
	\begin{gather}
		\text{\emph{Hamiltonian\;energy}}:\;\;\hspace{0cm}\mathcal{E}(u(t)) = \f12 \int |\nabla u(t)|^2 \;\d x - \frac{1}{16 \pi^2 } \int | \nabla \phi_{|u(t)|^2}|^2 \; \d x  = \mathcal{E}(u_0),\\[3pt]
		\text{\emph{$L^2$-norm (mass)}}:\;\; \hspace{0cm}\int |u(t)|^2 \;\d x = \int |u_0|^2\;\d x,\\[3pt]
		\text{\emph{Momentum}}:\;\; \hspace{0cm}\text{Im}\bigg( \int  \nabla u(t) \cdot \overline{u(t)}\;\d x \bigg) =  \text{Im}\bigg( \int \nabla u_0 \cdot \overline{u_0}\;\d x \bigg),
	\end{gather}
	and further shares many symmetries with the local NLS counterpart. If $ u(t,x) $ is a solution of \eqref{H} and $ \lambda \in \R_+, t_0 \in \R, \alpha, \beta \in \R^4, \gamma \in \R$, then $u_{\lambda, t_0, \alpha, \beta, \gamma}(t,x)$ defined via 
	\begin{align} \label{inv}
		u(t,x) \mapsto \lambda^2 u( \lambda^2 t + t_0, \lambda x  - \alpha - \beta t) e^{i(\f12 \beta \cdot x - \f14 |\beta|^2 t + \gamma)}
	\end{align}
	is also a solution of \eqref{H} given the appropriate transformation of initial data.\\[12pt]
	\emph{Background}. The 3D version of \eqref{H} was derived as a mean-field limit of many body quantum systems with great implications ever since. We refer for instance to \cite{Fr-Lenz} and the references therein. In fact nonlinear Schr\"odinger-type equations of the form
	\begin{align} \label{this-H}
		i \partial_t u + \Delta u = -  c_d \big( |x|^{2-d} \ast |u|^2\big)u, \;\;\; \text{in}\;  [0, T)\times \R^d,\; c_d > 0
	\end{align}
	arise rather naturally as effective descriptions related to e.g. Bose gases, cf. \cite{Fr-Lenz}. In comparison to the local NLS, the  convolution term on the right of \eqref{this-H} is essential to \emph{capture long-range effects} and necessary to accurately reproduce interactions of a quantum system with  e.g.  the Coulomb potential $ |x|^{-1}$ in $ d = 3$. Similar versions of \eqref{this-H} with relativistic (nonlocal) dispersion arise in connection to mean-field approximations of Boson stars, see \cite{El-Schlein}.\\[8pt]
	\emph{Well-posedness and blow-up}.  The local solution $u \in C^0([0,T), H^1(\R^4))$  of  \eqref{H} with $u_0 \in H^1(\R^4)$ is inextensible beyond $ T > 0$  if and only if $ \lim_{t \to T^-} \| \nabla u \|_{L_x^2}^2 = \infty$,  i.e. there is blow-up at $t = T$.  In fact, the existence of finite-time blow-up is immediate using Glassey's implicit argument \cite{Glassey}. 
    The solutions of \eqref{H} with finite variance, i.e. such that $u_0 \in \Sigma$ with
	$
	\Sigma = H^1 \cap L^2(|x|^2 \d x)
	$
	blow up in finite time due to the \emph{virial identity}
	\begin{align*}
		&\frac{\d^2}{\d t^2} \int |x|^2|u(t,x)|^2\;\d x	 = 16 \mathcal{E}(u_0)
	\end{align*}
	if $ \mathcal{E}(u_0)< 0$. The latter implies $ \| u_0 \|_{L^2_x}$ to be rather large. In fact, similarly to the $L^2$-critical local NLS, we observe threshold behavior through the \emph{ground state mass}.\\[3pt]
	Equation \eqref{H} has solutions $u_{\omega}(t,x)= e^{i \omega^2 t} Q_{\omega}(x)$, where $ Q_{\omega} = \omega^2 Q_1(\omega \cdot) $ and $Q := Q_1$ is the unique $H^1$-solution of 
	\begin{align}
		\Delta Q   - \phi_{|Q|^2} Q =  Q,\;\; Q(x) > 0,\; x \in \R^4.
	\end{align}
	Here, existence follows along the variational argument of Weinstein \cite{Weinstein} and uniqueness was proved by E. Lieb \cite{Lieb} (in $d=3$), in fact see \cite[Section 4]{KLR}   for the case $ d=4$. Note that this family of solutions (up to symmetry) is not orbitally stable and decays exponentially, i.e.
	\[
	Q(x) \leq C e^{- c |x|},\;\; x \in \R^4.
	\] 
	By using Gagliardo-Nirenberg interpolation combined with the Hardy-Littlewood-Sobolev inequality, we observe the bound (see e.g. \cite{class-min})
	\[
	\mathcal{E}(u) \geq \f12 \| \nabla u \|_{L^2_x}^2 \bigg(1 - \frac{\| u \|^2_{L^2_x}}{\| Q\|^2_{L^2_x}}\bigg),\;\; u \in H^1(\R^4),
	\]
	which implies global solutions for \eqref{H}  if $ u_0 \in H^1(\R^4) $ and $ \| u_0 \|_{L^2_x} < \| Q \|_{L^2_x}$. The latter is sharp due to \emph{pseudo-conformal symmetry}, i.e. if $u(t,x) $ solves \eqref{H}, then so does 
	\begin{align}\label{pseudo}
		u(t,x) \mapsto \frac{1}{t^2} \overline{u}(\frac{1}{t}, \frac{x}{t})e^{i \frac{|x|^2}{4t}}.
	\end{align}
	In particular, applying the symmetry to $e^{it}Q(x)$ and using time reflection $ u(t) \mapsto \overline{u}(t_0 -t)$ for \eqref{H},  we obtain \emph{minimal mass blow-up solutions} 
	\begin{align}
		&S(t,x) = \frac{1}{t^2} Q(\frac{x}{t}) e^{i( -  \frac{1}{t} +  \frac{|x|^2}{4t})},\;\; \| S(t) \|_{L^2_x} = \| Q\|_{L^2_x},\\[4pt]\nonumber
		&\|\nabla S(t) \|_{L^2_x} \sim |t|^{-1},\; t \to 0^-.
	\end{align}
  \;\\
	We note that \cite[Theorem~1.2]{class-min} provides a classification of such solutions, which is similar to \cite{M-dete} for the NLS. 
    For radial data,  Krieger-Lenzmann-Rapha\"el \cite{KLR} proved the finite-codimension stability of $S(t)$. Further the existence of blow-up at the rate  $\| \nabla u(t) \|_{L^2_x} \sim |t|^{-1}$ was proved for certain inhomogeneous perturbations \footnote{analogue to asking for minimal mass blow-up for \emph{inhomogeneous} NLS, see \cite{Ra-Sze},\cite{Merle-none}.} of the Newton potential $ |x|^{-2}$. Moreover,  formal and numerical calculations in \cite{YRZ-Hartree} show comparable blow-up mechanisms as for the $L^2$-critical NLS.\\[12pt]
    \emph{\textbf{In this article}}, we construct \emph{multisoliton  solutions of \eqref{H} with modulation trajectories following non-trapped solutions of an $N$-body problem}, which is derived from the Newton potential $|x|^{-2}$ in $d=4$ dimensions.\\[3pt]
	This method was pioneered for the 3D Hartree-NLS by Krieger-Martel-Rapha\"el \cite{KRM} for the construction of two-soliton solutions along asymptotically parabolic and hyperbolic orbits.  This approach was recently extended by the third author \cite{Wu} to 3D many-body solutions with (purely) parabolic, hyperbolic and hyperbolic-parabolic (cluster) cases.\\[4pt]
	\emph{Further, by applying the pseudo-conformal transformation to our solutions, we obtain}
    \begin{itemize}  \setlength\itemsep{.2cm}
        \item[$\bullet$]\emph{Finite-time collision blow-up} at a single point with the rate $ \|  \nabla u(t) \|_{L^2_x} \sim |t|^{-1}$ as $ t \to 0^-$ (for asymptotically parabolic orbits) and
		\item[$\bullet$]\emph{Finite-time blow-up at arbitrary $ m \geq 2$ distinct points} with  $ \|  \nabla u(t) \|_{L^2_x}  \sim |t|^{-1}$ as $ t \to 0^-$ (for asymptotically hyperbolic orbits).
    \end{itemize}
    \;\\[-.5cm]
    The results of this article are stated in Section \ref{sec:result}.
     \emph{We stress the result is technically different form the 3D case} in \cite{Wu} (see also \cite{KRM}) due to the pseudo-conformal invariance for the $L^2$-critical equation \eqref{H} and the related extended class of generalized root space elements. In particular by the required orthogonality in the estimate Lemma \ref{lem coercivity} and the vanishing $(Q, \Lambda Q) = 0$, where  $\Lambda Q = (\partial_\omega Q_{\omega})_{|_{\omega =1}}$, this includes the need for two extra modulation parameters, see Section \ref{sec:gs-op-properties} and Section \ref{sec:approx}. Further we need extra virial terms for the energy control in Section \ref{sec:mod}. For more details we refer to the explanations in Section \ref{sec:approx}, Section \ref{sec:bootstrap} and Section \ref{sec:mod-subsec2}.\\[14pt]
     Let us first explain some related aspects of the $L^2$-critical NLS \eqref{mass-NLS} and the terminology of the Kepler-type $N$-body law we encounter below.\\[8pt] 
	\emph{$L^2$-critical NLS}: The above results are reminiscent of the structure of the local NLS 
	\begin{align} \label{mass-NLS}
		i \partial_t u + \Delta u  = - |u|^{\frac{4}{d}}u, \;\;\text{in}\; \R^{1+d}.\;\; \
	\end{align}
	In comparison to \eqref{H}, however, the $L^2$-critical NLS \eqref{mass-NLS} is much better understood in terms of blow-up in $H^1$ and solitary dynamics, as well as their respective stability.\\[6pt] First considering the dynamics of a \emph{single bubble}, 
	Bourgain-Wang \cite{BW} constructed blow-up solutions (at pseudo-conformal rate) in the super-threshold regime, i.e. 
	\begin{align} \label{super-threshold}\|Q\|_{L^2_x} < \|u_0\|_{L^2_x} < \|Q\|_{L^2_x} + \epsilon 
	\end{align}
	where $ \epsilon \ll1 $, by certain weak perturbations of $S(t)$.   
	These were proved to be unstable, in fact \emph{finite-codimension stable}, see \cite{MRS} and also  \cite{Krieger-Schlag} for such solutions in $d=1$.\\[4pt]
	Further, in a series of works, \cite{Pere},\cite{Merle-Rapha2},\cite{Merle-Rapha3},\cite{Merle-Rapha4},\cite{Merle-Rapha5},\cite{raphael}, a \emph{stable blow-up regime} at the rate 
	\[ \| \nabla u(t) \|_{L^2_x} \sim \sqrt{\frac{\log(|\log(t)|)}{t}},\;\; t \to 0^+ \]
	has been identified and subsequently well understood (c.f. \cite{Merle-Rapha1}), where 
	$\mathcal{E}(u_0) < 0, \;u_0 \in \Sigma$ (and of course \eqref{super-threshold}). See also \cite{YRZ} for higher dimensions. Note that by the work of Rapha\"el \cite{raphael}, it is known that this rate is isolated in the sense that as $t \to 0^-$, \emph{all other finite-time blow-up solutions $u$ in $H^1$} with mass slightly above $\| Q\|_{L^2}^2$ must satisfy  the bound
	\[
	\| \nabla u(t) \|_{H^1} \gtrsim \frac{1}{|t|},\;\; t \to 0^-.
	\]
	\emph{Multi-bubble}. In \cite{Merle-k-bl} Merle proved the existence of \emph{multi-bubble solutions} ($d=2$) blowing up in finite time at arbitrary $K \in \Z_{\geq 2}$ distinct points with the rate $ \| \nabla u(t) \|_{L^2_x} \sim |t|^{-1}$. Strongly interacting multi-bubbles moving  along the vertices  of a regular  polygon have been constructed by Martel-Rapha\"el \cite{MR} ($d=2$).  In fact the bubbles contract into a collision blow-up with speed above the conformal rate $\|\nabla u(t) \|_{L^2_x} \sim |\frac{\log(|t|)}{t} | $ as $ t \to 0^-$. Note that pseudo-conformal symmetry transfers the analysis to  the asymptotic of the solutions as $ t \to \infty$ and vice versa.\\[8pt]
	\emph{Multisoliton solutions} contribute to the understanding of long time and blow-up behavior of dispersive equations. For the $L^2$-subcritical NLS we refer for instance to \cite{MM},  \cite{Koch-T}, \cite{RodSS}, \cite{Pere2} for construction, stability and collision analysis of such solutions. We further note there is a rich literature in the context of KdV and nonlinear Klein-Gordon equations.

	\subsection{\texorpdfstring{$m$}{m}-body dynamics} \label{sec:n-body-intro}

    Let $ m \in \Z_{\geq 2} $ and $ \alpha = (\alpha_1, \alpha_2, \dots, \alpha_m) \in \R^{ 4m} $ be the position configuration of $ m $ bodies with center $ \alpha_j \in \R^4$.   We define our \emph{$m$-body law} to be the Newtonian system 
	\begin{align} 
    \label{m-body}
		\begin{cases}
			\;\;\dot{\alpha}_j(t) =  2\beta_j(t),\; \;1 \leq  j \leq m,\;&\;\;\\[5pt]
			\;\;\displaystyle	\dot{\beta}_j(t) =  -  \| Q \|^2_{L^2_x} \sum_{ k \neq j }   \frac{\alpha_j - \alpha_k}{|\alpha_j - \alpha_k|^4}.&
		\end{cases}
	\end{align}
	The system \eqref{m-body} has a  first integral $H = K - U$, i.e.
	\begin{align}\label{energy}
	H(\alpha, \beta) =  2\sum_{j =1}^m |\beta_j(t)|^2 - \| Q \|^2_{L^2_x} \sum_{j < k} \frac{ 1}{|\alpha_j - \alpha_k|^2},
	\end{align} 
    with \emph{kinetic-} and \emph{potential energy} respectively.
    Similar to \cite{KRM,Wu}, we will use the trajectories $(\alpha(t), \beta(t))$ of \eqref{m-body} as a leading asymptotic of the modulation variables according to \eqref{inv} applied to the ground state $Q$.
	Clearly, the \emph{center of mass} 
    \begin{align}\label{center-of-mass}
    M(t) =\sum_{j=1}^m  \alpha_j(t)
    \end{align}
    evolves by free Galilean motion, i.e. $\ddot{M}(t) = 0$, and $\alpha(t) $ is confined to a hyperplane. In fact we shall define the \emph{central configurations} and the \emph{non-collision set} by
    \;\\
	\begin{align*}
		&\mathcal{X} : = \big \{ x = (x_1, x_2, \dots, x_m) \in \R^{4m}\;|\; { \textstyle   \sum_{j =1}^m  x_j = 0 } \big \},\;\; \mathcal{Y} : = \mathcal{X} \setminus \Delta,\\[3pt]
		& \Delta : = \big \{ x = (x_1, x_2, \dots, x_m) \in \mathcal{X}\;|\; \exists\; j \neq k: x_j = x_k \big \},
	\end{align*}
    \begin{Rem}
    We note that the assumption of vanishing center of mass is of course neither a restriction to \eqref{m-body}, since we simply need to add the linear motion of $M$ to each $\alpha_j(t)$, nor is it to the solutions of \eqref{H} in Theorem \ref{thm hyperbolic} and Theorem \ref{thm parabolic} due to \eqref{inv}.
    \end{Rem}
Let us also set  $\alpha_{j k} = \alpha_j - \alpha_k$ and the \emph{minimal distance} $a:= \min_{j < k} |\alpha_j- \alpha_k|$. We have the following along the arguments in \cite{Marchal-Saari}. 

\begin{Lemma}\label{prelim} Let $\alpha(t)$ be a global solution of \eqref{m-body}. Then $ \max_{j < k}|\alpha_{j k}(t)| = O(t)$ as $ t \to \infty$ and $ \liminf_{t \to \infty} a(t) > 0$. Furthermore, for some $ v \in \mathcal{X}$ we have $\alpha(t) = v t + O(t^{\f12})$ as $t \to \infty$.
The vector $ v = \lim_{t \to \infty} \frac{\alpha(t)}{t} \in \mathcal{X}$ is called the \emph{limit velocity} of $(\alpha, \beta)$.
\end{Lemma}
    \;\\
We now turn to global solutions for which \emph{all mutual distances grow} as $ t \to \infty$.\\[8pt]
\emph{Expansive orbits}. We call solutions $(\alpha(t), \beta(t))$  of \eqref{m-body} \emph{expansive} 
if they are \emph{global} and $ a(t) \to \infty$ for $ t \to \infty$. In fact, we introduce the following definitions according to the classification of trajectories in the classic $m$-body problem, see e.g. \cite{Wu}. 

\begin{Def} \label{def-traje-here} \phantom{a}
\begin{enumerate} \setlength\itemsep{4pt}
    \item If $v \in \mathcal{Y}$ we say $(\alpha(t), \beta(t))$ is \underline{\emph{hyperbolic}}, i.e. $ |\alpha_j(t) - \alpha_k(t)| \sim t$ as $ t \to \infty$ for all $ j \neq k$.
    \item If $v = 0$ we say $(\alpha(t), \beta(t))$ is \underline{\emph{parabolic}} in case $|\alpha_j(t) - \alpha_k(t)| \sim t^{\f12} $  for all $ j \neq k$.
    \item If $v \in \Delta \setminus \{0\}$ we say $(\alpha(t), \beta(t))$ is \underline{\emph{hyperbolic-parabolic}} in case the set of all $(j,k) $ with $j \neq k$ and  $ |\alpha_j(t) - \alpha_k(t)|\sim t^{\f12}$  is non-empty and otherwise the asymptotic is as in $(1)$. 
    \end{enumerate}
\end{Def}
   \;\\
 \noindent
The following proposition completely characterizes the asymptotic behavior of hyperbolic trajectories.
\begin{Prop}[hyperbolic solutions] \label{hyp} \phantom{a}
\begin{enumerate}
    \item Let $(\alpha(t), \beta(t))$ be a hyperbolic solution of \eqref{m-body}. Then there exist $v \in \mathcal{Y}$ and $x \in \mathcal{X}$ such that
    \begin{equation}
        \alpha(t)= x+ vt + o(1),\;\;\beta(t)=  \frac{v}{2}+ o(1),\;\;\text{as}\;\; t \to \infty.
    \end{equation}

    \item Let $v \in \mathcal{Y}$ and $x \in \mathcal{X}$. Then for $T_0>0$ large enough, there exists a solution $(\alpha(t), \beta(t))$ of \eqref{m-body} for $ t \geq T_0$ with
    \begin{equation}
        \alpha(t)= x+ vt + o(1),\;\;\beta(t)=  \frac{v}{2}+ o(1),\;\;\text{as}\;\; t \to \infty.
    \end{equation}
\end{enumerate}
\end{Prop} 

\noindent
For the parabolic case, we explicitly construct a family of solutions.

\begin{Prop}[parabolic solutions] \label{parab} 
Let $\eta \ge 0$ and $b \in \mathcal{Y}$ be a minimizer of 
$$U(x) = \f12 \|Q\|_{L^2}^2 \sum_{j < k} \frac{1}{|x_j - x_k|^2}$$ 
with the constraint $ \sum_{j} |x_j|^2 =1$. Then the following $(\alpha(t), \beta(t))$ solves \eqref{m-body} for $ t \ge 0$:
\[
\alpha(t) = \big( 4\sqrt{U(b)}\ t + \eta \big)^{\f12}b, \quad \beta(t)= \frac{1}{2} \dot{\alpha}(t).
\]

\end{Prop}

The proofs of the above two propositions are given in Appendix \ref{sec:furtherdetailsmbod}.

\subsection{Statement of the results} \label{sec:result}
	The following theorems are the main results of this article and state the existence of global multisoliton solutions with trajectories derived from solutions to \eqref{m-body} in Definition \ref{def-traje-here}.
 
	\begin{Th} \label{thm hyperbolic}
		Let $(\alpha^\infty, \beta^\infty)$ be a hyperbolic solution to \eqref{m-body} and $\lambda_j^\infty>0$. Then there exist a solution $ u \in C^0( [0, \infty), H^1(\R^4)) $ to \eqref{H} and $\gamma^\infty(t)$ such that 
		\begin{equation}
			\bigg \| u(t,\cdot)- \sum_{j=1}^m Q_{\lambda_j^{\infty}} \Big(  x-\alpha_j^\infty(t) \Big) e^{i\gamma_j^\infty(t)+ i\beta_j^\infty(t) \cdot x} \bigg \|_{H^1}= o(t^{-1+}) \quad \text{as } t \to +\infty.
		\end{equation}
	\end{Th}

	\begin{Th} \label{thm parabolic}
		Let $(\alpha^\infty, \beta^\infty)$ be a parabolic solution to \eqref{m-body} 
	 and $\lambda^\infty>0$. Then there exist a solution $ u \in C^0( [0, \infty), H^1(\R^4))$ to \eqref{H} and $\gamma^\infty(t)$ such that 
		\begin{equation}
			\bigg \| u(t,\cdot)- \sum_{j=1}^m Q_{\lambda^{\infty}} \Big(  x-\alpha_j^\infty(t) \Big) e^{i\gamma_j^\infty(t)+ i\beta_j^\infty(t) \cdot x} \bigg \|_{H^1}= o(t^{-1+}) \quad \text{as } t \to +\infty.
		\end{equation}
	\end{Th}

The following are direct consequences of Theorem \ref{thm hyperbolic} and Theorem \ref{thm parabolic} combined with the pseudo-conformal symmetry.
 
\begin{Cor}[Finite-time blow-up at distinct points] \label{blow-up-dist} 
Let $m \geq 2$, $ v_1, v_2, \dots, v_m \in \R^4$ be arbitrary distinct points and $\lambda_j > 0$ for $ j = 1,2, \dots, m$. Then \eqref{H} has a solution $ u \in C^0((- \infty,0), H^1(\R^4))$ that blows up at $t=0$ such that
\begin{align*}
	\|\nabla u(t) \|_{L^2_x} \sim |t|^{-1} \;\; \text{and} \;\; |u(t)|^2 \rightharpoonup \sum_{j=1}^m \|Q\|_{L^2_x}^2 \delta_{v_j} \;\; \text{as} \;\; t \to 0^-.
\end{align*}
Moreover, there exist $\alpha_j, \beta_j, \gamma_j \in C^0(- \infty, 0)$ with 
\begin{equation}
    \alpha_j(t) = v_j + O(|t|), \;\; \beta_j(t)= \frac{v_j}{2|t|} + O(1) \;\; \text{as} \;\; t \to 0^-,
\end{equation} 
such that when writing 
\begin{equation} \label{eq expansion pseudo}
    u(t, x) =  \sum_{j =1}^m \frac{1}{t^2} Q_{\lambda_j}\Big( \frac{ x- \alpha_j(t) }{t}\Big) e^{i( \gamma_j(t) +  \beta_j(t) \cdot x + \frac{|x|^2}{4t})}  + \varepsilon(t,x),
\end{equation}
we have $\displaystyle \lim_{t \to 0^-}\| \varepsilon(t) \|_{L^2 \cap L^{4^-}} = 0$.
\end{Cor}

\begin{Cor}[Finite-time collision blow-up] \label{blow-up-col}
Let $m \ge 2$ and $(\tilde{\alpha}(t), \tilde{\beta}(t))$ be a parabolic solution of \eqref{m-body} and $ \lambda > 0$. Assume $\tilde{\alpha}_j= b_j t^{\f12}+ O(1)$ as $t \to +\infty$. Then \eqref{H} has a solution $ u \in C^0((- \infty, 0), H^1(\R^4))$ that blows up at $t=0$ such that 
\begin{align*}
	\|\nabla u(t) \|_{L^2_x} \sim |t|^{-1} \quad \text{and} \quad |u(t)|^2 \rightharpoonup  m\|Q\|_{L^2_x}^2 \delta_0 \;\; \text{as} \;\; t \to 0^-.
\end{align*}
Moreover, there exist $\alpha_j, \beta_j, \gamma_j \in C^0(-\infty, 0)$ with
\begin{equation}
    \alpha_j(t) = b_j |t|^{\f12}+ O(|t|), \;\; \beta_j(t)= - \frac{b_j}{4} |t|^{- \f12} + O(|t|^{\f12}) \;\; \text{as} \;\; t \to 0^-,
\end{equation}
such that when writing
\begin{align*}
	u(t, x) = \sum_{j =1}^m   \frac{1}{t^2}Q_{\lambda}\Big( \frac{ x- \alpha_j(t) }{t}\Big) e^{i( \gamma_j(t) + \beta_j(t) \cdot x + \frac{|x|^2}{4t})} + \varepsilon(t,x),
\end{align*}
we have $\displaystyle \lim_{t \to 0^-}\| \varepsilon(t) \|_{L^2 \cap L^{4^-}} = 0$.
\end{Cor}

\;\\
	A few remarks are in order.\;
	\begin{Rem}
    \begin{enumerate} [label=(\roman*)]  \setlength\itemsep{-1pt}
    \item In the Corollaries we use the notation $L^{p^-} = \bigcap_{0 < \epsilon \ll1} L^{p - \epsilon}$.
        \item The solutions in Theorem \ref{thm hyperbolic} and Theorem \ref{thm parabolic} are of the strong interaction regime, in the sense described by Martel-Raph\"ael \cite{MR}. This means the trajectory of each soliton depends to leading order on the presence of the other solitons. 
        \item Theorem \ref{thm hyperbolic} and the blow-up in Corollary \ref{blow-up-dist} relate to Merle's solutions \cite{Merle-k-bl}. The prescribed distinct points are limit velocities 
        of the hyperbolic orbits, 
        which lie outside of the collision set (they must be chosen distinct). 
        \item  The assumption in Theorem \ref{thm parabolic} that the masses $\lambda_j$ are identical is the same as in the statements in \cite{KRM,Wu} concerning parabolic trajectories. We expect this to be a technical assumption, which in our argument implies necessary decay estimates in Section \ref{sec:mod} (see the end of the proof of Proposition \ref{prop estimate on G(epsilon)}).
    \end{enumerate}
	\end{Rem}

    \begin{Rem} (i) Taking the $m$-body trajectories of Proposition \ref{hyp} in Theorem \ref{thm hyperbolic}, we obtain the following.
        For any $x_j, v_j \in \R^4$, $\lambda_j > 0$, $j = 1, \dots m$ with $v_j \neq v_k$ if $ j \neq k$, there exist $\gamma_j(t) = O(t) $ and a global solution $u(t) $ of \eqref{H} such that 
            \begin{align} \label{linear-traject}
& \bigg \|u(t) -  \sum_{j =1}^m Q_{\lambda_j}(\cdot - x_j - v_j t) e^{i(\f12v_j\cdot x + \gamma_j(t) )} \bigg \|_{H^1} =  o(t^{-1+}),\;\; \text{as}\;t \to +\infty.
        \end{align}
Hence Theorem \ref{thm hyperbolic} relates to  \cite{MM}, where the solution $u(t)$ is \emph{asymptotically a sum of solitary solutions of \eqref{mass-NLS} moving with constant speed along a line (through Galilean motion)}. Here, however, our solutions are still strongly interacting, whereas $v_j \neq v_k$ is a typical decoupling condition for weak interaction in the NLS, see \cite{MM}.\\[2pt]
(ii) We expect our solutions to be unstable due to $L^2$-criticality, i.e. the linear instability of $Q$, and the norm in \eqref{linear-traject} to have a sharp finite decay rate (not exponential). This expectation is consistent with the NLS literature, however let us note even for the local NLS stability questions are generally unclear because of the lack of uniqueness results. 
\end{Rem}

\noindent
Let us now sketch the argument for the proof of Corollary \ref{blow-up-dist}. The proof of Corollary \ref{blow-up-col} follows similarly.
\begin{proof}[Sketch of proof of Corollary \ref{blow-up-dist}] \
The proof follows from Theorem \ref{thm hyperbolic} in combination with the pseudo-conformal transformation
\begin{equation}
    \pc u(t,x)= \frac{1}{t^2} \bar{u} \Big( \frac{1}{t}, \frac{x}{t} \Big) e^{i\frac{|x|^2}{4t}}.
\end{equation} 
According to the choice of $v_1, v_2, \dots, v_m$, we have a hyperbolic solution $(\tilde{\alpha}, \tilde{\beta})$ of \eqref{m-body}, a function $\tilde{\gamma}(t)$  and a corresponding  global solution $\tilde{u}(t,x)$ of \eqref{H} as in Theorem \ref{thm hyperbolic}. We apply the pseudo-conformal symmetry to obtain a solution $u(t,x) := \pc \tilde{u}(t,x)$ which, by energy conservation, satisfies 
\begin{equation}\begin{split} \label{en-con}
    \|\nabla u(t) \|_{L^2_x}^2 &= \frac{1}{8\pi^2} \int |\nabla \phi_{|u(t)|^2}|^2  + O(1) = \frac{1}{8\pi^2t^2} \int |\nabla \phi_{|\tilde{u}(t^{-1})|^2}|^2 + O(1)
\end{split}\end{equation}
as $ t \to 0^+$. Hence by Theorem \ref{thm hyperbolic} and the Hardy-Littlewood-Sobolev inequality we infer $\| \nabla u \|_{L^2} \sim t^{-1}$ as $t \to 0^+$. We  have $\| \tilde{\varepsilon} \|_{H^1}= o(t^{-1+})$ as $t \to +\infty$, where 
\begin{equation}
    \tilde{\varepsilon}(t,x)= \tilde{u}(t, x)- \sum_{j =1}^m  Q_{\lambda_j}\big( x- \tilde{\alpha}_j(t) \big) e^{i( \tilde{\gamma}_j(t) +  \tilde{\beta}_j(t) \cdot x)},
\end{equation}
thus we let $\varepsilon(t,x) := \pc \tilde{\varepsilon}(t,x)$. Then $\varepsilon$ satisfies \eqref{eq expansion pseudo} with
\[
    \alpha_j(t) = t \tilde{\alpha}_j(t^{-1}),\;\; \beta_j(t) = - t^{-1}\tilde{\beta}_j(t^{-1}),\;\; \gamma_j(t) = -\tilde{\gamma}_j(t^{-1}).
\]
Note that for $p>1$ we have
\begin{equation}
    {\| \varepsilon(t,\cdot) \|}_{L^p}= t^{\frac{4}{p}-2} {\| \tilde{\varepsilon}(t^{-1}, \cdot) \|}_{L^p}.
\end{equation}
Hence we obtain $\| \varepsilon(t) \|_{L^2 \cap L^{4-}} \to 0$ as $t \to 0^+$ by Sobolev's inequality. Furthermore, for the weak convergence statement, we note given any Schwartz function $\Phi \in \mathcal{S}(\R^4)$ and using the notation $ Q_j (t,\cdot ) = Q_{\frac{ \lambda_j}{t}}( \cdot - \tilde{\alpha}_j(t^{-1}))$ we have
\begin{align*}
\int \Phi |u(t)|^2 &= \sum_{j =1}^m \int \Phi Q^2_j(t) + \sum_{j \neq k} 2\int \Phi  Q_j(t) Q_k(t) +  \sum_{j =1}^m 2\int \Phi \mathrm{Re}\big( Q_j(t) \varepsilon(t)\big)\\
&\;\;\;+ \int \Phi |\varepsilon(t)|^2.
\end{align*}
The convergence to delta functions as $ t \to 0^+$ follows since $Q_j^2(t)$ are approximation kernels (up to normalizing in $L^1$),  we have upper bounds of order $O(e^{-\frac{c}{t}|v_j - v_k|})$ for the second integrals and the remaining terms vanish by the above statement for $\varepsilon(t)$. In order to conclude the final claim we use time reflection symmetry.
\end{proof}
\;\\
\emph{Outline}.\;The proofs of Theorem \ref{thm hyperbolic} and Theorem \ref{thm parabolic} follow the strategy of \cite{KRM} and in particular the adaptation in \cite{Wu}. To begin with, in Section \ref{sec:gs-op-properties}, we collect properties of the linearized operator about the ground state, which we will rely on in the remaining sections. This is followed by Section \ref{sec:approx}, where we describe the approximation scheme which improves the decay rates of the interaction part of the error. Then in Section \ref{sec:traje} we guarantee accuracy of the approximation and solve the modulation equations, followed by a reduction of the proofs of Theorem \ref{thm hyperbolic} and Theorem \ref{thm parabolic} to a bootstrap assumption in Section \ref{sec:bootstrap}. Finally in Section \ref{sec:mod} we prove the bootstrap argument and control the necessary modulation error.\\[10pt]
\emph{Notation}: We use the usual notation $a \lesssim b$ if $ a \leq Cb$ for a constant $C> 0 $ independent of $a,b$, Similarly we use $ a \gtrsim b$ and $ a\sim b$ if and only if $ a\lesssim b, a \gtrsim b$. We also use the typical $f(x)=O(g(x))$ and $f(x) = o(g(x))$ notation and further write $ O_{d}(g(x))$ if the implicit constant depends on a parameter $d$ and grows at most exponentially as $ d \to + \infty$.

\section{Properties of the ground state operator}\label{sec:gs-op-properties}

The ground state solutions $ u(t,x) = e^{it} Q(x) $ with $Q \in H^1(\R^4)$ satisfy
\begin{align} \label{Q}
- \Delta Q + \phi_{Q^2}Q =  -Q\;\;\;\text{on}\; \R^4.
\end{align}
The following Lemma proved in \cite{KLR} asserts that \eqref{Q} has a unique positive solution.
\begin{Lemma}[ground state]
The equation \eqref{Q} has a (up to symmetry)  unique radial positive solution $Q(r) > 0$ in $ H^1(\R^4)$.
\end{Lemma}
\noindent
We need to linearize \eqref{H} at $Q$, thus for $u(t) = e^{it}(Q + \varepsilon(t))$  in \eqref{H} with $ \varepsilon = \varepsilon_1 + i \varepsilon_2$, we have
\begin{align} \label{lin-flow}
 \partial_t \begin{pmatrix}
    \varepsilon_1\\
    \varepsilon_2
\end{pmatrix} + \begin{pmatrix}
    0 &  -L_-\\
    L_+ & 0
\end{pmatrix} \begin{pmatrix}
    \varepsilon_1\\
    \varepsilon_2
\end{pmatrix} = \begin{pmatrix}
     \text{Im}(\mathcal{N}(Q, \varepsilon))\\
     -\text{Re}(\mathcal{N}(Q, \varepsilon))
\end{pmatrix},
\end{align}
where $\mathcal{N}(Q, \varepsilon)$ is nonlinear in $\varepsilon$ and the linearized operators $L_{\pm}$ are given by
\begin{align}\label{lin-op}
L_{+} &= -\Delta + 1 + \phi_{Q^2} + 2 Q \phi_{(\cdot\;  Q)},\\ \label{lin-op2}
L_{-} &= -\Delta + 1 + \phi_{Q^2}. 
\end{align}
Therefore $L_{\pm}$ are posed on  $L^2(\R^4)$ with domain $ D = H^2(\R^4)\subset L^2(\R^4)$ and we recall 
\[
\phi_{Q^2} = -  |x|^{-2}\ast Q^2,\;\;\phi_{(\cdot \; Q)} = -   |x|^{-2} \ast (\;\cdot \; Q),\;\; 
\]
highlighting that the operator $L_+$ is non-local. We now state the following conclusion.
\begin{Lemma} \label{lem straight}\phantom{a}
\begin{itemize}
    \item[(i)] $L_{\pm}$ are self-adjoint, the essential spectrum is $ \emph{spec}_{\emph{e}}(L_\pm) = [1, \infty)$, $ 0 \in \emph{spec}(L_\pm)$ and since $Q$ is the ground state $L_- \geq 0$.
    \item[(ii)] The  following are elements of the generalized root space of the linear flow in \eqref{lin-flow} 
    \begin{align*}
    i Q, \;\Lambda Q, \; \partial_{x_j}Q,\; i |x|^2Q,\; \rho(x),\;  i x_j Q,\;\; j = 1,2,3,4,
    \end{align*}
    where $\Lambda = 2 + x \cdot \nabla$ and 
    \begin{align}
    &L_-Q = L_+(\partial_{x_j}Q) = 0,\; L_+(\Lambda Q) = -2 Q,\; L_-(|x|^2 Q) = -4 \Lambda Q,\\
    & L_-(x_j Q) = -2\partial_{x_j} Q,\; L_+\rho = - |x|^2 Q.
    \end{align}
    \item[(iii)] We have 
    \begin{align*}
    &(Q,\rho) =   \f12 (\Lambda Q,|x|^2 Q) =   -\f12\| x Q\|_{L^2}^2, \;(Q, \Lambda Q) = 0,\\
    &(x Q, \nabla Q ) = - 2 \|Q\|^2_{L^2}. 
    \end{align*}
\end{itemize}
\end{Lemma}

\begin{Rem} We note that $\rho$ is the unique radial solution from \cite[Lemma 2.1]{KLR}, and furthermore the generalized root space of \eqref{lin-flow} is completely characterized by Lemma \ref{lem straight}.
\end{Rem}
\begin{Rem}  $L_+$ is non-local and there may be infinitely many eigenmodes in the gap $ (0, 1)$. However, similarly to e.g. \cite{KRM} and \cite{KLR},  this is not relevant here. The bootstrap procedure in Section \ref{sec:bootstrap} \& Section \ref{sec:mod} only relies on the coercivity in Lemma \ref{lem coercivity} below, for which we will show $L_+ \geq 0$ on $Q^{\perp} \subset L^2(\R^4)$ in the following,  and thus positive eigenvalues do not need to be modulated. 
\end{Rem}
\;\\
In this section we establish further spectral and positivity properties of \eqref{lin-op} and \eqref{lin-op2}.

\begin{Lemma}[Non-degeneracy of $\ker(L_{+})$] \label{non-deg} We have non-degeneracy of the kernel of $L_{\pm}$, i.e. if $L_{+} u = 0$, then $ u = a \cdot \nabla Q$ for some $ a \in \R^4$. Likewise if $L_{-} u = 0$, then $ u = b Q$ for some $ b \in \R$.
\end{Lemma}

\begin{proof}
    We verify the claim for the local operator $L_-$ directly. The proof for $L_+$ follows the argument of Proposition 4 in \cite[Section 7]{LenzmannNondegen} with \cite[Theorem 4]{KLR} as a starting point, i.e. $\ker(L_+) = \{0\}$ on the radial sector. However, the changes in dimension have to be accounted for.
    We begin by decomposing 
    \[L^2(\R^4) = \bigoplus_{\el \geq 0} \mathcal{H}_{(\el)},\;\;\mathcal{H}_{(\el)} = L^2(\R_+, r^3\;\d r) \otimes \mathcal{Y}_{\el},\] in terms of spherical harmonics \cite{MullerSphericalHarmonics} where $\mathcal{Y}_{\el}$ is the $\el$-th eigenspace of $\Delta_{\mathbb{S}^3}$ on $L^2(\mathbb{S}^3)$. Thus we formally split the operator $L_{+}$ into
    \[ L_{+,(\el)} f (r) = -f''(r) -\frac{3}{r} f'(r) + \frac{\el(\el+2)}{r^2} f(r) + f(r)+ V(r) f(r) + W_{(\el)}f(r) ,\]
   which acts on $f \in L^2(\R_+, r^3 \d r)$. Here
    \[ V(r) = -\phi_{Q^2}(r) = (\abs{x}^{-2} \conv \abs{Q}^2)(r) \in L^{\infty},\]
    and
    \[ W_{(\el)}f (r) = -2\omega_{3} Q(r) \int_{s > 0} r_{>}^{-2} \Gamma_{\el}\left(\frac{r_<}{r_>}\right) Q(s) f(s) s^3 \d s, \]
    where $\omega_{3}$ is the Lebesgue measure of $\sphere^{3}$, $r_> = \max \{r,s\}$, $r_< = \min \{ r,s \}$ and $\Gamma_{\el}$ are the (suitably modified) coefficients of $(1+t^2-2t\eta)^{-1}$ in the basis $\{ P_{\el} \}_{\el \geq 0}$ of Legendre Polynomials,
    \begin{equation}\label{eqn:GammaCoefficients}
        \Gamma_{\el}(t) = \frac{2\el+1}{2(\el+1)^2} \int_{-1}^{1} \frac{1}{1+t^2-2t\eta} P_{\el}(\eta) \d \eta, \quad t \in (0,1) .
    \end{equation}
    We call $\Delta_{(\el)}$ the Laplacian restricted to $\mathcal{H}_{(\el)} $ and in fact
    \[ \Delta_{(\el)} = -\partial_r^2 - \frac{3}{r} \partial_r + \frac{\el(\el+2)}{r^2} .\]
    The case $\el=0$ (i.e. the restriction to radial functions) is proven in 
    \cite[Theorem 4]{KLR} and for $\el=1$, the fact that $L_+ \partial_{x_i} Q =0$ follows from a straightforward computation. The rest of the proof hinges on a Perron-Frobenius property of $L_{+,(\el)}$, namely the lowest eigenvalue is shown to be simple, positive and has a positive eigenfunction.  Thus it rules out the possibility that the kernel of $L_+$ may include a function different from $\partial_{x_i} Q $.\\[4pt] 
    Before we establish this property for $\el \geq 2$, we note that $L_{+, (\el)} \geq 0$ for $ \el \geq 1$ because $ \mathcal{H}_{(\el)} \perp Q$ and $(L_{+})_{|Q^{\perp}} \geq 0$. For the latter we use the adaption of Weinstein's argument for local NLS, see Proposition 2.7 in \cite{Weinstein2} and in fact the beginning of the proof of Lemma \ref{lem coercivity} below.\\[3pt]
    In order to show the Perron-Frobenius property of $L_{+,(\el)}$, we will need to change (7.13) in \cite{LenzmannNondegen} by
    \begin{equation}\label{eqn:HeatKernel}
        e^{t \Delta_{(\el)}}(r,s) = t^{-3/2} \frac{\omega_3}{(4\pi)^2} \cdot \frac{2\el+1}{2(\el+1)^2} e^{-\frac{r^2+s^2}{4t}} \sqrt{\frac{\pi}{rs}} I_{\el+ 1/2}\left(\frac{rs}{2t}\right),
    \end{equation}
    where $I_{k}(z)$ is the modified spherical Bessel function of the first kind of order $k$. A proof of identity \eqref{eqn:HeatKernel} can be found in Appendix~\ref{appendix:Nondegeneracy}.
    Lastly, we may conclude as in the proof of Proposition~4 in \cite[Section~7]{LenzmannNondegen}
    by exploiting the positivity and monotonicity of the coefficients $\Gamma_{\el}$, proved in Appendix~\ref{appendix:Nondegeneracy}.
\end{proof}
\;\\
Let us now show the coercivity of the linearized operators with orthogonality to elements in the generalized root space. This is a convenient choice of orthogonality, since the elements in Lemma \ref{lem straight} naturally appear by the symmetries in the calculation of the  error $S_j^{(N)}(t,x)$ in Section \ref{sec:approx}.\\[5pt]
The following Lemma  is essentially established by the non-degeneracy of the kernel of $L_{\pm}$. 

\begin{Lemma}[Coercivity of $L_{\pm}$] \label{lem coercivity}
For all real-valued functions $ v \in H^1$ we have 
\begin{equation} \begin{aligned}
    (L_+v, v) &\ge c \| v \|_{H^1}^2- C(v,Q)^2- C(v,xQ)^2- C(v,|x|^2 Q)^2, \\
    (L_-v, v) &\ge c \| v \|_{H^1}^2- C(v,\rho)^2,
\end{aligned} \end{equation}
where  $c, C > 0$ are positive constants (independent of $v$).
\end{Lemma}

\begin{proof}
It suffices to prove that for some $c>0$, if $v \in H^1$, then 
\begin{equation} \label{eq coercivity of L+, L-} \begin{split}
    &(v,Q)= (v,xQ)= (v,|x|^2Q)= 0 \implies (L_+v,v) \ge c \Vert v \Vert_{H^1}^2, \\[3pt]
    &(v, \rho)= 0 \implies (L_-v,v) \ge c \Vert v \Vert_{H^1}^2.
\end{split} \end{equation} 
We only prove the first line of \eqref{eq coercivity of L+, L-}, as the proof of the second line is similar. 
\noindent
Recall that $Q$ is a minimizer of 
\begin{equation}
    \mathcal{E}(u)= \frac{1}{2} \int |\nabla u|^2\;\d x- \frac{1}{16\pi^2} \int |\nabla \phi_{|u|^2} |^2\;\d x, \quad \text{where } u \in H^1 \text{ and } \Vert u \Vert_{L^2}= \Vert Q \Vert_{L^2}.
\end{equation}
Thus $Q$ is a minimizer in $H^1 \backslash \{0\}$ of
\begin{equation}
    \J(u):= \frac{\Vert Q \Vert_{L^2}^2}{\Vert u \Vert_{L^2}^2} \int |\nabla u|^2\;\d x- \frac{\Vert Q \Vert_{L^2}^4}{8\pi^2\Vert u \Vert_{L^2}^4} \int |\nabla \phi_{|u|^2} |^2\;\d x.
\end{equation}
Assume $f \in H^1$ and $(f,Q)=0$. By direct computation, we have
\begin{equation} \begin{aligned}
    \frac{1}{2} \frac{\d^2}{\d \epsilon^2} \bigg|_{\epsilon=0} \J(Q+ \epsilon f) = & \int |\nabla f|^2\;\d x - \frac{\Vert \nabla Q \Vert_{L^2}^2}{\Vert Q \Vert_{L^2}^2} \int f^2\;\d x- \frac{1}{2\pi^2} \int |\nabla \phi_{Qf}|^2\;\d x  \\
    &- \frac{1}{4\pi^2} \int \nabla \phi_{Q^2} \cdot \nabla \phi_{f^2}\;\d x + \frac{\Vert \nabla \phi_{Q^2} \Vert_{L^2}^2}{4\pi^2 \Vert Q \Vert_{L^2}^2} \int f^2\;\d x.
\end{aligned} \end{equation}
Using the equation of $Q$ and integration by parts, we get
\begin{equation}
    \Vert \nabla Q \Vert_{L^2}^2- \frac{1}{4\pi^2} \Vert \nabla \phi_{Q^2} \Vert_{L^2}^2+ \Vert Q \Vert_{L^2}^2=0.
\end{equation}
Thus we obtain
\begin{equation}
    \frac{1}{2} \frac{\d^2}{\d \epsilon^2} \bigg|_{\epsilon=0} \J(Q+ \epsilon f) = (L_+f,f).
\end{equation}
By the minimality of $\J(Q)$, we deduce that 
\begin{equation}
    f \in H^1,\ (f,Q)=0 \implies (L_+f,f) \ge 0. 
\end{equation}

Therefore, if the conclusion fails to be true, then there exist $f_n \in H^1$ such that $(f_n,Q)= (f_n,xQ)= (f_n, |x|^2Q)= 0$, $\Vert f_n \Vert_{H^1}=1$ and $(L_+ f_n, f_n) \to 0$. By passing to subsequence, we may assume $f_n \rightharpoonup f_0$ in $H^1$. We have $(f_0,Q)= (f_0,xQ)= (f_0, |x|^2Q)=0$, and by the Rellich-Kondrachov theorem and the decay of $Q$, we have
\begin{equation}
    \int \phi_{Q^2} f_n^2\;\d x  \to \int \phi_{Q^2} f_0^2\;\d x \quad \text{and} \quad \int |\nabla \phi_{Qf_n}|^2\;\d x  \to \int |\nabla \phi_{Qf_0}|^2\;\d x .
\end{equation}
On the other hand, $\Vert f_0 \Vert_{H^1} \le \liminf \Vert f_n \Vert_{H^1}$. We deduce that $(L_+f_0, f_0) \le 0$.

By the non-negativity of $L_+$ on $Q^\perp$, we know that $f_0$ is a nonzero minimizer of $(L_+u,u)$, where $u \in H^1$ and $(u,Q)=0$. By computing the first variation, we get
\begin{equation}
    L_+f_0= aQ \quad \text{for some } a \in \mathbb{R}.
\end{equation}
Since $L_+(\Lambda Q)= -2Q$ and by Lemma \ref{non-deg}  $\ker L_+$ is spanned by $\nabla Q$, we now know $f_0$ is a linear combination of $\Lambda Q$ and $\nabla Q$. But using $(f_0, xQ)= (f_0, |x|^2 Q)= 0$, we must have $f_0=0$, which is a contradiction. \footnote{Note that  $(f_0,Q)=0$ is not helpful since $(\Lambda Q, Q)=0$ does hold. This explains why we have $(v,|x|^2 Q)$ on the right-hand side in the lemma.}
\end{proof}

\begin{Lemma}[Inversion of $L_{\pm}$] \label{Inversion-of-L}
Let $f$ be a real-valued with exponential decay as $|x| \to +\infty$.
\begin{enumerate}
    \item If $\langle f, \nabla Q \rangle = 0$, there exists a real-valued exponentially decaying solution to $ L_+u = f$.
    \item If $\langle f, Q \rangle = 0$,  there exists a real-valued exponentially decaying  solution  to $ L_-u = f$. Further if $ f$ is radial, then $u$ can be chosen radial.
\end{enumerate}
\end{Lemma}

\begin{proof}
The proof of $(1), (2) $ is a straightforward adaption of the argument in \cite{KRM}, with only slight differences. Writing
\[
L_+ = \big[ I + (\phi_{Q^2} + Q 
\phi_{(\cdot Q)}) \circ (- \Delta + I)^{-1}\big] (- \Delta + I),
\]
we know $ L_+ (- \Delta + I)^{-1} - I $ is a compact operator  mapping  $\mathcal{H} := \
\text{span} (\nabla Q)^{\perp}$ into itself and has the adjoint 
\[
\Pi \circ\big[ (- \Delta + I)^{-1} L_+ -I \big],
\]
where $ \Pi $ is the orthogonal projection onto $ \mathcal{H}$. A simple bootstrap argument shows that any solution  to 
\[
v + \Pi \big[ (- \Delta + I)^{-1} L_+v -v \big] = 0
\]
is smooth and then obviously  
$$  L_+v = (- \Delta + I) [ a \cdot \nabla Q]$$
for some $ a \in \R^4$. Integrating over the product with  $a \cdot \nabla Q $ and using Lemma \ref{non-deg}, we have $v = 0$. This implies that the Fredholm operator
\[
L_+ (- \Delta + I)^{-1}  =  I + (L_+ (- \Delta + I)^{-1} - I)
\]
is invertible. Now if $ L_+ u =f $ as in (1), we show exponential decay following \cite{KRM}, which is standard argument due to Agmon (see the reference in \cite{KRM}). Precisely, if we use the kernel $ (- \Delta + I)^{-1}(x,y) = G(|x-y|)$, we can write 
\begin{align} \label{int-inv-L}
u(x) =   \int_{\R^4 } G(|x-y|) f(y) \;\d y - \int_{\R^4} G(|x-y|) \big[ \phi_{Q^2} u + 2 Q \phi_{u Q}\big](y)\;\d y,
\end{align}
where $ G(r) $ with $ r > 0$ is a real-valued smooth function such that for some constants $c_1, c_2 > 0$ 
\[
G(r) = \frac{c_1}{r^2} e^{-r},\;\;r \gg 1,\;\;\;\;
G(r) =     \frac{c_2}{r^2} + O(1), \;\;r \ll 1.
\]
\begin{Rem} In fact it is well known that with some $\beta \in \R$ we have  
$$ |x-y| \cdot G(|x-y|) = \beta  H_1^{(1)}(i |x-y|),$$
where $ H_1^{(1)}$ is the first  Hankel function of  order one.
\end{Rem}
Bootstrapping \eqref{int-inv-L}, we infer $u$ to be in $L^{\infty}(\R^4)$ by Sobolev's embedding. Further, assuming $ |f(y) | \lesssim e^{- c |y|}$ for some $ 0 < c < 1 $, we chose any $ \delta < \frac12 c$ and $ L \gg1 $ large enough such that 
\[
 (\sup_{|y| > L} \phi_{Q^2}(y)) \| e^{\delta |x|} G(|x|)  \|_{L^1(\R^4)} < \frac12. 
\]
Splitting the second integral in \eqref{int-inv-L} via $\{ |y| > L \}$ and $ \{ |y| \leq L \}$, this directly implies a pointwise estimate for any $M > 0$
\[
\min\{ M, e^{\delta |x|}\} |u(x)| \leq \f12 \sup_{y \in \R^4}\big(\min\{M, e^{i |y|}|u(y)|\big) + C_0,
\]
where $ C_0$ depends on $u$ but not on $ M > 0$. Taking the supremum on the left and letting $M \to \infty$ concludes the argument.
\end{proof}
\section{Construction of approximate solutions}\label{sec:approx}
\noindent
We now introduce an approximation scheme by modified arguments in \cite{Wu} (see also \cite{KRM}).\\[4pt]
Let $\alpha= (\alpha_1, \cdots, \alpha_m)$ and similarly for $(\beta, \lambda, \mu, \delta, \gamma)$, which are (possibly) time dependent parameters. We further denote 
\begin{equation} \label{eq notation} \begin{split}
    &P= (\alpha, \beta, \lambda, \mu, \delta), \quad g=(P,\gamma), \quad g_j = (P_j, \gamma_j) =  (\alpha_j, \beta_j, \lambda_j, \mu_j, \delta_j, \gamma_j),\\[3pt]
    &\alpha_{jk}= \alpha_j- \alpha_k, \quad \beta_{jk}= \beta_j- \beta_k, \quad a= \min_{j \neq k} |\alpha_{jk}|.
\end{split} \end{equation}
For a function $v : \R \times \R^4 \to \C$ we then modulate via the path $g_j =(P_j, \gamma_j)$ as follows
\begin{equation} \label{eq action}
    g_j v(t,x)  := \lambda_j^2 v \big(t, \lambda_j (x-\alpha_j) \big) e^{i\gamma_j+ i\beta_j \cdot x+ i \mu_j |x|^2}.
\end{equation}
We note here the parameters $(\alpha_j, \beta_j, \lambda_j, \gamma_j)$ correspond to the symmetry \eqref{inv} as in \cite{Wu}, however the additional parameters $(\mu_j, \delta_j)$ are due to the pseudo-conformal symmetry (c.f.  \eqref{pseudo}) of \eqref{H}. In particular they guarantee extended orthogonality conditions in Lemma \ref{lem orthogonality} of Section \ref{sec:bootstrap}  and hence the coercive control in Lemma \ref{lem coercivity}.\\[4pt]
The parameter $\delta_j$ corresponds to $ L_-\rho = -|x|^2Q $ in Lemma \ref{lem straight} and is used (contrary to \cite{KRM, Wu}) to model the approximation on $Q(x) + \delta_j \rho(x)$ via 
\begin{align}
   V_j =  Q(x) + \delta_j \rho(x) + W_j.
\end{align}
where $W_j$ are correction terms which are homogeneous in the modulation parameters. This modified ansatz guarantees the existence of our correction terms in Proposition \ref{prop construction of approximate bubbles} below, see the second line of \eqref{this-othog}. We hence stress, contrary to \cite{Wu} and due to the orthogonality $(Q, \Lambda Q) = 0$ in our case, we rely on this modification to guarantee the existence of the following approximation.\\[4pt] 
For $u = g_j v$ we now have (let us omit the subscript $j = 1, \dots, m$) the following calculations
\begin{align} \label{eq time deriv}
    e^{-i\gamma- i\beta \cdot x- i \mu |x|^2} i\partial_t u &= 2i \lambda \dot{\lambda} v+ i\lambda^2 \partial_t v+ i\lambda^2 \big( \dot{\lambda} (x-\alpha)- \lambda \dot{\alpha} \big) \nabla v \\
    &\quad \;- \lambda^2 \big( \dot{\gamma}+ \dot{\beta} \cdot (x-\alpha)+ \dot{\beta} \cdot \alpha+ \dot{\mu} |x|^2 \big) v, \\[3pt]
    e^{-i\gamma- i\beta \cdot x- i \mu |x|^2} \nabla u &= \lambda^2 \big( \lambda \nabla v+ i(\beta+ 2\mu x) v \big), \\[3pt]
    e^{-i\gamma- i\beta \cdot x- i \mu |x|^2} \Delta u &= \lambda^4 \Delta v+ 2i \lambda^3 (\beta+ 2\mu x) \cdot \nabla v+ 8i\lambda^2 \mu v\\
    & \quad \;- \lambda^2 \big( |\beta|^2+ 4\mu \beta \cdot x+ 4\mu^2|x|^2 \big) v,
\end{align}
as well as for the nonlinearity
\begin{align} \label{eq nonlinearity}
    e^{-i\gamma- i\beta \cdot x- i \mu |x|^2} \phi_{|u|^2} u = \lambda^4 \phi_{|v|^2} v.
    \end{align}
Now if $v_j$  denotes any function depending on the parameters $(\alpha_j, \beta_j, \lambda_j, \mu_j, \delta_j)$ and the spatial variable, we set
\begin{equation}
    u(t,x) := \sum_{j=1}^m u_j(t,x) := \sum_{j=1}^m g_jv_j(t,x)
\end{equation}
and  by \eqref{eq time deriv}, \eqref{eq nonlinearity} observe the following expression
\begin{equation} \label{eq expression}
    i\partial_t u+ \Delta u- \phi_{|u|^2}u= \sum_{j=1}^m E_j(t, y_j) e^{i\gamma+ i\beta \cdot x+ i \mu |x|^2}- \sum_{k \neq j} \phi_{\re (u_k \overline{u_j})} u,
\end{equation}
where $y_j= \lambda_j (x-\alpha_j)$. We note the rightmost term on the right-hand side of \eqref{eq expression} will be small by separation of the solitons and we identify $E_j(t, y_j)$ to be the main error. In particular, setting $\Lambda v_j= 2v_j+ y_j \cdot \nabla v_j$, we have 
\begin{align} \label{calc-E_j}
    E_j(t, y_j) &=  \lambda_j^4 \big( \Delta v_j- v_j- \delta_j |y_j|^2 v_j \big) \\[2pt]
    &\quad - \big( \dot{\mu}_j+ 4\mu_j^2- \lambda_j^4 \delta_j \big) |y_j|^2 v_j- \lambda_j \big( \dot{\beta}_j+ 4\mu_j \beta_j+ (\dot{\mu}_j+ 4\mu_j^2) \alpha_j \big) y_j v_j \\[2pt]
    &\quad - \lambda_j^2 \big( \dot{\gamma}_j+ (\dot{\beta}_j+ 4\mu_j \beta_j) \cdot \alpha_j+ (\dot{\mu}_j+ 4\mu_j^2) |\alpha_j|^2+ |\beta_j|^2- \lambda_j^2 \big) v_j \\[2pt]
    &\quad +i\lambda_j \big( \dot{\lambda}_j+ 4\lambda_j \mu_j \big) \Lambda v_j- i\lambda_j^3 \big( \dot{\alpha}_j- 2\beta_j- 4\mu_j \alpha_j \big) \nabla v_j \\[2pt]
    &\quad+ i \lambda_j^2 \sum_{k=1}^m \Big( \frac{\partial v_j}{\partial \alpha_k} \dot{\alpha}_k+ \frac{\partial v_j}{\partial \beta_k} \dot{\beta}_k+ \frac{\partial v_j}{\partial \lambda_k} \dot{\lambda}_k+ \frac{\partial v_j}{\partial \mu_k} \dot{\mu}_k+ \frac{\partial v_j}{\partial \delta_k} \dot{\delta}_k \Big) \\[2pt]
    &\quad - \lambda_j^4 \bigg[ \phi_{|v_j|^2}+ \sum_{k \neq j} \Big( \frac{\lambda_k}{\lambda_j} \Big)^2 \phi_{|v_k|^2} \Big( \frac{\lambda_k y_j}{\lambda_j}+ \lambda_k \alpha_{jk} \Big) \bigg] v_j. 
\end{align} 
\;\\
For the error term in the last line we write
\begin{equation} \label{this-term}
    \left( \frac{\lambda_k}{\lambda_j} \right)^2 \phi_{|v_k|^2} \Big( P(t), \frac{\lambda_k y_j}{\lambda_j}+ \lambda_k \alpha_{jk} \Big)= -\frac{1}{\lambda_j^2} \int_{\mathbb{R}^4} \frac{|v_k(P(t),\xi)|^2}{|\alpha_{jk}+ \lambda_j^{-1} y_j- \lambda_k^{-1} \xi|^2} \d \xi,
\end{equation}
and we consider the Taylor expansion
\begin{equation} \label{taylor}
    \frac{1}{|\alpha-\zeta|^2}= \sum_{n=1}^N F_n(\alpha,\zeta)+ O \left( \frac{|\zeta|^N}{|\alpha|^{N+2}} \right) \quad \text{as } \zeta \to 0,
\end{equation}
where $F_n(\alpha,\zeta)$ is homogeneous of degree $-n-1$ in $\alpha$ and of degree $n-1$ in $\zeta$. Let us calculate the first few terms explicitly:
\begin{equation} \begin{aligned}
    F_1(\alpha, \zeta) &= |\alpha|^{-2},\;\;F_2(\alpha, \zeta) = |\alpha|^{-4} 2(\zeta \cdot \alpha), \\[3pt]
    F_3(\alpha, \zeta) &= |\alpha|^{-6} \big( 4(\zeta \cdot \alpha)^2- |\zeta|^2 |\alpha|^2 \big), \\[3pt]
    F_4(\alpha, \zeta) &= |\alpha|^{-8} \big( 8(\zeta \cdot \alpha)^3- 4(\zeta \cdot \alpha) |\zeta|^2 |\alpha|^2 \big), \\[3pt]
    F_5(\alpha, \zeta) &= |\alpha|^{-10} \big( 16 (\zeta \cdot \alpha)^4- 12 (\zeta \cdot \alpha)^2 |\zeta|^2 |\alpha|^2+ |\zeta|^4 |\alpha|^4 \big).
\end{aligned} \end{equation} 
We now replace the convolution  kernel in \eqref{this-term} by \eqref{taylor} and define the following approximation
\begin{equation} \begin{aligned} \label{approx-nonlinearity}
    \phi_{|v_k|^2}^{(N)}(t,y_j) &:= \sum_{n=1}^N \psi_{|v_k|^2}^{(n)}(t,y_j) := \sum_{n=1}^N -\frac{1}{\lambda_j^2} \int_{\mathbb{R}^4} |v_k(t,\xi)|^2 F_n \big( \alpha_{jk}, \lambda_k^{-1} \xi- \lambda_j^{-1} y_j\big) \d \xi.
\end{aligned} \end{equation}
Note that if $v_k$ is homogeneous in the parameters of $P(t)$, then so is $\psi_{|v_k|^2}^{(n)}(t,y_j)$. Further, for \eqref{eq expression} - \eqref{calc-E_j}, we make the ansatz  $v_j(P(t), y_j) = V_j^{(N)}(P(t), y_j) $ with
\begin{align} \label{eq V_j^N}
     V_j^{(N)}(P(t), y_j) := Q(y_j)+ \delta_j \rho(y_j)+ W_j^{(N)}(\alpha_j, \beta_j, \lambda_j, \mu_j, \delta_j, y_j),
\end{align}
where the function $\rho(y_j)$ is as in Lemma \ref{lem straight}.  Moreover  $W_j^{(N)}$ has the form
\[
W_j^{(N)}(P(t), y_j) = \sum_{n =1}^N T_{j}^{(n)}(P(t), y_j)
\]
where $T_{j}^{(n)}$ decays exponentially in $y_j$ and enjoys a certain homogeneity property of degree $n$ in the modulation parameters. See Definition  \ref{def admissible} and Proposition \ref{prop construction of approximate bubbles} below. Let us now define
\begin{equation} \label{eq approximate solution}
    R_g^{(N)}(t,x):= \sum_{j=1}^m R_{j,g}^{(N)}(t,x):= \sum_{j=1}^m g_j V_j^{(N)}(P(t),x) ,
\end{equation}
for which we will omit the subscript $g$ of $R^{(N)}$ if convenient. As above, we write
\begin{equation} \label{eq equation of approximate solution} \begin{aligned}
    &i\partial_t R^{(N)}+ \Delta R^{(N)}- \phi_{|R^{(N)}|^2} R^{(N)} \\[3pt]
    =\;&\sum_{j=1}^m E_j^{(N)}(t,y_j) e^{i\gamma_j+ i\beta_j \cdot x +i\mu_j|x|^2}- \sum_{k \neq j} \phi_{\mathrm{Re} (R_k^{(N)} \overline{R_j^{(N)}})} R^{(N)} \\[3pt]
    &+ \sum_{j=1}^m \lambda_j^4 \sum_{k \neq j} \bigg[ \phi_{\left| V_k^{(N)} \right|^2}^{(N)}- \Big( \frac{\lambda_k}{\lambda_j} \Big)^2 \phi_{ \left| V_k^{(N)} \right|^2} \Big( \frac{\lambda_k y_j}{\lambda_j}+ \lambda_k \alpha_{jk} \Big) \bigg] V_j^{(N)} e^{i\gamma_j+ i\beta_j \cdot x+i\mu_j|x|^2},
\end{aligned} \end{equation} 
where the latter line is the error from the approximation \eqref{approx-nonlinearity}. 
Considering \eqref{calc-E_j}, we now split 
$
E_j^{(N)}= \tilde{E}_j^{(N)}+ S_j^{(N)}
$
into the \emph{interaction part} and the \emph{modulation part} of the error. In principle the former error should consist of the terms
\begin{align} \label{interact-error}
    &\lambda_j^4 \Big( \Delta V_j^{(N)}- V_j^{(N)}- \phi_{\left| V_j^{(N)} \right|^2} V_j^{(N)}- \delta_j |y_j|^2 V_j^{(N)} \Big)- \lambda_j^4 \sum_{k \neq j} \phi_{\left| V_k^{(N)} \right|^2}^{(N)} V_j^{(N)} \\
    &+ i\lambda_j^2 \sum_{k=1}^m \bigg( \frac{\partial W_j^{(N)}}{\partial \alpha_k} \dot{\alpha}_k+ \frac{\partial W_j^{(N)}}{\partial \beta_k} \dot{\beta}_k + \frac{\partial W_j^{(N)}}{\partial \lambda_k} \dot{\lambda}_k+ \frac{\partial W_j^{(N)}}{\partial \mu_k} \dot{\mu}_k + \frac{\partial W_j^{(N)}}{\partial \delta_k} \dot{\delta}_k \bigg) 
\end{align}
However, we now show how a more precise dependence of $W_j$ on $P(t)$ improves the decay of these terms.  This is achieved by removal of slowly decaying terms otherwise appearing in the interaction part
\[
\lambda_j^4 \sum_{k \neq j} \phi_{\left| V_k^{(N)} \right|^2}^{(N)} V_j^{(N)}.
\]
In order to outline the idea, let us note three observations.\\[5pt]
(1)\; In the second line of \eqref{interact-error} we expect to replace $\dot{\alpha}_k(t), \dot{\beta}_k(t), \dot{\lambda}_k(t), \dot{\mu}_k(t), \dot{\delta}_k(t)$ with the modulation equations in the second to fourth line of \eqref{calc-E_j} since we aim to control these terms in the modulation analysis (see Section \ref{sec:traje}).\\[5pt]
(2)\; Likewise, considering the same modulation part in \eqref{calc-E_j}, we \emph{modify each of these equations by adding correction terms} for $(\delta_j, \beta_j, \lambda_j)$ of the form
    \begin{align} \label{corrections-first}
    i\lambda_j^2 D_j^{(N)} \cdot \rho,\;\;\lambda_j B_j^{(N)} \cdot y_j V_j^{(N)},\;\; i\lambda_j M_j^{(N)} \cdot \Lambda V_j^{(N)},
    \end{align}
    where $D_j^{(N)}, B_j^{(N)}, M_j^{(N)} $ are sums of functions homogeneous in the respective parameters.\\[5pt]
    (3)\; We linearize the first line of \eqref{interact-error} at $Q$ and construct $T_j^{(n)}$ by inversion of $L_{\pm}$. Here $D_j^{(N)}$ and $B_j^{(N)}$  are essential in order to apply Lemma \ref{Inversion-of-L} (see also Lemma \ref{Inversion-of-L-admissible}) and $M_j^{(N)}$ will be essential in the next Section \ref{sec:traje}.\\[6pt]
In particular we make the following ansatz 
\begin{align} \nonumber
    \tilde{E}_j^{(N)}(t,y_j) =& \;\lambda_j^4 \Big( \Delta V_j^{(N)}- V_j^{(N)}- \phi_{\left| V_j^{(N)} \right|^2} V_j^{(N)}- \delta_j |y_j|^2 V_j^{(N)} \Big)- \lambda_j^4 \sum_{k \neq j} \phi_{\left| V_k^{(N)} \right|^2}^{(N)} V_j^{(N)} \\[3pt] \label{eq definition of E_j tilde} 
    &\;+ i\lambda_j^2 D_j^{(N)} \rho- \lambda_j B_j^{(N)} \cdot y_j V_j^{(N)}+ i\lambda_j M_j^{(N)} \Lambda V_j^{(N)} \\[3pt] \nonumber
    &\;+ i\lambda_j^2 \sum_{k=1}^m \bigg( \frac{\partial W_j^{(N)}}{\partial \alpha_k} \cdot \big( 2\beta_k+ 4\mu_k \alpha_k \big)+ \frac{\partial W_j^{(N)}}{\partial \beta_k} \cdot \big( B_k^{(N)}- 4\mu_k \beta_k- \lambda_k^4 \delta_k \alpha_k \big) \\[3pt] \nonumber
    &\qquad \qquad \quad \ + \frac{\partial W_j^{(N)}}{\partial \lambda_k} \big( M_k^{(N)}- 4\mu_k \lambda_k \big)+ \frac{\partial W_j^{(N)}}{\partial \mu_k} \big( \lambda_k^4 \delta_k- 4\mu_k^2 \big)+ \frac{\partial W_j^{(N)}}{\partial \delta_k} D_k^{(N)} \bigg).
\end{align}
\;\\[-.2cm]
By definition, we then have for the modulation error (note the modulation equations are now modified by $B_j^{(N)}, D_j^{(N)}, M_j^{(N)}$)
\begin{align}\label{eq definition of S_j^N} 
    &S_j^{(N)}(t,x) \\ \nonumber
    &= - i\lambda_j^3 \big( \dot{\alpha}_j- 2\beta_j- 4\mu_j \alpha_j \big) \nabla V_j^{(N)}+ i\lambda_j \big( \dot{\lambda}_j+ 4\lambda_j \mu_j- M_j^{(N)} \big) \Lambda V_j^{(N)}+ i\lambda_j^2 \big( \dot{\delta}_j- D_j^{(N)} \big) \rho \\[3pt] \nonumber
    &- \big( \dot{\mu}_j+ 4\mu_j^2- \lambda_j^4 \delta_j \big) |y_j|^2 V_j^{(N)}- \lambda_j \big( \dot{\beta}_j+ 4\mu_j \beta_j+ (\dot{\mu}_j+ 4\mu_j^2) \alpha_j- B_j^{(N)} \big) \cdot y_j V_j^{(N)} \\[3pt]  \nonumber
    &- \lambda_j^2 \big( \dot{\gamma}_j+ (\dot{\beta}_j+ 4\mu_j \beta_j) \cdot \alpha_j+ (\dot{\mu}_j+ 4\mu_j^2) |\alpha_j|^2+ |\beta_j|^2- \lambda_j^2 \big)  V_j^{(N)} \\[3pt]  \nonumber
    &+ i\lambda_j^2 \sum_{k=1}^m \Bigg[ \frac{\partial W_j^{(N)}}{\partial \alpha_k} \cdot \left( \dot{\alpha}_k- 2\beta_k- 4\mu_k \alpha_k \right)+ \frac{\partial W_j^{(N)}}{\partial \beta_k} \cdot \left( \dot{\beta}_k+ 4\mu_k \beta_k+ \lambda_k^4 \delta_k \alpha_k- B_k^{(N)} \right) \\[3pt]  \nonumber
    &\quad \ + \frac{\partial W_j^{(N)}}{\partial \lambda_k} \big( \dot{\lambda}_k+ 4\mu_k \lambda_k- M_k^{(N)} \big)+ \frac{\partial W_j^{(N)}}{\partial \mu_k} \big( \dot{\mu}_k+ 4\mu_k^2- \lambda_k^4 \delta_k \big)+ \frac{\partial W_j^{(N)}}{\partial \delta_k} \big( \dot{\delta}_k- D_k^{(N)} \big) \Bigg].
\end{align}
We emphasize the  approximation, which will be defined via the $T_j^{(n)}$ functions, requires to compare degrees of (parameter) homogeneity for the source terms in $ \tilde{E}_j^{(N)}$. Furthermore, $M_j^{(N)}$ are free to choose (see Section \ref{sec:traje}) and $B_j^{(N)}, D_j^{(N)}$ are determined by orthogonality in the proof of Proposition \ref{prop construction of approximate bubbles}.\\[6pt]
The following notion will be useful to guarantee the accuracy of the approximation.
\begin{Def}[\textbf{Admissible functions}] \label{def admissible} 
Recalling \eqref{eq notation}, let $\Omega$ denote the space of non-collision positions:
\begin{equation} \begin{aligned}
    \Omega:= \Big\{ P= (\alpha, \beta, \lambda, \mu, \delta) \in \mathbb{R}^{4m} \times \mathbb{R}^{4m} \times \mathbb{R}_+^m \times \mathbb{R}^m \times \mathbb{R}^m \ \big| \ \alpha_j \neq \alpha_k,\ \forall j \neq k \Big\}.
\end{aligned} \end{equation}
\;\\[-.3cm]
(1) Let $n \in \mathbb{N}$. Define $S_n$ to be the set of all functions $\sigma: \Omega \to \mathbb{R}$ which are finite sums of 
\begin{equation}
    c\prod_{j \neq k} |\alpha_j- \alpha_k|^{-q_{jk}} \prod_{j=1}^m \alpha_j^{p_j} \beta_j^{k_j} \lambda_j^{l_j} \mu_j^{m_j} \delta_j^{n_j},
\end{equation}
where $c \in \mathbb{R}$, $q_{jk} \in \mathbb{N}$, $p_j \in \mathbb{N}^4$, $k_j \in \mathbb{N}^4$, $l_j \in \mathbb{Z}$, $m_j \in \mathbb{N}$, $n_j \in \mathbb{N}$ and
\begin{equation} \label{eq degree counting} \begin{gathered}
    \sum_{j \neq k} q_{jk}+ \sum_{j=1}^m \big( 2m_j+ 3n_j- |p_j| \big)= n \quad \textbf{in the hyperbolic case}; \\
    \sum_{j \neq k} q_{jk}+ \sum_{j=1}^m \big( |k_j| + 4m_j+ 6n_j- |p_j| \big)= n \quad \textbf{in the parabolic case}.
\end{gathered} \end{equation}
\;\\
(2) We say a function $u: \Omega \times \mathbb{R}^4 \to \mathbb{C}$ is \textbf{admissible} if $u$ is a finite sum of 
\begin{equation}
    z\sigma(\alpha, \beta, \lambda, \mu, \delta) \tau(x),
\end{equation}
where $z \in \mathbb{C}$, $\sigma \in S_n$ for some $n \in \mathbb{N}$ and $\tau \in C^\infty$ satisfies
\begin{equation}
    \big| \nabla^k \tau(x) \big| \le e^{-c_k|x|}, \qquad \forall k \ge 0,\ x \in \mathbb{R}^4.
\end{equation}
If $n$ is the same for all addends, then we say $u$ is admissible of degree $n$. Otherwise, taking $n$ as the minimal one among all addends, we say $u$ is admissible of degree $\ge n$.
\end{Def}
\;\\
\emph{The degree} in Definition \ref{def admissible}  is such that we count powers of  $a^{-1}$, i.e. functions in the class $S_n$ are essentially of order $O(a^{-n})$ given the expected decay of the parameters $P(t)$, c.f. the modulation part of \eqref{calc-E_j} and the error $S_j^{(N)}(t, y_j)$. This decay will be obtained below in Sections \ref{sec:traje} - Section \ref{sec:bootstrap}.\\[5pt]
Hence the counting formula \eqref{eq degree counting} is derived from the expected contributions to the decay rates for each parameter. For example, in the case of  hyperbolic trajectories, we have $ \alpha_j(t) \sim v_j t + O(1)$. Thus $ a(t) = t$ and we will further prove as $t \to + \infty$
\[
|\beta_j| \lesssim 1,\;\;\lambda_j \sim 1,\;\;|\mu_j| \lesssim a^{-2},\;\; |\delta_j| \lesssim a^{-3}.
\]
Therefore the contributions in the first line of \eqref{eq degree counting} are as follows:
The power $ - | p_j| $ for $\alpha_j$, no contribution for $\beta_j$ and $ \lambda_j$, the power $2m_j$ for $\mu_j$, as well as the power $3n_j$ for $\delta_j$.\\[8pt]
In particular we have the following decay estimate for admissible functions.
\begin{Lemma} \label{lem properties of admissible functions-2}
Let $n \in \mathbb{N}$, $u$ be admissible of degree $n$ and 
\begin{equation} \label{eq boundedness of g} \begin{split}
    &|\alpha| \lesssim a, \ |\beta| \lesssim 1, \ \lambda \sim 1, \ |\mu| \lesssim a^{-2}, \ |\delta| \lesssim a^{-3} \quad \textbf{hyperbolic}, \\[3pt]
    &|\alpha| \lesssim a, \ |\beta| \lesssim a^{-1}, \ \lambda \sim 1, \ |\mu| \lesssim a^{-4}, \ |\delta| \lesssim a^{-6} \quad \textbf{parabolic},
\end{split} \end{equation}
in the case of either hyperbolic or parabolic trajectories. Then there are $ c_k > 0$ with 
\[
|\nabla^k u(P, x) | \lesssim a^{-n} e^{-c_k|x|},\;\; k \geq 0.
\]
\end{Lemma}
\;\\[-.2cm]
We also state useful algebraic properties for admissible functions. 
\begin{Lemma} \label{lem properties of admissible functions}
Let $n,m \in \mathbb{N}$ and $u,v$ be admissible of degree $n,m$, respectively. Then
\begin{enumerate} [label=(\arabic*)]
    \item $uv$ is admissible of degree $n+m$;
    \item $\phi_u v$ is admissible of degree $n+m$;
    \item $\forall k \ge 1$, $\psi_u^{(k)} v$ is admissible of degree $k+1+n+m$;
    \item $\forall N \ge 1$, $\phi_u^{(N)} v$ is admissible of degree $\ge 2+n+m$;
\end{enumerate}
\end{Lemma}
\;\\[-.2cm]
The following Lemma states that the inversion of $L_{\pm}$ in Lemma \ref{Inversion-of-L} is consistent with the admissibility in Definition \ref{def admissible}.
\begin{Lemma}[Inversion of $L_{\pm}$ for admissible functions] \label{Inversion-of-L-admissible}
Let $f$ be a real-valued admissible function of degree $n \in \N$. If $\langle f, \nabla Q \rangle = 0$, then the solution to $ L_+u = f$ in Lemma \ref{Inversion-of-L} is admissible of degree $n$. Moreover if $\langle f, Q \rangle = 0$, then the solution  to $ L_-u = f$ in Lemma \ref{Inversion-of-L} is admissible of degree $n$. 
\end{Lemma}
\;\\[-.2cm]
The next proposition implies the construction of the approximation. 
\begin{Prop} \label{prop construction of approximate bubbles}
Given $m_j^{(n)} \in S_{n+1}$. For $n \ge 1$ and $1 \le j \le m$, there exist $d_j^{(n)}, b_j^{(n)} \in S_{n+1}$ and $T_j^{(n)}$ that is admissible of degree $n+1$ such that: for any $N \ge 1$, if we set
\begin{gather*}
    M_j^{(N)}(P)= \sum_{n=1}^N m_j^{(n)}(P), \\
    W_j^{(N)}(P,y_j)= \sum_{n=1}^N T_j^{(n)}(P,y_j),  \\
    D_j^{(N)}(P)= \sum_{n=1}^N d_j^{(n)}(P), \quad B_j^{(N)}(P)= \sum_{n=1}^N b_j^{(n)}(P),
\end{gather*}
then $\tilde{E}_j^{(N)}$ defined by \eqref{eq definition of E_j tilde} is admissible of degree $\ge N+2$.
\end{Prop}
\;\\[-.3cm]
The proof is very similar to the proof of \cite[Proposition~2.4]{Wu}, however we will give some details.

\begin{proof} 
Let $\hat{E}_j^{(N)}$ denote the terms in $\tilde{E}_j^{(N)}$ of degree $N+2$. Further let us write $T_j^{(n)}= X_j^{(n)}+ i Y_j^{(n)}$ and recall $V_j^{(N)}$ is defined via \eqref{eq V_j^N}. We prove by induction over $N \geq 1$.\\[4pt]
Let us also make the assumption $m_j^{(1)}=0$ which simplifies the initial calculation and is consistent with the procedure in Section \ref{sec:traje}. However, we can remove it along the same line of arguments with slightly more effort.\\[4pt]
For the first step $N=1$ take $d_j^{(1)}= b_j^{(1)}=0$ and let $T_j^{(1)}$ be admissible of degree $2$, real-valued, as well as independent of $\beta, \mu, \delta$.\\[4pt]
Now we have
\begin{align} \nonumber
    \tilde{E}_j^{(1)}(t,y_j) =& \;\lambda_j^4 \Big( \Delta V_j^{(1)}- V_j^{(1)}- \phi_{\left| V_j^{(1)} \right|^2} V_j^{(1)}- \delta_j |y_j|^2 V_j^{(1)} \Big)- \lambda_j^4 \sum_{k \neq j} \phi_{\left| V_k^{(1)} \right|^2}^{(1)} V_j^{(1)}\\[3pt]
    &\;+ i\lambda_j^2 \sum_{k=1}^m \bigg( \frac{\partial T_j^{(1)}}{\partial \alpha_k} \cdot \big( 2\beta_k+ 4\mu_k \alpha_k \big)- \frac{\partial T_j^{(1)}}{\partial \lambda_k} 4\mu_k \lambda_k \bigg) \\ \nonumber
    = &\; \lambda_j^4 \Big( \Delta T_j^{(1)}- T_j^{(1)}- \phi_{Q^2} T_j^{(1)}- 2\phi_{Q T_j^{(1)}} Q- \sum_{k \neq j} \psi_{Q^2,k}^{(1)} Q \Big)+ error \\ \nonumber
    = &\; \lambda_j^4 \Big( -L_+ T_j^{(1)}+ \sum_{k \neq j} \psi_{Q^2,k}^{(1)} Q \Big)+ error,
\end{align}
where $error$ consists of higher order terms from the interaction part and the above second line, so that it is admissible of degree greater than or equal to $3$. Since $L_+(\Lambda Q)= -2Q$, we may in particular take 
\begin{equation} 
    T_j^{(1)}= - 2\sum_{k \neq j} \psi_{Q^2,k}^{(1)} \Lambda Q
\end{equation}
to cancel the degree $2$ terms. This proves the conclusion when $N=1$.\\[4pt]
Now we go from $N$ to $N+1$ and consider each term in $\tilde{E}_j^{(N+1)}- \tilde{E}_j^{(N)}$. Using that $T_j^{(n)}$ is admissible of degree $n+1$, we obtain $\tilde{E}_j^{(N+1)}- \tilde{E}_j^{(N)}$ is of degree greater than or equal to $N+2$. Further the terms of degree $N+2$ are given by
\begin{align}
    &- \lambda_j^4 \big( L_+ X_j^{(N+1)}+ iL_- Y_j^{(N+1)} \big)- \lambda_j^4 \sum_{k \neq j} \psi_{Q^2,k}^{(N+1)} Q \\
    &+ i\lambda_j^2 d_j^{(N+1)} \rho- \lambda_j b_j^{(N+1)} y_j Q+ i\lambda_j m_j^{(N+1)} \Lambda Q.
\end{align}
For details on these observations we refer to the proof of \cite[Proposition~2.4]{Wu}, and in particular the expansion of $\tilde{E}_j^{(N+1)}- \tilde{E}_j^{(N)}$.\\[5pt]
We want to eliminate all terms in $\tilde{E}_j^{(N+1)}$ of degree $N+2$, therefore it suffices to require
\begin{equation} 
    \left\{ \begin{aligned}
        &L_+ X_j^{(N+1)}= -\lambda_j^{-3} b_j^{(N+1)} \cdot y_j Q- \sum_{k \neq j} \psi_{Q^2,k}^{(N+1)} Q+ \lambda_j^{-4} \mathrm{Re} \ \hat{E}_j^{(N)}, \\
        &L_- Y_j^{(N+1)}= \lambda_j^{-2} d_j^{(N+1)} \rho+ \lambda_j^{-3} m_j^{(N+1)} \Lambda Q+ \lambda_j^{-4} \mathrm{Im} \ \hat{E}_j^{(N)},
    \end{aligned} \right.
\end{equation}
By Lemma \ref{Inversion-of-L} and Lemma \ref{Inversion-of-L-admissible}, such $X_j^{(N+1)}$ and $Y_j^{(N+1)}$ exist if
\begin{equation} \label{this-othog}
    \left\{ \begin{aligned}
        & \Big( -\lambda_j^{-3} b_j^{(N+1)} \cdot y_j Q- \sum_{k \neq j} \psi_{Q^2,k}^{(N+1)} Q+ \lambda_j^{-4} \mathrm{Re} \ \hat{E}_j^{(N)}, \nabla Q \Big)=0, \\ 
        &\Big( \lambda_j^{-2} d_j^{(N+1)} \rho+ \lambda_j^{-3} m_j^{(N+1)} \Lambda Q+ \lambda_j^{-4} \mathrm{Im} \ \hat{E}_j^{(N)}, Q \Big)=0.
    \end{aligned} \right.
\end{equation}
In particular $b_j^{(N+1)}$ and $d_j^{(N+1)}$ exist because $(y_j Q, \nabla Q)$ is invertible and $(\rho, Q) \neq 0$.
\end{proof}

\begin{Rem}
We can extend $\tilde{E}_j^{(N)}$ to $N=0$ by setting $V_j^{(0)}= Q+ \delta_j \rho$ naturally. Then the base case $N=1$ is similar to the induction process. Nevertheless, the precise formula of $T_j^{(1)}$ will be useful in Section \ref{sec:traje}.
\end{Rem}

\;\\
Finally Proposition \ref{prop construction of approximate bubbles}  implies proper control over the interaction part of the error term.
\begin{Prop} \label{prop accuracy of approximate solution}
Let $V_j^{(N)}$ be as in Proposition \ref{prop construction of approximate bubbles}. For $R_g^{(N)}$ defined by \eqref{eq approximate solution}, let 
\begin{equation} \label{eq definition of Psi}
    \Psi^{(N)}= i\partial_t R_g^{(N)}+ \Delta R_g^{(N)}- \phi_{|R_g^{(N)}|^2} R_g^{(N)}- \sum_{j=1}^m S_j^{(N)} e^{i\gamma_j+ i\beta_j \cdot x+ i\mu|x|^2}.     
\end{equation}
\noindent
If \eqref{eq boundedness of g} is satisfied, then there exist constants $ c, C > 0$  such that
\begin{equation} \label{eq estimate of Psi}
    |\Psi^{(N)}(t,x)| \le \frac{C}{a^{N+2}(t)}  \max \limits_{j = 1, \dots,m} e^{-c|x-\alpha_j(t)|},\;\; x \in \R^4.
\end{equation} 
Recall $ \displaystyle a = \min_{j < k} |\alpha_{j k}(t)| \sim  t$ in the hyperbolic and $ a \sim t^{\f12}$ in the parabolic case.
\end{Prop}

\begin{proof}
For simplicity, we omit the superscript $N$ and the subscript $g$.

By \eqref{eq equation of approximate solution}, we have the following expression
\begin{equation} \begin{aligned}
    \Psi &= \sum_{j=1}^m \tilde{E}_j(t,y_j) e^{-i\gamma_j+ i\beta_j \cdot x}+ \sum_{j \neq k} \phi_{\mathrm{Re} (R_j \overline{R_k})} R \\
    &\ +\sum_{j=1}^m \lambda_j^4 \sum_{k \neq j} \bigg[ \phi_{\left| V_k \right|^2}^{(N)}- \Big( \frac{\lambda_k}{\lambda_j} \Big)^2 \phi_{ \left| V_k \right|^2} \Big( \frac{\lambda_k y_j}{\lambda_j}+ \lambda_k \alpha_{jk} \Big) \bigg] V_j e^{-i\gamma_j+ i\beta_j \cdot x}.
\end{aligned} \end{equation}
Proposition \ref{prop construction of approximate bubbles} states that $\tilde{E}_j^{(N)}$ is admissible of degree greater than or equal to $ N+2$, so the first term is controlled by the right-hand side of \eqref{eq estimate of Psi} using \eqref{eq boundedness of g} and the definition of the degree. The second term has an exponential decay in $a$ by the exponential decay of $Q$.\\[5pt]
For the last term, using the Taylor expansion and exponential decay of $V_j$, we have
\begin{equation} \label{eq accuaracy of dipole expansion, preparation}
    \left| \phi_{\left| V_k \right|^2}^{(N)}- \left( \frac{\lambda_k}{\lambda_j} \right)^2 \phi_{|V_k|^2} \Big( \frac{\lambda_k y_j}{\lambda_j}+ \lambda_k \alpha_{jk} \Big) \right| \le \left\{ \begin{aligned}
        &C (1+|y_j|)^N,\ |y_j| \ge \frac{\lambda_j |\alpha_{jk}|}{3}, \\
        &\frac{C(1+|y_j|)^N}{|\alpha_{jk}|^{N+2}},\ |y_j| \le \frac{\lambda_j |\alpha_{jk}|}{3}.
    \end{aligned} \right.
\end{equation}
A detailed proof can be found in \cite[Proposition~2.2]{Wu}. This yields
\begin{equation} \label{eq accuaracy of dipole expansion}
    \left| \phi_{\left| V_k \right|^2}^{(N)}- \left( \frac{\lambda_k}{\lambda_j} \right)^2 \phi_{|V_k|^2} \left( \frac{\lambda_k y_j}{\lambda_j} + \lambda_k\alpha_{jk} \right) \right| \le \frac{C(1+|y_j|)^{2N}}{a^{N+2}},
\end{equation}
which provides control over the last term.
\end{proof}

\section{The trajectories}\label{sec:traje}

Next we study the condition under which we have $S_j^{(N)}=0$. Since $\gamma$ can be solved from the other parameters, we essentially need to solve
\begin{equation} \label{eq parameters}
    \left\{ \begin{aligned}
        &\dot{\alpha}_j- 2\beta_j- 4\mu_j \alpha_j=0, \\
        &\dot{\beta}_j+ 4\mu_j \beta_j+ \lambda_j^4 \delta_j \alpha_j- B_j^{(N)}=0, \\
        &\dot{\lambda}_j+ 4\lambda_j \mu_j- M_j^{(N)}=0, \\
        &\dot{\mu}_j+ 4\mu_j^2- \lambda_j^4 \delta_j=0, \\
        &\dot{\delta}_j- D_j^{(N)}=0.
    \end{aligned} \right.
\end{equation}
In fact, we wish to view it as a perturbation of the $m$-body problem \eqref{m-body}, so we need precise information of the first few terms of $D_j^{(N)}$ and $B_j^{(N)}$.\\
\;\\
\textbf{Hyperbolic case:}\\
We take $m_j^{(n)}=0$ for $n \neq 2$, and \begin{equation}
    m_j^{(2)}= -\frac{2\| Q \|_{L^2}^2}{\lambda_j} \sum_{k \neq j} \frac{\alpha_{jk} \cdot \beta_{jk}}{|\alpha_{jk}|^4}.
\end{equation}
We will show we can take
\begin{equation} \label{eq first few terms hyperbolic}
    b_j^{(1)}=0, \quad d_j^{(1)}= d_j^{(2)}= d_j^{(3)}= 0.
\end{equation}

Since $\hat{E}_j^{(0)}=0$ and $\big( \sum \limits_{k \neq j} \psi_{Q^2,k}^{(1)} Q, \nabla Q \big)=0$, we take 
\begin{equation}
    b_j^{(1)}= d_j^{(1)}=0, \quad T_j^{(1)}= -\frac{\| Q \|_{L^2}^2}{2\lambda_j^2} \Big( \sum_{k \neq j} \frac{1}{|\alpha_{jk}|^2} \Big) \Lambda Q.
\end{equation}

Then we have
\begin{equation} \begin{aligned}
    \hat{E}_j^{(1)} &= -\lambda_j^4 \delta_j (L_+ \rho+ |y_j|^2 Q)+ i\lambda_j^2 \sum_{k=1}^m \frac{\partial T_j^{(1)}}{\partial \alpha_k} \cdot 2\beta_k \\
    &= i\lambda_j^2 \sum_{k=1}^m \frac{\partial T_j^{(1)}}{\partial \alpha_k} \cdot 2\beta_k.
\end{aligned} \end{equation}
Since $\re \ \hat{E}_j^{(1)}=0$, $\imm \big( \hat{E}_j^{(1)}, Q \big)=0$, and given our choice of $m_j^{(2)}$, we take
\begin{equation}
    b_j^{(2)}= -\| Q \|_{L^2}^2 \sum_{k \neq j} \frac{\alpha_{jk}}{|\alpha_{jk}|^4}, \quad d_j^{(2)}=0, \quad T_j^{(2)}=0. 
\end{equation}

Then we have
\begin{equation} 
    \imm \ \hat{E}_j^{(2)}= \lambda_j^2 \sum_{k=1}^m \Big( \frac{\partial T_j^{(1)}}{\partial \alpha_k} \cdot 4\mu_k \alpha_k- \frac{\partial T_j^{(1)}}{\partial \lambda_k} 4\mu_k \lambda_k \Big).
\end{equation}
Since $(\Lambda Q, Q)=0$, we take $d_j^{(3)}=0$.

\begin{Prop} \label{prop hyperbolic trajectory}
Let $(\alpha^\infty, \beta^\infty)$ be a hyperbolic solution of \eqref{m-body}. Assume $\lambda^\infty \in \R_+^m$ and $\mu^\infty= \delta^\infty=0$. Then there exists a solution $P^{(N)}$ of \eqref{eq parameters} such that 
\begin{equation}
    \Big| P^{(N)}(t)- P^\infty(t) \Big| \to 0 \;\; \text{as} \;\; t \to +\infty.
\end{equation}
\end{Prop}

\begin{proof}
Let $\epsilon=\frac{1}{10}$. Define $X= \Big\{ P \in C \big( [T_0,+\infty), \Omega \big) \ \big| \ {\| P-P^\infty \|}_1 \le 1 \Big\}$, where
\begin{equation}
    {\| P \|}_1:= \sum_{j=1}^m \sup_{t \ge T_0} \Big( t^{1-4\epsilon} |\alpha_j(t)|+ t^{2-3\epsilon} |\beta_j(t)|+ t^{2-3\epsilon} |\lambda_j(t)|+ t^{3-2\epsilon} |\mu_j(t)|+ t^{4-\epsilon} |\delta_j(t)| \Big).
\end{equation}
Consider the map $\Gamma$ defined on $X$ by
\begin{equation} \label{eq Gamma} \begin{aligned}
    \Gamma \alpha_j(t) &= \alpha_j^\infty(t)+ \int_t^\infty \big( 2\beta_j^\infty- 2\beta_j- 4\mu_j \alpha_j \big) \d \tau, \\
    \Gamma \beta_j(t)&= \lim_{t \to \infty} \beta_j^\infty(t)+ \int_t^\infty \Big( 4\mu_j \beta_j+ \lambda_j^4 \delta_j \alpha_j- B_j^{(N)}(P) \Big) \d \tau, \\
    \Gamma \lambda_j(t)&= \lambda_j^\infty+ \int_t^\infty \Big( 4\lambda_j \mu_j- M_j^{(N)}(P) \Big) \d \tau, \\
    \Gamma \mu_j(t)&= \int_t^\infty \big( 4\mu_j^2- \lambda_j^4 \delta_j \big) \d \tau, \\
    \Gamma \delta_j(t)&= -\int_t^\infty D_j^{(N)}(P) \d \tau.
\end{aligned} \end{equation}
We prove that $\Gamma$ maps $X$ to $X$ and is a contraction. 

Recall that $M_j^{(N)}= m_j^{(2)}$ and we have \eqref{eq first few terms hyperbolic}. Take $P \in X$. Then 
\begin{equation} \begin{aligned}
    |\Gamma \alpha_j(t)- \alpha_j^\infty(t)| &\le \int_t^\infty \big( \tau^{3\epsilon-2}+ \tau^{2\epsilon-3} \tau \big) \d \tau \lesssim t^{3\epsilon-1}, \\
    |\Gamma \beta_j(t)- \beta_j^\infty(t)| &\lesssim \int_t^\infty \Big( \tau^{2\epsilon-3}+ \tau^{\epsilon-4} \tau+ \big| b_j^{(2)}(P^\infty) \big|+ \sum_{n=2}^N |b_j^{(n)}(P)| \Big) \d \tau \\
    &\lesssim t^{2\epsilon-2}+ \int_t^\infty \tau^{-3} \d \tau \lesssim t^{2\epsilon-2}, \\
    |\Gamma \lambda_j(t)- \lambda_j^\infty| &\lesssim \int_t^\infty \big( \tau^{2\epsilon-3}+ \tau^{-3} \big) \d \tau \lesssim t^{2\epsilon-2}, \\
    |\Gamma \mu_j(t)| &\lesssim \int_t^\infty \big( \tau^{4\epsilon-6}+ \tau^{\epsilon-4} \big) \d \tau \lesssim t^{\epsilon-3}, \\
    |\Gamma \delta_j(t)| &\le \int_t^\infty \sum_{n=4}^N |d_j^{(n)}(P)| \ \d \tau \lesssim \int_t^\infty \tau^{-5} \d \tau \lesssim t^{-4}.
\end{aligned} \end{equation}
We deduce $\Gamma P \in X$ when $T_0$ is large enough.

Next assume $P, P' \in X$. Using the fundamental theorem of calculus, we have
\begin{equation} \label{this-bound}
    \big| b_j^{(n)}(P)- b_j^{(n)}(P') \big|+ \big| d_j^{(n)}(P)- d_j^{(n)}(P') \big| \le t^{-n-1} {\| P- P' \|}_1
\end{equation}
Then by the same calculation as above, we have
\begin{gather}
    |\Gamma \alpha_j(t)- \Gamma \alpha_j'(t)| \lesssim t^{3\epsilon-1} {\| P- P' \|}_1, \quad |\Gamma \beta_j(t)- \Gamma \beta_j'(t)| \lesssim t^{2\epsilon-2} {\| P- P' \|}_1, \\
    |\Gamma \lambda_j(t)- \Gamma \lambda_j'(t)| \lesssim t^{2\epsilon-2} {\| P- P' \|}_1, \quad |\Gamma \mu_j(t)- \Gamma \mu_j'(t)| \lesssim t^{\epsilon-3} {\| P- P' \|}_1, \\
    |\Gamma \delta_j(t)- \Gamma \delta_j'(t)| \lesssim t^{-4} {\| P- P' \|}_1.
\end{gather}
We deduce that $\Gamma$ is a contraction. Then the fixed point of $\Gamma$ in $X$ is the desired solution.
\end{proof}

\;\\
\textbf{Parabolic case:}\\
We take $m_j^{(n)}=0$ for $n \neq 3$ and
\begin{equation}
    m_j^{(3)}= -\frac{2\| Q \|_{L^2}^2}{\lambda_j} \sum_{k \neq j} \frac{\alpha_{jk} \cdot \beta_{jk}}{|\alpha_{jk}|^4}.
    \end{equation}
We will show we can take
\begin{equation} \label{eq first few terms b parabolic}
    b_j^{(1)}= b_j^{(3)}= b_j^{(4)}=0, \quad b_j^{(2)}= -\| Q \|_{L^2}^2 \sum_{k \neq j} \frac{\alpha_{jk}}{|\alpha_{jk}|^4} 
\end{equation}
and 
\begin{equation} \label{eq first few terms d parabolic}
    d_j^{(1)}= d_j^{(2)}= d_j^{(3)}= d_j^{(4)}= d_j^{(5)}= d_j^{(6)}= d_j^{(7)}= 0. 
\end{equation}

Let us write 
\begin{equation} \label{eq f}
    f= \frac{1}{2} \sum_{k \neq j} \psi_{Q^2,k}^{(1)}= -\frac{\| Q \|_{L^2}^2}{2\lambda_j^2} \Big( \sum_{k \neq j} \frac{1}{|\alpha_{jk}|^2} \Big), \quad h= -\lambda_j \sum_{k=1}^m \frac{\partial f}{\partial \alpha_k} \cdot 2\beta_k. 
\end{equation}
Note that $f$ only depends on $\alpha$ and $\lambda_j$. Also $m_j^{(3)}= h$.

Since $\hat{E}_j^{(0)}=0$ and $\big( \sum \limits_{k \neq j} \psi_{Q^2,k}^{(1)} Q, \nabla Q \big)=0$, we take
\begin{equation}
    b_j^{(1)}= d_j^{(1)}=0, \quad T_j^{(1)}= f \Lambda Q.
\end{equation}

Then we have $\hat{E}_j^{(1)}=0$, so we take
\begin{equation}
    b_j^{(2)}= -\| Q \|_{L^2}^2 \sum_{k \neq j} \frac{\alpha_{jk}}{|\alpha_{jk}|^4}, \quad d_j^{(2)}=0, \quad T_j^{(2)}=0. 
\end{equation}

Then we have
\begin{equation} 
    \imm \ \hat{E}_j^{(2)}= \lambda_j^2 \sum_{k=1}^m \frac{\partial T_j^{(1)}}{\partial \alpha_k} \cdot 2\beta_k= -\lambda_j h \Lambda Q 
\end{equation}
and 
\begin{equation} \begin{aligned}
    \re \ \hat{E}_j^{(2)} &= -\lambda_j^4 \big( 2\phi_{QT_j^{(1)}} T_j^{(1)}+ \phi_{|T_j^{(1)}|^2} Q \big)- \lambda_j^4 \sum_{k \neq j} \Big( \psi_{Q^2,k}^{(1)} T_j^{(1)}+ 2\psi_{Q T_k^{(1)}}^{(1)} Q \Big) \\
    &= -\lambda_j^4 \big( 2\phi_{QT_j^{(1)}} T_j^{(1)}+ \phi_{|T_j^{(1)}|^2} Q \big)- \lambda_j^4 \sum_{k \neq j} \psi_{Q^2,k}^{(1)} T_j^{(1)} \\
    &= -\lambda_j^4 f^2 \Big( 2\phi_{Q \Lambda Q} \Lambda Q+ \phi_{(\Lambda Q)^2} Q+ 2\Lambda Q \Big).
\end{aligned} \end{equation}
By the choice of $m_j^{(3)}$ and parity, we take  
\begin{equation}
    b_j^{(3)}= d_j^{(3)}=0, \quad T_j^{(3)}= T_j^{(3)}(\alpha, \lambda) \text{ real valued}.
\end{equation}
Moreover, since 
\begin{equation} \label{eq L+ Lambda2Q}
    -\frac{1}{2} L_+ (\Lambda \Lambda Q- 2\Lambda Q)= 2\phi_{Q \Lambda Q} \Lambda Q+ \phi_{(\Lambda Q)^2} Q+ 2\Lambda Q,
\end{equation} 
we have
\begin{equation}
    T_j^{(3)} = \frac{f^2}{2} (\Lambda \Lambda Q- 2\Lambda Q)- L_+^{-1} \sum_{k \neq j} \psi_{Q^2,k}^{(3)} Q. 
\end{equation}
\;\\
We make the following observation.

\begin{Lemma} \label{lem F_n harmonic}\phantom{a}
\begin{enumerate}
    \item $\Delta_x F_n(\alpha,x)=0$.
    \item If $u$ is continuous and radial then $F_n * u= (\int u) F_n$.
\end{enumerate}
\end{Lemma}
\begin{proof}
(1) Since $\sum_{n=1}^\infty F_n(\alpha,x)$ is a power series in $x$ with radius of convergence $r \ge \frac{|\alpha|}{3}$, we can differentiate term by term within $B_r= \{ |x|< r \}$. The sum of the power series in $B_r$ equals $\frac{1}{|\alpha-x|^2}$, which is harmonic. Thus we have $\sum \Delta_x F_n=0$ in $B_r$. Since the $\Delta F_n$ are polynomials of different degrees, we have $\Delta_x F_n=0$ in $B_r$ for any $n$. This means $\Delta_x F_n=0$ for any $x$ because $F_n$ is a polynomial.\\[4pt]
(2) We have 
\begin{align*}
    F_n * u(x) &= \int F_n(x-y) u(y) \d y= \int_0^{+\infty} \int_{\partial B_1} F_n(x-r\omega) u(r) r^3 \d \omega \d r \\
    &= F_n(x) \int_0^{+\infty} \int_{\partial B_1} u(r) r^3 \d \omega \d r= \left(\int u\right) F_n(x)
\end{align*}
by the mean value property of harmonic functions.
\end{proof}

For two functions $u$ and $v$, we write $u \equiv v$ if $u-v$ is an even function with exponential decay, and is $L^2$-orthogonal to all radial functions with exponential decay. By the proof of Proposition \ref{prop construction of approximate bubbles}, if we only want to determine the value of $b_j^{(n)}$ and $d_j^{(n)}$, then we can do calculations up to $\equiv$. 

By Lemma \ref{lem F_n harmonic}, we have $\psi_{Q^2,k}^{(3)} Q \equiv 0$, and thus $L_+^{-1} \psi_{Q^2,k}^{(3)} Q \equiv 0$. This means
\begin{equation}
    T_j^{(3)} \equiv \frac{f^2}{2} (\Lambda \Lambda Q- 2\Lambda Q).
\end{equation}

Then we have
\begin{equation} \begin{aligned}
    \hat{E}_j^{(3)}= &- \lambda_j^4 \sum_{k \neq j} \Big( \psi_{Q^2,k}^{(2)} T_j^{(1)}+ 2\psi_{Q T_k^{(1)}}^{(2)} Q \Big)- \lambda_j b_j^{(2)} y_j T_j^{(1)}= 0.
\end{aligned} \end{equation}
By Lemma \ref{lem F_n harmonic}, we also have $\nabla \psi_{Q^2,k}^{(4)} \equiv 0$. Thus we take 
\begin{equation} \begin{gathered}
    b_j^{(4)}= d_j^{(4)}=0, \quad T_j^{(4)}= T_j^{(4)}(\alpha, \lambda) \text{ real valued}
\end{gathered} \end{equation}
with
\begin{equation}
    L_+ T_j^{(4)}= - \sum_{k \neq j} \psi_{Q^2,k}^{(4)} Q.
\end{equation}

Then we have
\begin{equation} \begin{aligned}
    \imm \ \hat{E}_j^{(4)} &= \lambda_j m_j^{(3)} \Lambda T_j^{(1)}+ \lambda_j^2 \sum_{k=1}^m \bigg( \frac{\partial T_j^{(3)}}{\partial \alpha_k} \cdot 2\beta_k+ \frac{\partial T_j^{(1)}}{\partial \alpha_k} \cdot 4\mu_k \alpha_k+ \frac{\partial T_j^{(1)}}{\partial \lambda_k} \big( m_k^{(3)}- 4\mu_k \lambda_k \big) \bigg) \\
    &\equiv \lambda_j fh \Lambda \Lambda Q- \lambda_j fh (\Lambda \Lambda Q- 2\Lambda Q)+ \lambda_j^2 \sum_{k=1}^m \Big( \frac{\partial f}{\partial \alpha_k} \cdot 4\mu_k \alpha_k \Big) \Lambda Q- 2\lambda_j f \Lambda Q(h- 4\mu_j \lambda_j) \\
    &\equiv 4\lambda_j^2 \bigg( 2\mu_j f+ \sum_{k=1}^m \Big( \frac{\partial f}{\partial \alpha_k} \cdot \mu_k \alpha_k \Big) \bigg) \Lambda Q
\end{aligned} \end{equation}
and
\begin{equation} \begin{aligned}
    \re \ \hat{E}_j^{(4)} &=- \lambda_j^4 \Big( 2\phi_{QT_j^{(3)}} T_j^{(1)}+ 2\phi_{T_j^{(1)} T_j^{(3)}}Q+ 2\phi_{Q T_j^{(1)}} T_j^{(3)}+ \phi_{|T_j^{(1)}|^2} T_j^{(1)} \Big) \\
    &\quad -\lambda_j^4 \sum_{k \neq j} \Big( \psi_{Q^2,k}^{(3)} T_j^{(1)}+ 2\psi_{QT_k^{(1)}}^{(3)} Q+ \psi_{Q^2,k}^{(1)} T_j^{(3)} \\
    &\qquad \qquad \qquad+ 2\psi_{QT_k^{(3)}}^{(1)} Q+ \psi_{|T_k^{(1)}|^2}^{(1)} Q+ 2\psi_{QT_k^{(1)}}^{(1)} T_j^{(1)} \Big) \\
    &\equiv -\lambda_j^4 f^3 \Big( \phi_{Q( \Lambda \Lambda Q- 2\Lambda Q)} \Lambda Q+ \phi_{\Lambda Q( \Lambda \Lambda Q- 2\Lambda Q)} Q \\
    &\qquad \qquad \ \ + \phi_{Q \Lambda Q} (\Lambda \Lambda Q- 2\Lambda Q)+ \phi_{(\Lambda Q)^2} \Lambda Q+ (\Lambda \Lambda Q- 2\Lambda Q) \Big).
\end{aligned} \end{equation}
Thus we take
\begin{equation}
    b_j^{(5)}= d_j^{(5)}=0, \quad X_j^{(5)}= X_j^{(5)}(\alpha, \lambda), \quad Y_j^{(5)}= Y_j^{(5)}(\alpha, \beta, \lambda) .
\end{equation}
Moreover, since 
\begin{align}
    -L_+(\Lambda \Lambda \Lambda Q- 6\Lambda \Lambda Q+ 8\Lambda Q) &= \phi_{Q( \Lambda \Lambda Q- 2\Lambda Q)} \Lambda Q+ \phi_{\Lambda Q( \Lambda \Lambda Q- 2\Lambda Q)} Q \\
    &\quad + \phi_{Q \Lambda Q} (\Lambda \Lambda Q- 2\Lambda Q)+ \phi_{(\Lambda Q)^2} \Lambda Q+ (\Lambda \Lambda Q- 2\Lambda Q),
\end{align} 
we have  
\begin{equation}
    X_j^{(5)} \equiv \frac{f^3}{6} (\Lambda \Lambda \Lambda Q- 6\Lambda \Lambda Q+ 8\Lambda Q), \quad Y_j^{(5)} \equiv -\lambda_j^{-2} \bigg( 2\mu_j f+ \sum_{k=1}^m \Big( \frac{\partial f}{\partial \alpha_k} \cdot \mu_k \alpha_k \Big) \bigg) |y_j|^2 Q.
\end{equation}

Then we have 
\begin{equation}
    \imm \ \hat{E}_j^{(5)}= \lambda_j^2 \sum_{k=1}^m \frac{\partial T_j^{(4)}}{\partial \alpha_k} \cdot 2\beta_k.
\end{equation} 
Since $T_j^{(4)}$ is odd, $\imm \big( \hat{E}_j^{(5)}, Q \big)=0$, so we take $d_j^{(6)}=0$. 

Then we have
\begin{equation} \begin{aligned}
    \imm \ \hat{E}_j^{(6)} &= -2\lambda_j^4 \phi_{Q T_j^{(1)}} Y_j^{(5)}- \lambda_j^4 \sum_{k \neq j} \psi_{Q^2,k}^{(1)} Y_j^{(5)}+ \lambda_j m_j^{(3)} \Lambda T_j^{(3)} \\
    &\quad +\lambda_j^2 \sum_{k=1}^m \bigg( \frac{\partial X_j^{(5)}}{\partial \alpha_k} \cdot 2\beta_k+ \frac{\partial T_j^{(3)}}{\partial \alpha_k} \cdot 4\mu_k \alpha_k+ \frac{\partial T_j^{(3)}}{\partial \lambda_k} \big( m_k^{(3)}- 4\mu_k \lambda_k \big) \bigg) \\
    &= -2\lambda_j^4 f \phi_{Q \Lambda Q} Y_j^{(5)}- 2\lambda_j^4 f Y_j^{(5)}+ \lambda_j m_j^{(3)} \Lambda T_j^{(3)} \\
    &\quad + \lambda_j^2 \sum_{k=1}^m \frac{\partial X_j^{(5)}}{\partial \alpha_k} \cdot 2\beta_k+ \lambda_j^2 \sum_{k=1}^m \frac{\partial T_j^{(3)}}{\partial \alpha_k} \cdot 4\mu_k \alpha_k+ \lambda_j^2 \frac{\partial T_j^{(3)}}{\partial \lambda_j} \big( m_j^{(3)}- 4\mu_j \lambda_j \big) \\
    &\equiv 2\lambda_j^2 f \bigg( 2\mu_j f+ \sum_{k=1}^m \Big( \frac{\partial f}{\partial \alpha_k} \cdot \mu_k \alpha_k \Big) \bigg) (\phi_{Q \Lambda Q}+ 1) |y_j|^2 Q+ \frac{1}{2} \lambda_j f^2 h \Lambda (\Lambda \Lambda Q- 2\Lambda Q) \\
    &\quad - \frac{1}{2} \lambda_j f^2 h (\Lambda \Lambda \Lambda Q- 6\Lambda \Lambda Q+ 8\Lambda Q)- 2\lambda_j f^2 h (\Lambda \Lambda Q- 2\Lambda Q) \\
    &\quad + 4\lambda_j^2 f \bigg( 2\mu_j f+ \sum_{k=1}^m \Big( \frac{\partial f}{\partial \alpha_k} \cdot \mu_k \alpha_k \Big) \bigg) (\Lambda \Lambda Q- 2\Lambda Q) \\
    &\equiv 2\lambda_j^2 f \bigg( 2\mu_j f+ \sum_{k=1}^m \Big( \frac{\partial f}{\partial \alpha_k} \cdot \mu_k \alpha_k \Big) \bigg) \big( \phi_{Q\Lambda Q} |y_j|^2 Q+ |y_j|^2 Q +2(\Lambda \Lambda Q- 2\Lambda Q) \big).
\end{aligned} \end{equation}
Since $L_-(\Lambda Q)= -2\phi_{Q \Lambda Q}Q- 2Q$ and $L_-(|y_j|^2Q)= -4\Lambda Q$, we have
\begin{align}
    \Big( \phi_{Q\Lambda Q} |y_j|^2 Q+ |y_j|^2 Q +2(\Lambda \Lambda Q- 2\Lambda Q), Q \Big)= \big( \Lambda Q, -\frac{1}{2} L_-(|y_j|^2 Q) \big)- 2(\Lambda Q, \Lambda Q)=0
\end{align}
by integration by parts. Thus we take $d_j^{(7)}=0$.

\begin{Prop} \label{prop parabolic trajectory}
Let $(\alpha^\infty, \beta^\infty)$ be a parabolic solution of \eqref{m-body}. Assume $\lambda^\infty \in \R_+^m$ and $\mu^\infty= \delta^\infty=0$. Then there exists a solution $P^{(N)}$ of \eqref{eq parameters} such that 
\begin{equation}
    \Big| P^{(N)}(t)- P^\infty(t) \Big| \to 0 \text{ as } t \to +\infty.
\end{equation}
\end{Prop}

\begin{proof}
Take $\epsilon= \frac{1}{100}$. Define $Y= \Big\{ P \in C \big( [T_0,+\infty), \Omega \big) \ \big| \ {\| P-P^\infty \|}_2 \le 1 \Big\}$, where
\begin{equation}
    {\| P \|}_2:= \sum_{j=1}^m \sup_{t \ge T_0} \Big( t^{1- 3\epsilon} |\alpha_j(t)|+ t^{2- 3\epsilon} |\beta_j(t)|+ t^{2-3\epsilon} |\lambda_j(t)|+ t^{\frac{5}{2}-2\epsilon} |\mu_j(t)|+ t^{\frac{7}{2}-\epsilon} |\delta_j(t)| \Big).
\end{equation}
We want to find a solution of \eqref{eq parameters} in $Y$. 

Recall that $M_j^{(N)}= m_j^{(3)}$ and we have \eqref{eq first few terms b parabolic} and \eqref{eq first few terms d parabolic}. In contrast to the hyperbolic case, we need to estimate $b_j^{(2)}(P)- b_j^{(2)}(P^\infty)$ more delicately. By the Taylor formula, we have
\begin{equation}
    b_j^{(2)}(P)- b_j^{(2)}(P^\infty)= -\| Q \|_{L^2}^2 \sum_{k \neq j} \left( \frac{\alpha_{jk}- \alpha_{jk}^\infty}{|\alpha_{jk}^\infty|^4}- \frac{4\alpha_{jk}^\infty \cdot (\alpha_{jk}- \alpha_{jk}^\infty)}{|\alpha_{jk}^\infty|^6} \alpha_{jk}^\infty \right)+ O(t^{-\frac{7}{2}+ 6\epsilon}).
\end{equation}
Thus there exists $A_j \in \R^{4 \times 4m}$ such that
\begin{equation}
    b_j^{(2)}(P)- b_j^{(2)}(P^\infty)= \frac{A_j (\alpha- \alpha^\infty)}{t^2}+ O(t^{-3-2\epsilon}),
\end{equation}
where $\alpha$ is understood as a column vector.

Set $A=(A_1^T, \cdots, A_m^T)^T \in \R^{4m \times 4m}$. Then \eqref{eq parameters} can be rewritten as

\begin{equation}
    \left\{ \begin{aligned}
        &\dot{\alpha}- \dot{\alpha}^\infty= 2(\beta- \beta^\infty)+ O(t^{2\epsilon- 2}), \\
        &\dot{\beta}- \dot{\beta}^\infty= \frac{A (\alpha- \alpha^\infty)}{t^2} + O(t^{2\epsilon- 3}), \\
        &\dot{\lambda}= O(t^{2\epsilon-2}), \ \dot{\mu}= O(t^{\epsilon-\frac{7}{2}}), \ \dot{\delta}= O(t^{-\frac{9}{2}}). 
    \end{aligned} \right.
\end{equation}
Then we need to use a modified version of Lemma 5.1 in \cite{Wu}.

\begin{Lemma} \label{lem ODE}
Let $0 < \delta< \kappa$, $n \in \mathbb{N}$, $A \in \mathbb{R}^{n \times n}$, $F_1 \in C^1 \big( \mathbb{R}_+ \times \mathbb{R}^n; \mathbb{R}^n \big)$ and $F_2 \in C^1 \big( \mathbb{R}_+ \times \mathbb{R}^{n+n}; \mathbb{R}^n \big)$. Assume 
\begin{equation} \label{eq uniform C^1 estimate of F_1} \begin{gathered}
    \sup_{|x| \le t^{-\delta}} \Big( |F_1(t,x)|+ t|\partial_t F_1(t,x)|+ |\nabla_x F_1(t,x)| \Big) \le t^{-1-\kappa}, \quad \forall t>0.
\end{gathered} \end{equation}
\begin{equation} \label{eq uniform C^1 estimate of F_2} 
    \sup_{|x|+t|y| \le t^{-\delta}} \Big( |F_2(t,x,y)|+ |\nabla_x F_2(t,x,y)|+ |\nabla_y F_2(t,x,y)| \Big) \le t^{-2-\kappa}, \quad \forall t>0.
\end{equation}
Then there exists $T>0$ and $x,y \in C^1 \big( [T,+\infty), \mathbb{R}^n \big)$ such that
\begin{equation}
    \left\{ \begin{aligned}
        \dot{x}(t) &= y(t)+ F_1(t,x(t)), \\
        \dot{y}(t) &= \frac{A x(t)}{t^2}+ F_2(t, x(t), y(t)) 
    \end{aligned} \right.
    \quad \text{and} \quad 
    \left\{ \begin{aligned}
        |x(t)| &\le t^{-\delta}, \\
        |y(t)| &\le t^{-1-\delta},
    \end{aligned} \right. \quad \forall t \ge T. 
\end{equation}
\end{Lemma}

\begin{proof}
We may assume $x,y$ and $F_1, F_2$ take complex values. If the complex counterpart is proved, then the lemma is proved by taking real parts. 

First assume $A$ is diagonalizable over $\C$ and $-\frac{1}{4}$ is not an eigenvalue. Let $v_1, \cdots, v_n$ be independent eigenvectors of $A$ and $c_1, \cdots, c_n$ be corresponding eigenvalues. Let $a_j, b_j$ be the two roots of $\lambda^2- \lambda= c_j$. Since $c_j \neq -\f14$, we have $a_j \neq b_j$. 

Let us consider the space 
\begin{equation}
    B= \Big\{ (x,y) \in C^1  \big( [T,+\infty), \C^{n+n} \big) \ \big| \ |x(t)|+ t|y(t)| \le t^{-\delta} \Big\}
\end{equation}
with norm ${\| \cdot \|}_3$ defined by
\begin{equation}
    {\| (x,y) \|}_3 := \sup_{t \ge T} \Big( t^\delta |x(t)|+ t^{1+\delta} |y(t)| \Big).
\end{equation}

For $(x,y) \in B$ we write
\begin{equation}
    \partial_t F_1(t,x)+ \nabla_x F_1(t,x) \cdot y+ F_2(t, x, y)= \sum_{j=1}^n f_j(t,x,y) v_j.
\end{equation}
Note here that $f_j$ also satisfies \eqref{eq uniform C^1 estimate of F_2}. We now define $\Gamma (x,y)$ by
\begin{equation} \begin{aligned}
    (\Gamma x)(t) &= \sum_{j=1}^n G_{a_j,b_j}f_j(t,x,y) v_j, \\
    (\Gamma y)(t) &= \frac{\d}{\d t} \sum_{j=1}^n G_{a_j,b_j}f_j(t,x,y) v_j- F_1(t,x(t)),
\end{aligned} \end{equation}
where 
\begin{equation}
    G_{a,b}f(t,x,y)= \frac{G_a f(t,x,y)- G_b f(t,x,y)}{a-b}
\end{equation}
and
\begin{equation}
    G_a f(t,x,y)= \left\{ \begin{aligned}
        &\ t^a \int_1^t \tau^{1-a} f(\tau, x(\tau), y(\tau)) \d \tau, && \mathrm{Re}(a) \le -\kappa, \\
        &-t^a \int_t^\infty \tau^{1-a} f(\tau, x(\tau), y(\tau)) \d \tau, && \mathrm{Re}(a) >- \kappa.
    \end{aligned} \right. 
\end{equation}
Using \eqref{eq uniform C^1 estimate of F_2} and direct calculation, we see $\Gamma: B \to B$ and is a contraction, so there exists a fixed point of $\Gamma$. This fixed point solves the ODE since
\begin{equation} \begin{aligned}
    \frac{\d}{\dt} (\Gamma x) &= \Gamma y+ F_1(t,x), \\
    \frac{\d}{\dt} (\Gamma y) &= \frac{A}{t^2} (\Gamma x)+ F_2(t,x,y). 
\end{aligned} \end{equation}

For the general case, let $\tilde{A}$ be within the previous case and $\| A- \tilde{A} \|< c_0$ and replace $F_2$ by $F_2+ \frac{(A-\tilde{A})x}{t^2}$. If $c_0$ is small enough, then $\Gamma$ defined similarly to the above will be a contraction, and the above argument can be applied to conclude.
\end{proof}

Let us go back to the proof of the proposition. The lemma accounts for solving $\alpha$ and $\beta$. (Note that $F_1= \mu(t) x+ \mu(t) \alpha^\infty(t)$ with $x=\alpha- \alpha^\infty$ satisfies the condition.) Then we use the argument in the proof of Proposition \ref{prop hyperbolic trajectory} in order to deal with $\lambda$, $\mu$ and $\delta$. 
\end{proof}

\begin{Rem} \label{rmk P^N- P^infty}
For $P^{(N)}$ constructed in Proposition \ref{prop hyperbolic trajectory} and \ref{prop parabolic trajectory}, using \eqref{eq parameters} and information about the first few terms \eqref{eq first few terms hyperbolic}, \eqref{eq first few terms b parabolic} and \eqref{eq first few terms d parabolic}, we have the following more precise asymptotic behavior:
\begin{equation}
    |\alpha^{(N)}(t)- \alpha^\infty(t)| =o(t^{-1+}),\; |\beta^{(N)}(t)- \beta^\infty(t)|= o(t^{-2+}), \; |\lambda^{(N)}(t)- \lambda^\infty|=o(t^{-2+}).
\end{equation}
Moreover,
\begin{itemize}
    \item in the hyperbolic case, we have 
    \begin{align} \label{hyp N est}
        |\alpha^{(N)}(t) | \lesssim t,\; |\beta^{(N)}(t)| \lesssim 1, \; \lambda^{(N)}(t) \sim 1,\;|\mu^{(N)}(t)| \lesssim t^{-2},\;|\delta^{(N)}(t)| \lesssim t^{-3},
    \end{align}
    \item in the parabolic case, we have 
    \begin{equation}\label{parab N est}
        |\alpha^{(N)}(t) | \lesssim t^{\f12},\; |\beta^{(N)}(t)| \lesssim t^{-\f12}, \; \lambda^{(N)}(t) \sim 1,\; |\mu^{(N)}(t)| \lesssim t^{-2},\; |\delta^{(N)}(t)| \lesssim t^{-3},
    \end{equation}
\end{itemize}
where the implicit constants only depend on $P^\infty$.
\end{Rem}

In summary, we have constructed $V_j^{(N)}(t,x)$ and $P^{(N)}(t)$ such that $S_j^{(N)}(t,x)=0$ and $\| V_j^{(N)}- Q \|_{H^1} \to 0$, $\| P^{(N)}- P^\infty \| \to 0$ as $t \to +\infty$. Furthermore, the conclusion of Proposition \ref{prop accuracy of approximate solution} holds true. 

\section{Bootstrap reduction}\label{sec:bootstrap}
In this section we introduce the main bootstrap argument, see Proposition \ref{prop uniform estimate}, which will then essentially require to bootstrap the modulation parameters and the error function, see Proposition \ref{prop bootstrap} below.\\[4pt]
Let us first reduce the proofs of Theorem \ref{thm hyperbolic} and Theorem \ref{thm parabolic} to the following uniform backwards estimate. 

\begin{Prop}[Uniform bound] \label{prop uniform estimate}
Let $P^{(N)}$ be defined as in Proposition \ref{prop hyperbolic trajectory} or \ref{prop parabolic trajectory} and $\gamma_j^{(N)}(t)$ be such that
\begin{equation}
    \gamma_j^{(N)}(0)=0, \ \dot{\gamma}_j^{(N)}= (\lambda_j^{(N)})^2- |\beta_j^{(N)}|^2- \big( \dot{\beta}_j^{(N)}+ 4\mu_j^{(N)} \beta_j^{(N)} \big) \cdot \alpha_j^{(N)}- \big( \dot{\mu}_j^{(N)}+ 4(\mu_j^{(N)})^2 \big) |\alpha_j^{(N)}|^2.
\end{equation}

Let $T_n \to +\infty$ and $u_n$ be the solution to 
\begin{equation} \label{eq equation of un}
    \left\{ \begin{aligned}
        &i\partial_t u_n+ \Delta u_n- \phi_{|u_n|^2} u_n= 0, \\
        &u_n(T_n,\cdot)= R_{g^{(N)}} ^{(N)}(T_n,\cdot),
    \end{aligned} \right.
\end{equation}
with backwards maximal interval $(t_n, T_n]$.
Then there exists $T_0=T_0(N)$ such that for $N$ large and $T_* \in [T_0,T_n] \cap (t_n, T_n]$, if
\begin{equation} \label{eq uniform estimate, bootstrap}
    \left\| u_n(t)- R_{g^{(N)}}^{(N)}(t) \right\|_{H^1} \le 2t^{-\frac{N}{9}}, \qquad \forall n\ge 1,\ \forall t \in [T_*,T_n],
\end{equation}
then
\begin{equation}
    \left\| u_n(t)- R_{g^{(N)}}^{(N)}(t) \right\|_{H^1} \le t^{-\frac{N}{9}}, \qquad \forall n\ge 1,\ \forall t \in [T_*,T_n].
\end{equation}
\end{Prop}

We show that the above implies the theorems.

\begin{proof} [Proof of Theorem \ref{thm hyperbolic} and \ref{thm parabolic} by Proposition \ref{prop uniform estimate}] \ \par
Let us fix $N \gg1 $ large such that the conclusion holds. By the standard bootstrap argument, we know \eqref{eq uniform estimate, bootstrap} actually holds for any $T_* \in [T_0,T_n] \cap (t_n, T_n]$. Further, by the blow-up criterion, if $t_n>- \infty$, then $\| u_n(t) \|_{H^1} \to +\infty$ as $t \to t_n$. Thus we must have $t_n=- \infty$ or $t_n< T_0$, where the former can be understood as a special case of the latter. Moreover, we know \eqref{eq uniform estimate, bootstrap} holds for $T_*= T_0$. 

We deduce that there exists $C>0$ such that
\begin{equation} \label{eq uniform estimate of un, boundedness}
    \| u_n(t) \|_{H^1} \le C, \quad \forall n \ge 1,\ \forall t \in [T_0,T_n].
\end{equation}
Also, for any $\delta>0$, there exist $r=r(\delta)>0$ and $t_0=t_0(\delta)>T_0$ such that 
\begin{equation}
    \int_{|x|>r} |u_n(t_0,x)|^2 \d x< \delta, \quad \forall n \ge 1.
\end{equation}
We claim that there exists $r'=r'(\delta)>0$ such that
\begin{equation} \label{eq uniform estimate of un, decay}
    \int_{|x|>r'} |u_n(T_0,x)|^2 \d x< 2\delta, \quad \forall n \ge 1.
\end{equation}

In order to prove \eqref{eq uniform estimate of un, decay}, let $\Phi \in C^\infty(\mathbb{R})$ be a cutoff function satisfying
\begin{equation}
    0 \le \Phi \le 1, \quad 0 \le \Phi' \le 2, \quad \Phi(x)= \left\{ \begin{aligned}
        0, &&x \le 0, \\
        1, &&x \ge 1.
    \end{aligned} \right. 
\end{equation}
For  $L>0$ we now define $z(t)= \int |u_n(t,x)|^2 \Phi( \frac{|x|-r}{L}) \d x$. Then $z(t_0) \le \delta$. Since
\begin{equation}
    z'(t)= -2 \mathrm{Im} \int \Delta u_n \overline{u_n} \Phi \Big( \frac{|x|-r}{L} \Big) \d x= \frac{2}{L} \mathrm{Im} \int \nabla u_n \overline{u_n} \cdot \frac{x}{|x|} \Phi' \Big( \frac{|x|-r}{L} \Big) \d x,
\end{equation}
we have $|z'(t)| \le \frac{4}{L} \| u_n \|_{H^1}^2$. Integrating in $t$ and using \eqref{eq uniform estimate of un, boundedness}, we get 
\begin{equation}
    \int_{|x|>L+r} |u_n(T_0,x)|^2 \d x \le z(T_0) \le \frac{4C^2(t_0-T_0)}{L}+ \delta.
\end{equation}
We deduce \eqref{eq uniform estimate of un, decay} by taking $L=L(\delta)$ large enough and $r'=L+r$.

Now \eqref{eq uniform estimate of un, boundedness} and \eqref{eq uniform estimate of un, decay} imply the existence of a subsequence $u_{n_k}(T_0)$ of $u_n(T_0)$ that converges in $L^2$ to some $U_0$ and $U_0 \in H^1$. Let $U$ be the solution to
\begin{equation} 
    \left\{ \begin{aligned}
        &i\partial_t U+ \Delta U- \phi_{|U|^2} U= 0, \\
        &U(T_0)=U_0
    \end{aligned} \right.
\end{equation}
with forwards maximal interval $[T_0,T)$. By local well-posedness, we have $u_{n_k}(t) \to U(t)$ in $L^2$ for any $t \in [T_0,T)$. Thanks to \eqref{eq uniform estimate of un, boundedness}, by passing to a subsequence, we may assume $u_{n_k}(t) \rightharpoonup U(t)$ in $H^1$. Using \eqref{eq uniform estimate, bootstrap} and Fatou's lemma, we deduce
\begin{equation}
    \left\| U(t)- R_{g^{(N)}}^{(N)}(t) \right\|_{H^1} \le 2t^{-\frac{N}{9}}, \quad \forall t \in [T_0,T).
\end{equation}
Then by the blow-up criterion, we see $T=+\infty$.

Finally, take $\mu^\infty= \delta^\infty=0$ and let $\gamma^\infty(t)$ be such that
\begin{equation}    
    \gamma_j^\infty(0)=0, \quad \dot{\gamma}_j^\infty= (\lambda_j^\infty)^2- |\beta_j^\infty|^2- \big( \dot{\beta}_j^\infty+ 4\mu_j^\infty \beta_j^\infty \big) \cdot \alpha_j^\infty- \big( \dot{\mu}_j^\infty+ 4(\mu_j^\infty)^2 \big) |\alpha_j^\infty|^2.
\end{equation}
Using the fundamental theorem of calculus, we have 
\begin{equation} \label{taylor-boot}
    \left\| R_{g^{(N)}}^{(N)}(t)- R_{g^{\infty}}^{(N)}(t) \right\|_{H^1} \lesssim |g^{(N)}(t)- g^\infty(t)|+ a |\beta^{(N)}(t)- \beta^\infty(t)|+ a^2 |\mu^{(N)}(t)|.
\end{equation}
Then we obtain the conclusion of the theorem by Remark \ref{rmk P^N- P^infty}.
\end{proof}

In order to control the error in the bootstrapping step,  we still need to impose orthogonality between the modulation trajectories and the zero modes of the linearized operators. However, compared to subcritical cases, the linearized operators in the mass-critical equation \eqref{H} have more elements in the generalized root space due to the pseudo-conformal symmetry.\\[5pt]
Consequently we use the related additional parameters $\mu_j$ and $\delta_j$ to obtain more orthogonality conditions, which is an essential difference compared to \cite{KRM,Wu}. In particular, we have the following Lemma.

\begin{Lemma} \label{lem orthogonality}
Let $N,n \ge 1$. Then there exist $T_0=T_0(N)>0$ and a unique modulation parameter $g \in C^1([T_0,+\infty), \Omega \times (\R/2\pi\Z)^m)$ such that: if 
\begin{equation}
    \varepsilon(t,x)= u_n(t,x)- R_g^{(N)}(t,x),
\end{equation}
then for $t \ge T_0$ and $1\le j \le m$, we have
\begin{equation} \label{eq orthogonality} \begin{aligned}
    &\ \re \Big( \varepsilon(t), g_j Q \Big)= \re \Big( \varepsilon(t), g_j(x Q) \Big)= \re \Big( \varepsilon(t), g_j \big( |x|^2 Q \big) \Big) \\
    &= \imm \Big( \varepsilon(t), g_j \big( \Lambda Q \big) \Big)= \imm \Big( \varepsilon(t), g_j \big( \nabla Q \big) \Big)= \imm \Big( \varepsilon(t), g_j \rho \Big)= 0.
\end{aligned} \end{equation} 
In particular, we have
\begin{equation} \label{eq value of epsilon at T_n}
    g(T_n)=g^{(N)}(T_n), \quad \varepsilon(T_n)=0.
\end{equation}
\end{Lemma} 

\begin{proof}
The proof is similar to the proof of \cite[Lemma~3.1]{Wu} and we refer to \cite{Wu} for details. We also note that the proof relies on the assumption \eqref{eq uniform estimate, bootstrap}.
\end{proof}
\;\\
We now further reduce Proposition \ref{prop uniform estimate} to the modulation estimate in the following bootstrap argument.

\begin{Prop}[Bootstrap argument] \label{prop bootstrap}
For $N$ and $T_0=T_0(N)$ large enough, $\forall n\ge 1,\ T_* \in [T_0,T_n] \cap (t_n, T_n]$, if
\begin{equation} \label{eq bootstrap assumption}
   \left\{ \begin{aligned}
       &\| \varepsilon(t) \|_{H^1} \le t^{-\frac{N}{4}}, \quad \| x \varepsilon(t) \|_{L^2} \le t^{-\frac{N}{4}+2}, \\
       &\sum_{j=1}^m \left| \lambda_j(t)- \lambda_j^{(N)}(t) \right|+ \left| \beta_j(t)- \beta_j^{(N)}(t) \right| \le t^{-1-\frac{N}{8}},\\
       &\sum_{j=1}^m \left| \gamma_j(t)- \gamma_j^{(N)}(t) \right|+ \left| \alpha_j(t)- \alpha_j^{(N)}(t) \right| \le t^{-\frac{N}{8}},\\
       &\sum_{j=1}^m \left| \mu_j(t)- \mu_j^{(N)}(t) \right| \le t^{-\frac{5}{2} - \frac{N}{8}},\;\;\;\sum_{j=1}^m \left| \delta_j(t)- \delta_j^{(N)}(t) \right| \le t^{-\frac{7}{2} - \frac{N}{8}},
   \end{aligned} \right. 
\end{equation}
for any $t \in [T_*,T_n]$, then
\begin{equation} \label{eq bootstrap conclusion}
   \left\{ \begin{aligned}
       &\| \varepsilon(t) \|_{H^1} \le \frac{1}{2} t^{-\frac{N}{4}}, \quad \| x \varepsilon(t) \|_{L^2} \le \frac{1}{2} t^{-\frac{N}{4}+2}, \\
       &\sum_{j=1}^m \left| \lambda_j(t)- \lambda_j^{(N)}(t) \right|+ \left| \beta_j(t)- \beta_j^{(N)}(t) \right| \le \frac{1}{2} t^{-1-\frac{N}{8}}, \\
       &\sum_{j=1}^m \left| \gamma_j(t)- \gamma_j^{(N)}(t) \right|+ \left| \alpha_j(t)- \alpha_j^{(N)}(t) \right| \le \frac{1}{2} t^{-\frac{N}{8}},\\
       & \sum_{j=1}^m \left| \mu_j(t)- \mu_j^{(N)}(t) \right| \le \frac{1}{2} t^{-\frac{5}{2} - \frac{N}{8}},\;\;\; \sum_{j=1}^m \left| \delta_j(t)- \delta_j^{(N)}(t) \right| \le \frac{1}{2}t^{-\frac{7}{2} - \frac{N}{8}}    
   \end{aligned} \right. 
\end{equation}
for any $t \in [T_*,T_n]$.
\end{Prop}

\begin{proof} [Proof of Proposition \ref{prop uniform estimate} by Proposition \ref{prop bootstrap}] \ \par
Since the left-hand sides are continuous in $t$, by a bootstrap argument, we know \eqref{eq bootstrap assumption} actually holds for any $t \in [T_0,T_n] \cap (t_n, T_n]$. Then we have
\begin{equation} \begin{aligned}
    \left\| u_n(t)- R_{g^{(N)}}^{(N)} \right\|_{H^1} &\le \left\| u_n(t)- R_g^{(N)} \right\|_{H^1}+ \left\| R_g^{(N)}- R_{g^{(N)}}^{(N)} \right\|_{H^1} \\
    &\le \| \varepsilon(t) \|_{H^1}+ C_N a^2 |g- g^{(N)}| \le t^{-\frac{N}{4}}+ C_N t^{-\frac{N}{8}+2} \le t^{-\frac{N}{9}}
\end{aligned} \end{equation}
when $T_0(N)$ is large enough. Here we also used the estimate \eqref{taylor-boot}. This finishes the proof of Proposition \ref{prop uniform estimate}.
\end{proof}
Let us give a short explanation on the bootstrap assumptions. We note that, in the proof of Proposition \ref{prop bootstrap}, it is fundamental to impose the orthogonality conditions from Lemma \ref{lem orthogonality} and hence the coercivity result in Proposition \ref{prop coercivity} below in Section \ref{sec:mod-subsec2}. To obtain the orthogonality conditions, we used assumption \eqref{eq uniform estimate, bootstrap}.\\[4pt]
Likewise \eqref{eq bootstrap assumption} will be essential to obtain both, the parameter estimates in \eqref{eq bootstrap conclusion}, proved in Section \ref{sec:mod-subsection}, and the error estimate in Proposition \ref{prop estimate on G(epsilon)} below in Section \ref{sec:mod-subsec2}. In combination with the coercivity, this will give us the final control in the first line of \eqref{eq bootstrap conclusion}.
\section{Modulation and error control}\label{sec:mod}
In this section, we give a proof of Proposition \ref{prop bootstrap} in the previous Section \ref{sec:bootstrap}. 
\subsection{Modulation} \label{sec:mod-subsection}
Let us start and consider the  \emph{modulation parameters $P$ in Proposition \ref{prop bootstrap}},  for which we have the orthogonality conditions in Lemma \ref{lem orthogonality} available.\;\\[4pt]
First we note combining \eqref{hyp N est} and \eqref{parab N est} with the bootstrap assumption \eqref{eq bootstrap assumption} in Proposition \ref{prop bootstrap}, we obtain on $ [T_*, T_n]$
\begin{align}\label{param-est-hyp}
&|\alpha(t) | \lesssim t,\; |\beta(t)| \lesssim 1,\; |\lambda(t)| \sim 1,\; |\mu(t)| \lesssim t^{-2},\;|\delta(t)| \lesssim t^{-3}\\ \label{param-est-par}
& |\alpha(t) | \lesssim  t^\f12,\; |\beta(t)| \lesssim t^{-\f12},\; |\lambda(t)| \sim 1,\; |\mu(t)| \lesssim t^{-2},\; |\delta(t)| \lesssim t^{-3},
\end{align}
for the modulation parameters of Lemma \ref{lem orthogonality}. Here, as usual, \eqref{param-est-hyp} is the hyperbolic and \eqref{param-est-par} the parabolic case.\\[5pt]
\emph{\textbf{Step 1.}} \emph{Modulation estimate.}\;  We write $u = u_n$ for simplicity and let the error for the modulation equations \eqref{eq parameters} be defined as
\begin{equation} \label{Mod-error}\begin{aligned} 
    Mod(t) := &\ \big| \dot{\alpha}_j- 2\beta_j- 4\mu_j \alpha_j \big|+ \big| \dot{\beta}_j+ 4\mu_j \beta_j+ (\dot{\mu}_j+ 4\mu_j^2) \alpha_j- B_j^{(N)} \big| \\
    &+ \big| \dot{\lambda}_j+ 4\lambda_j \mu_j- M_j^{(N)} \big| + \big| \dot{\mu}_j+ 4\mu_j^2- \lambda_j^4 \delta_j \big|+ \big| \dot{\delta}_j- D_j^{(N)} \big| \\
    &+ \big| \dot{\gamma}_j+ (\dot{\beta}_j+ 4\mu_j \beta_j) \cdot \alpha_j+ (\dot{\mu}_j+ 4\mu_j^2) |\alpha_j|^2+ |\beta_j|^2- \lambda_j^2 \big|,
\end{aligned} \end{equation}
where now $M_j^{(N)} = M_j^{(N)}(P),\;B_j^{(N)} = B_j^{(N)}(P),\; D_j^{(N)} = D_j^{(N)}(P)  $. Expanding the Hartree equation \eqref{H} by $u = \varepsilon + R$ with $R =R_g^{(N)}$ in \eqref{eq approximate solution}, we obtain $\varepsilon$ satisfies
\begin{align} \label{eq equation of epsilon}
    i\partial_t \varepsilon+ \Delta \varepsilon- &\phi_{|R|^2} \varepsilon- 2\phi_{\re (\varepsilon \overline{R}) } R\\  \nonumber
    &= \Psi+ \cN(\varepsilon)+ \sum_{j=1}^m S_j(y_j) e^{i(\gamma_j+ \beta_j \cdot x+ \mu_j |x|^2)},
\end{align}
where $\Psi= \Psi^{(N)}$ is as in \eqref{eq definition of Psi} and 
\begin{align}\label{eq nonlin epsilon}
    \cN(\varepsilon)= 2\phi_{\mathrm{Re} (\varepsilon \overline{R})} \varepsilon+ \phi_{|\varepsilon|^2}R+ \phi_{|\varepsilon|^2} \varepsilon.
\end{align}
\begin{Prop} \label{prop modulation estimate}
We have
\begin{equation} \label{eq modulation estimate}
Mod(t) \le \frac{C\| \varepsilon \|_{H^1}}{a^2}+ \frac{C_N}{a^{N+2}}+ C_N \| \varepsilon \|_{H^1}^2\;\;\;\forall \; t \in [T_*, T_n]
\end{equation}
\end{Prop}
\begin{proof}
For $\Psi$ and $ \cN(\varepsilon)$ in \eqref{eq equation of epsilon} we infer 
\begin{equation} \label{some estimates}
    |\Psi| \lesssim a^{-N-2}, \quad |\cN(\varepsilon)| \lesssim \| \varepsilon \|_{H^1}^2,
\end{equation}
where for the latter we use the  Hardy-Littlewood-Sobolev inequality and Soboelv inequality. We derive \eqref{eq modulation estimate} by exploiting the orthogonality conditions in Lemma \ref{lem orthogonality}. Let $\theta_j= g_j \theta$ be the modulation of a smooth decaying function, i.e. such that
\begin{align} \label{eq decay of theta}
  &\theta_j(y_j) = \lambda_j^2 \theta(\lambda_j(x - \alpha_j)) e^{i(\gamma_j + \beta_j \cdot x + \mu_j|x|^2)},\\[5pt] \nonumber
  &|\nabla^k \theta(x)| \le C_k e^{-c_k|x|}, \quad \forall x \in \mathbb{R}^4,\;\; k \geq 0.
\end{align}
We then compute 
\begin{equation} \begin{aligned}
    \frac{\d}{\d t} \bigg(\mathrm{Im} \int \varepsilon \overline{\theta_j}\;\d x\bigg)= &\ \mathrm{Re} \int \varepsilon \left( \overline{i\partial_t \theta_j+ \Delta \theta_j- 2\phi_{\mathrm{Re} (\theta_j \overline{R})}R- \phi_{|R|^2} \theta_j} \right) \d x \\
    &+ \mathrm{Re} \int (\Psi- \cN(\varepsilon)) \overline{\theta_j} \;\d x+ \sum_{j=1}^m \mathrm{Re} \int S_j(t,x) e^{i\gamma_j+ i\beta_j \cdot x+ i\mu_j |x|^2} \overline{\theta_j} \;\d x.
\end{aligned} \end{equation}

For $\theta_j$ we calculate 
\begin{equation} \begin{aligned}
    i\partial_t \theta_j+ \Delta \theta_j &= \lambda_j^4 \left( \Delta \theta(y_j)- \theta(y_j) \right) e^{i\gamma_j+ i\beta_j \cdot x + i \mu_j |x|^2} \\
    &\quad + O \left( \frac{1}{a^3}+ \frac{C_N}{a^4}+ Mod \right) \left( e^{-c |x-\alpha_j|}+ \frac{C_N}{a} e^{-c_N|x-\alpha_j|} \right),\\[6pt]
    \phi_{\mathrm{Re} (\theta_j \overline{R})} R &= \lambda_j^4  \phi_{\mathrm{Re} (\theta \overline{V_j})} V_j(y_j) e^{i\gamma_j+ i\beta_j \cdot x + i \mu_j |x|^2}+ \phi_{\mathrm{Re} (\theta_j \overline{R}_j)} \sum_{k \neq j} R_k \\
    &\quad + O_N \Big( e^{-c_N a} \max_k e^{-c_N|x-\alpha_k|} \Big),
\end{aligned} \end{equation}
and
\begin{equation} \begin{aligned}
    \phi_{|R|^2} \theta_j = \lambda_j^4  \bigg( \phi_{|V_j|^2}+ \sum_{k \neq j} \psi_{|V_k|^2}^{(1)} \bigg) \theta(y_j)  e^{i\gamma_j+ i\beta_j \cdot x + i \mu_j |x|^2}+ O_N \left( \frac{1}{a^3} e^{-c_N |x-\alpha_j|} \right).
\end{aligned} \end{equation}

Considering \eqref{eq equation of epsilon}, we now collect the terms of degree $1$ in $\theta$ and let
\begin{equation}
    L_j \theta:= -\Delta \theta+ \theta+ 2\phi_{\mathrm{Re} (\theta \overline{V_j})} V_j+ \bigg( \phi_{|V_j|^2}+ \sum_{k \neq j} \psi_{|V_k|^2}^{(1)} \bigg) \theta. 
\end{equation}
If we take $T_0(N)$ large enough, then we have
\begin{equation} \begin{aligned}
    i\partial_t \theta_j+ \Delta \theta_j- &2\phi_{\mathrm{Re} (\theta_j \overline{R})}R- \phi_{|R|^2} \theta_j \\[5pt]
    =&\ -\lambda_j^4 (L_j \theta)(y_j) e^{i\gamma_j+ i\beta_j \cdot x + i \mu_j |x|^2}- 2\phi_{\mathrm{Re} (\theta_j \overline{R}_j)} \sum_{k \neq j} R_k\\
    &\ + O_N \left( \frac{1}{a^3}+ Mod \right) e^{-c_N \min_j |x-\alpha_j|}.
\end{aligned} \end{equation}
Inserting this into the previous formula, we get
\begin{equation} \begin{aligned}
    \frac{\d}{\dt} \bigg(\mathrm{Im} \int \varepsilon \overline{\theta_j}\; \d x \bigg)= &\ -\mathrm{Re} \int \varepsilon \lambda_j^2\overline{ (L_j \theta) e^{i\gamma_j+ i\beta_j \cdot x + i \mu_j |x|^2}}- 2\mathrm{Re} \int \varepsilon \phi_{\mathrm{Re} (\theta_j \overline{R}_j)} \sum_{k \neq j} \overline{R_k}\; \d x \\
    &+ \lambda_j^6 \mathrm{Re} \int S_j \overline{\theta}+ O_N \left( \frac{\| \varepsilon \|_{H^1}}{a^3}+ Mod \| \varepsilon \|_{H^1}+ \frac{1}{a^{N+2}}+ \| \varepsilon \|_{H^1}^2 \right).  
\end{aligned} \end{equation}

If we take $\theta$ to be any element in the set 
$$  \Xi = \{ iQ,\; ix Q,\;i|x|^2 Q,\; \Lambda Q,\;\nabla Q,\;\rho \}$$
we have
\begin{equation}
    L_j \theta=f+ O\left( \frac{1}{a^2}+ \frac{C_N}{a^3} \right) \left( e^{-c |x-\alpha_j|}+ \frac{C_N}{a} e^{-c_N|x-\alpha_j|} \right),
\end{equation}
where $\re \ (\varepsilon, g_j f)=0$ using \eqref{eq orthogonality}. By the proof of Proposition \ref{prop accuracy of approximate solution}, we further have
\begin{equation}
    \phi_{\mathrm{Re} (\theta_j \overline{R_j})}= O \left( \frac{1}{a^2}+ \frac{C_N}{a^3} \right) \big( 1+|x-\alpha_j| \big)^2.
\end{equation}
Finally, we check 
\begin{equation} \label{eq estimate on theta: S_j}
    \sum_{\theta \in \Xi } \sum_{j=1}^m \left| \mathrm{Re} \int S_j(t,x) \overline{\theta}(t,x)\; \d x \right| \ge c\ Mod(t)- \frac{C_N}{a} Mod(t).
\end{equation}
Combining these together, we obtain
\begin{equation}
    Mod(t) \le \frac{C\| \varepsilon \|_{H^1}}{a^2}+ \frac{C_N \| \varepsilon \|_{H^1}}{a^3}+ \frac{C_N Mod(t)}{a}+ \frac{C_N}{a^{N+2}} + C_N\| \varepsilon \|_{H^1}^2. 
\end{equation}
Taking $T_0(N)$ large enough to absorb all $O_N$ terms, we get
\begin{equation}
    Mod(t) \le \frac{C\| \varepsilon \|_{H^1}}{a^2}+ \frac{C_N}{a^{N+2}}+ C_N \| \varepsilon \|_{H^1}^2 
\end{equation} 
for $ t \in [T_*, T_n]$. This concludes the claim.
\end{proof}
\;\\
\emph{\textbf{Step 2.}} \emph{Integration of the modulation estimate.}\;\; We now prove the second, third and fourth line of \eqref{eq bootstrap conclusion} assuming \eqref{eq bootstrap assumption} in Proposition \ref{prop bootstrap}. First we note taking $ N \gg1 $ and $T_0 \gg1$ large we have from Proposition \ref{prop modulation estimate}
\[
Mod(t) \leq C t^{- \frac{N}{4}-1},\;\;\forall\; t \in [T_*, T_n].
\]
Recall that we have \eqref{param-est-hyp} and \eqref{param-est-par} (see also \eqref{hyp N est} and \eqref{parab N est}).

Using the admissibility of $ M_j^{(N)}, B_j^{(N)}$ at least of degree greater than or equal to  $3$ in Section \ref{sec:traje}, we have in the hyperbolic case
\begin{align} 
&\quad \ \left | M_j^{(N)}(P) - M_j^{(N)}(P^{(N)})\right |  + \left | B_j^{(N)}(P) - B_j^{(N)}(P^{(N)})\right |\\
&\leq C\sum_{j=1}^m\Big(\frac{|\alpha_j - \alpha_j^{(N)}| }{a^4} + \frac{|\beta_j - \beta_j^{(N)}|+ |\lambda_j - \lambda_j^{(N)}|}{a^3} + \frac{|\mu_j - \mu_j^{(N)} |}{a} + |\delta_j - \delta_j^{(N)} | \Big),
\end{align}
and in the parabolic case 
\begin{align} 
&\quad \ \left | M_j^{(N)}(P) - M_j^{(N)}(P^{(N)})\right |  + \left | B_j^{(N)}(P) - B_j^{(N)}(P^{(N)})\right |\\
&\leq C\sum_{j=1}^m\Big(\frac{|\alpha_j - \alpha_j^{(N)}| }{a^4} + \frac{|\beta_j - \beta_j^{(N)}|}{a^2} + \frac{|\lambda_j - \lambda_j^{(N)}|}{a^3} + a|\mu_j - \mu_j^{(N)} | + a^3|\delta_j - \delta_j^{(N)} | \Big).
\end{align}
By \eqref{eq bootstrap assumption}, in both cases we have
\begin{equation} \label{first bound mod int}
    \left | M_j^{(N)}(P) - M_j^{(N)}(P^{(N)})\right |  + \left | B_j^{(N)}(P) - B_j^{(N)}(P^{(N)})\right | \leq C t^{-2- \frac{N}{8}}.
\end{equation}
Thus we further conclude from \eqref{eq parameters}, \eqref{first bound mod int} and \eqref{eq bootstrap assumption} 
\begin{align}
&\quad \ |\dot{\lambda}_j - \dot{\lambda}^{(N)}_j | + |\dot{\beta}_j - \dot{\beta}^{(N)}_j |\\
&\leq\;  Mod(t) + \left | M_j^{(N)}(P) - M_j^{(N)}(P^{(N)})\right |  + \left | B_j^{(N)}(P) - B_j^{(N)}(P^{(N)})\right |\\
&\;\;\;\;+ | \lambda_j \mu_j  - \lambda_j^{(N)} \mu_j^{(N)}| +  | \mu_j \beta_j  - \mu_j^{(N)} \beta_j^{(N)}|+ |\lambda_j^4 \delta_j \alpha_j   - (\lambda_j^{(N)})^4 \delta_j^{(N)} \alpha_j^{(N)} |\\
&\leq C t^{- \frac{N}{8} -2}.
\end{align}
Integrating and using \eqref{eq value of epsilon at T_n} we infer 
\[
|\lambda_j(t) - \lambda^{(N)}_j(t) | + |\beta_j(t) - \beta^{(N)}_j(t) | \leq \frac{C}{N}t^{- \frac{N}{8} -1} \leq \frac{1}{2}t^{- \frac{N}{8} -1} 
\]
after taking $N \gg1 $ large.\\[4pt]
Next we note, by the construction in Section \ref{sec:traje}, $\delta_j $ in $D_j^{(N)}$ can only appear in functions of degree greater than or equal to $5$ in the hyperbolic case, hence
\begin{align}
&\quad \ \left | D_j^{(N)}(P) - D_j^{(N)}(P^{(N)}) \right | \\
&\leq C\sum_{j=1}^m\Big(\frac{|\alpha_j - \alpha_j^{(N)}| }{a^6} + \frac{|\beta_j - \beta_j^{(N)}| + |\lambda_j - \lambda_j^{(N)} |}{a^5} + \frac{|\mu_j - \mu_j^{(N)} |}{a^3} + \frac{|\delta_j - \delta_j^{(N)} |}{a^2} \Big).
\end{align}
In the parabolic case, $D_j^{(N)}$ is of degree greater than or equal to $8$, so we have
\begin{align}
&\quad \ \left | D_j^{(N)}(P) - D_j^{(N)}(P^{(N)}) \right | \\
&\leq C\sum_{j=1}^m\Big(\frac{|\alpha_j - \alpha_j^{(N)}| }{a^9} + \frac{|\beta_j - \beta_j^{(N)}| }{a^7} +\frac{|\lambda_j - \lambda_j^{(N)} |}{a^8}+ \frac{|\mu_j - \mu_j^{(N)} |}{a^4} + \frac{|\delta_j - \delta_j^{(N)} |}{a^2} \Big).
\end{align}
In both cases, we obtain
\begin{equation}
    |\dot{\delta_j} - \dot{\delta}^{(N)}_j | \leq Mod(t)+ \left | D_j^{(N)}(P) - D_j^{(N)}(P^{(N)}) \right | \le C t^{-\frac{9}{2}- \frac{N}{8}}.
\end{equation}
Therefore integrating as before gives $|\delta_j - \delta^{(N)}_j | \leq \f12 t^{-\frac{7}{2}- \frac{N}{8}}$. 

Directly from \eqref{eq parameters} we infer 
\begin{align}
& | \dot{\alpha_j} - \dot{\alpha}_j^{(N)}|  \leq Mod(t) + 2| \beta_j - \beta_j^{(N)}| + 4| \mu_j \alpha_j - \mu_j^{(N)} \alpha_j^{(N)}| \leq C t^{-1- \frac{N}{8}},\\   
&  | \dot{\mu_j} - \dot{\mu}_j^{(N)}|  \leq Mod(t) + 4 |\mu_j^2 - (\mu_j^{(N)})^2 | + | \lambda_j^4 \delta_j - (\lambda_j^{(N)})^4 \delta_j^{(N)}| \leq C t^{-\frac{7}{2}- \frac{N}{8}}.
\end{align} 
Hence integrating and taking $ N \gg1 $ large implies  again
\[
| \alpha_j(t) - \alpha_j^{(N)}(t)| \leq \f12  t^{- \frac{N}{8}},\;\;|\mu_j(t) - \mu_j^{(N)}(t)| \leq \f12 t^{-\frac{5}{2}- \frac{N}{8}}.
\]

Finally, for the phase $\gamma_j$ we collect the above estimates and infer
\begin{align}
|\dot{\gamma}_j - \dot{\gamma}_j^{(N)}| &\leq C\Big( |\lambda_j - \lambda_j^{(N)} | + |\beta_j - \beta_j^{(N)} | +a |\dot{\beta}_j - \dot{\beta}_j^{(N)} | + \frac{|\alpha_j - \alpha_j^{(N)}|}{a^2} \\
&\qquad \;\; + a |\mu_j \beta_j- \mu_j^{(N)} \beta_j^{(N)}|+ a^2 |\dot{\mu}_j- \dot{\mu}_j^{(N)}| \Big) + Mod(t)\\
&\leq   C t^{-1- \frac{N}{8}}.
\end{align}
Then we conclude by integrating the above inequality.

\subsection{Error estimate} \label{sec:mod-subsec2} Here we give a proof of the first line of the bootstrap estimate \eqref{eq bootstrap conclusion}.\\[5pt]
To begin with, as in \cite{KRM} and \cite{Wu}, we linearize the conserved energy around the approximation $ R$. Thus writing again $ u = R + \varepsilon$ we have
\begin{align}\label{energy-cons}
    2\mathcal{E}(u_0) =& \;\;2 \mathcal{E}(R) -2 \mathrm{Re} (\varepsilon, \overline{\Delta R - \phi_{|R|^2} R}) + \int |\nabla \varepsilon|^2\;\d x\\ \nonumber
    & \;\;+ \int \phi_{|R|^2} |\varepsilon|^2 \;\d x - \frac{1}{2\pi^2} \int |\nabla \phi_{\mathrm{Re}(\varepsilon \bar{R})}|^2 \; \d x + 2 \int \phi_{\mathrm{Re}(\varepsilon \bar{R})}|\varepsilon|^2\; \d x \\ \nonumber
    &\;\; - \frac{1}{8\pi^2} \int |\nabla \phi_{|\varepsilon|^2}|^2\; \d x.
\end{align}
For deriving \eqref{eq bootstrap conclusion} the conservation law \eqref{energy-cons} is not enough and therefore we add localizations of mass, momentum, their center and the variance along $ \alpha(t)$ to quadratic and higher order terms in \eqref{energy-cons}.\\[4pt]
We stress that localized mass and momentum also appear in \cite{Wu}, however the functional $\G_4$ and $\G_5$ below are new in the 4D problem \eqref{H} and necessary for the final error control. We use the following Lemma without proof.
\begin{Lemma}  \label{lem cutoff}
There exist $c,C>0$ and $\varphi_j \in C^{1,\infty} (\mathbb{R}_+ \times \mathbb{R}^4)$ for $1 \le j \le m$ such that
\begin{equation}\label{eq cutoff} \begin{gathered}
    0 \le \varphi_j(t,x) \le 1, \quad \sum_{j=1}^m \varphi_j(t,x) \equiv 1, \\ 
    |\partial_t \varphi_j|+ |\nabla \varphi_j| \le \frac{C}{a}, \quad |\partial_t \sqrt{\varphi_j}|+ |\nabla \sqrt{\varphi_j}| \le \frac{C}{a}, \\
    \varphi_j(t,x)= \left\{ \begin{aligned}
        &1, \quad |x-\alpha_j(t)| \le ca(t), \\
        &0, \quad |x-\alpha_k(t)| \le ca(t),\ k \neq j.
    \end{aligned} \right.
\end{gathered} \end{equation}
\end{Lemma}
We note that the cut-off functions $\varphi_j$ localize the approximate solution $R$ in \eqref{eq approximate solution} of Section \ref{sec:approx} to $R_j$. To be precise we have using Lemma \ref{lem properties of admissible functions-2} for some $c> 0$
\begin{align}\label{est local varphi}
    \sup_{x \in \R^4}| \varphi_j(t)R(t) - R_j(t) | \leq C_N e^{- c a(t)},\;\; t \geq T_0.
\end{align}
Let us now define $\displaystyle \G(\varepsilon)= \sum_{k=1}^5\G_k(\varepsilon)$, where
\begin{align*}
    \G_1(\varepsilon)=&\; \int |\nabla \varepsilon|^2+ \int \phi_{|R|^2} |\varepsilon|^2- \frac{1}{2\pi^2}\int |\nabla \phi_{\mathrm{Re} (\varepsilon \overline{R}) }|^2+ 2\int \phi_{\mathrm{Re} (\varepsilon \overline{R}) } |\varepsilon|^2\\
    &\;\;- \frac{1}{8 \pi^2} \int |\nabla \phi_{|\varepsilon|^2} |^2,\\[3pt]
    \G_2(\varepsilon)=&\; \sum_{j=1}^m \Big( \lambda_j^2+ |\beta_j|^2 \Big) \int \varphi_j |\varepsilon|^2,\;\;\; \G_3(\varepsilon)=\; -2\sum_{j=1}^m \beta_j \int \varphi_j \mathrm{Im} (\nabla \varepsilon \overline{\varepsilon}),\\[3pt]
    \G_4(\varepsilon) =&\; 4 \sum_{j=1}^m \mu_j^2 \int \varphi_j  |x|^2|\varepsilon|^2, \;\;\; \G_5(\varepsilon) =\; 4 \sum_{j=1}^m \mu_j \beta_j \int \varphi_j x |\varepsilon|^2 - 4 \sum_{j=1}^m \mu_j \int \varphi_j \mathrm{Im} (x \nabla \varepsilon \overline{\varepsilon}).
\end{align*}

\;\\[4pt]
We proceed by proving coercivity for the functional $\G$, which relies on Lemma \ref{lem coercivity} in Section \ref{sec:gs-op-properties} combined with the orthogonality in \eqref{eq orthogonality} of Lemma \ref{lem orthogonality}. Then we combine the following Proposition \ref{prop coercivity} with an upper bound in Proposition \ref{prop estimate on G(epsilon)} assuming the bootstrap condition \eqref{eq bootstrap assumption}.
\begin{Prop} \label{prop coercivity}
Let $N \ge 2$. For $T_0=T_0(N)$ large enough, there exists $c_0>0$ such that
\begin{equation}
    \G(\varepsilon(t)) \ge c_0 \| \varepsilon(t) \|_{H^1}^2,
\end{equation}
for all $ t \in [T_*,T_n]$.
\end{Prop}

\begin{proof}
Let $\varepsilon_j= \varepsilon \sqrt{\varphi_j}$ and $\tilde{\varepsilon}_j= g_j^{-1} \varepsilon_j$, or more precisely, define
\begin{equation}
    \tilde{\varepsilon}_j(t,y_j)= \frac{1}{\lambda_j^2(t)} \varepsilon_j(t, \lambda_j^{-1}(t)y_j+ \alpha_j(t)) e^{- i(\gamma_j(t) + \beta_j(t) \cdot (\lambda_j^{-1}y_j+ \alpha_j(t)) + \mu_j |\lambda_j^{-1}y_j+ \alpha_j(t)|^2)}.
\end{equation}
For $T_0(N) \gg1$ large enough, by Lemma \ref{lem coercivity}, Lemma \ref{lem properties of admissible functions-2} and the orthogonality \eqref{eq orthogonality}, we have
\begin{equation}
    \Big( L_+ \mathrm{Re}(\tilde{\varepsilon}_j), \mathrm{Re}(\tilde{\varepsilon}_j) \Big)+ \Big( L_- \mathrm{Im}(\tilde{\varepsilon}_j), \mathrm{Im}(\tilde{\varepsilon}_j) \Big) \ge c \| \tilde{\varepsilon}_j \|_{H^1}^2, \quad \forall t\ge T_0.
\end{equation}
In the following, we always assume $T_0$ is large enough so that the above holds. Set $Q_j= g_j Q$, then we have 
\begin{equation} \begin{aligned}
    &\Big( L_+ \mathrm{Re}(\tilde{\varepsilon}_j), \mathrm{Re}(\tilde{\varepsilon}_j) \Big)+ \Big( L_- \mathrm{Im}(\tilde{\varepsilon}_j), \mathrm{Im}(\tilde{\varepsilon}_j) \Big) \\
    = &\ \int |\nabla \tilde{\varepsilon}_j|^2+ \int |\tilde{\varepsilon}_j|^2+ \int \phi_{Q^2} |\tilde{\varepsilon}_j|^2+ 2\int \phi_{\mathrm{Re}(Q \tilde{\varepsilon}_j)} \mathrm{Re}(Q \tilde{\varepsilon}_j) \\[4pt]
    = &\ \frac{1}{\lambda_j^2} \Bigg[ \int |\nabla \varepsilon_j|^2- 2\beta_j \mathrm{Im}\Big(\int  \nabla \varepsilon_j \overline{\varepsilon_j} \Big) + |\beta_j|^2 \int |\varepsilon_j|^2 - 4 \mu_j \mathrm{Im} \Big(\int \nabla \varepsilon_j x \overline{\varepsilon_j}\Big)\\
    &\;\;\;\;\;\;\;\;\; + 4 \beta_j \mu_j \cdot \int x |\varepsilon_j|^2 + 4 \mu_j^2 \int |x\varepsilon_j|^2  \Bigg]\\[4pt]
    &\;\;+  \int |\varepsilon_j|^2 + \frac{1}{\lambda_j^2} \int \phi_{Q_j^2} |\varepsilon_j|^2+  \frac{2}{\lambda_j^2} \int \phi_{\mathrm{Re}(\varepsilon_j \overline{Q_j})} \mathrm{Re} (\varepsilon \overline{Q_j}). 
\end{aligned} \end{equation}
Therefore we obtain 
\begin{equation} \begin{aligned}
    &\Big( L_+ \mathrm{Re}(\tilde{\varepsilon}_j), \mathrm{Re}(\tilde{\varepsilon}_j) \Big)+ \Big( L_- \mathrm{Im}(\tilde{\varepsilon}_j), \mathrm{Im}(\tilde{\varepsilon}_j) \Big) \\
    = &\ \frac{1}{\lambda_j^2} \Bigg( \int |\nabla \varepsilon_j|^2+ \int \phi_{|R_j|^2} |\varepsilon_j|^2 - \frac{1}{2\pi^2}\int \big| \nabla \phi_{\mathrm{Re}(\varepsilon_j \overline{{R_j}})} \big|^2  + \Big( \lambda_j^2+ |\beta_j|^2 \Big) \int |\varepsilon_j|^2 \Bigg)\\[4pt]
    &\;\; + \frac{2}{\lambda_j^2} \Bigg( 2 \mu_j \beta_j \cdot \int x |\varepsilon_j|^2 -  \beta_j  \mathrm{Im} \Big(\int \nabla \varepsilon_j \overline{\varepsilon_j} \Big) + 2 \mu_j^2 \int |x\varepsilon_j|^2 - 2\mu_j \mathrm{Im} \Big(\int x \nabla \varepsilon_j \overline{\varepsilon_j} \Big) \Bigg)\\
    &\;\;+ O_N \left( \frac{\| \varepsilon_j \|_{H^1}^2}{a} \right),
\end{aligned} \end{equation}
using $ |R_j(t) - Q_j(t)| \lesssim \frac{C}{a} e^{-c |x|}$ by Lemma \ref{lem properties of admissible functions-2}.
We then deduce that 
\begin{equation} \label{eq coercivity of H_j} 
    \H_j(\varepsilon_j) \ge c \| \varepsilon_j \|_{H^1}^2- \frac{C_N}{a} \| \varepsilon_j \|_{H^1}^2 \ge c\int \varphi_j ( |\nabla \varepsilon|^2+ |\varepsilon|^2)- \frac{C_N}{a} \| \varepsilon \|_{H^1}^2,
\end{equation}
where we set
\begin{equation} \begin{aligned}
    \H_j(\varepsilon_j) = &\ \int |\nabla \varepsilon_j|^2+ \int \phi_{|R_j|^2} |\varepsilon_j|^2- \frac{1}{2\pi^2}\int \big| \nabla \phi_{\mathrm{Re}(\varepsilon_j \overline{{R_j}})} \big|^2 \\[4pt]
    &\ + \Big( \lambda_j^2+ |\beta_j|^2 \Big) \int |\varepsilon_j|^2 - 4 \mu_j \int \mathrm{Im} (x \nabla \varepsilon_j \overline{\varepsilon_j}) - 2\beta_j \int \mathrm{Im} (\nabla \varepsilon_j \overline{\varepsilon_j})\\[4pt]
    &\ + 4 \mu_j \beta_j \int x |\varepsilon_j|^2 + 4 \mu_j^2 \int |x|^2 |\varepsilon_j|^2.
\end{aligned} \end{equation}
\noindent
\;\\
Next, we consider the truncated functional
\begin{equation} \begin{aligned}
    \H_{j,\varphi}(\varepsilon) = &\ \int \varphi_j|\nabla \varepsilon|^2+ \int \phi_{|R_j|^2} |\varepsilon|^2- \frac{1}{2\pi^2}\int \big| \nabla \phi_{\mathrm{Re}(\varepsilon \overline{{R_j}})} \big|^2 \\
    &\ + \Big( \lambda_j^2+ |\beta_j|^2 \Big) \int \varphi_j |\varepsilon|^2 - 4 \mu_j \int \varphi_j \mathrm{Im}  (x \nabla \varepsilon \overline{\varepsilon}) - 2\beta_j \int \varphi_j  \mathrm{Im} (\nabla \varepsilon \overline{\varepsilon})\\
    &\ + 4 \mu_j \beta_j \int \varphi_j x |\varepsilon|^2 + 4 \mu_j^2 \int \varphi_j |x|^2 |\varepsilon|^2.
\end{aligned} \end{equation}
We then have using \eqref{eq cutoff} and \eqref{est local varphi}
\begin{equation} \label{eq coercivity of H_j,phi} \begin{aligned}
    \H_{j,\varphi}(\varepsilon)= &\int |\nabla \varepsilon_j|^2+ \int \phi_{|R_j|^2} |\varepsilon|^2- \frac{1}{2\pi^2}\int \big| \nabla \phi_{\mathrm{Re}(\varepsilon \overline{{R_j}})} \big|^2 \\
    &\ + \Big( \lambda_j^2+ |\beta_j|^2 \Big) \int  |\varepsilon_j|^2 - 4 \mu_j \int \mathrm{Im}  (x \nabla \varepsilon_j \overline{\varepsilon_j}) - 2\beta_j \int  \mathrm{Im} (\nabla \varepsilon_j \overline{\varepsilon_j})\\
    &\ + 4 \mu_j \beta_j \int  x |\varepsilon_j|^2 + 4 \mu_j^2 \int  |x|^2 |\varepsilon_j|^2 + O \left( \frac{\| \varepsilon \|_{H^1}^2}{a} \right) \\
    = &\ \H_j(\varepsilon_j)+ \int (1- \varphi_j) \phi_{|R_j|^2} |\varepsilon|^2- \frac{1}{2\pi^2}\int \big| \nabla \phi_{\mathrm{Re} \big( (1-\sqrt{\varphi_j}) \varepsilon \overline{R_j} \big)} \big|^2 \\
    &\ + 2 \int \phi_{\mathrm{Re}( \varepsilon_j \overline{R_j} )} \mathrm{Re} \big( (1-\sqrt{\varphi_j}) \varepsilon \overline{R_j}  \big) + 2 \int \phi_{\mathrm{Re}( (1- \sqrt{\varphi_j})\varepsilon \overline{R_j}  )} \mathrm{Re} \big(  \varepsilon \overline{R_j}  \big) + O \left( \frac{\| \varepsilon \|_{H^1}^2}{a} \right) \\
    \ge &\ c\int \varphi_j ( |\nabla \varepsilon|^2+ |\varepsilon|^2)- \frac{C_N}{a} \| \varepsilon \|_{H^1}^2.
\end{aligned} \end{equation}

Finally, using the fact the $\varphi_j$'s in Lemma \ref{lem cutoff} form a partition of unity, we write
\begin{equation} \begin{aligned}
    \G(\varepsilon)= &\ \sum_{j=1}^m \H_{j,\varphi}(\varepsilon)+ 2\int \phi_{\mathrm{Re} (\varepsilon \overline{R})} |\varepsilon|^2- \frac{1}{8\pi^2} \int |\nabla \phi_{|\varepsilon|^2}|^2 \\
    &\ + \sum_{j \neq k} \int \phi_{\mathrm{Re} (R_k \overline{R_j})} |\varepsilon|^2- \frac{1}{2\pi^2}\sum_{j \neq k} \int \nabla \phi_{\mathrm{Re} (\varepsilon \overline{R_j})} \cdot \nabla \phi_{\mathrm{Re} (\varepsilon \overline{R_k})}. 
\end{aligned} \end{equation}
The first term is controlled by \eqref{eq coercivity of H_j,phi}, whereas the other terms in the first line are $O \big( t^{-N/4} \| \varepsilon \|_{H^1}^2 \big)$ by Proposition \ref{prop accuracy of approximate solution} and the bootstrap assumption \eqref{eq bootstrap assumption}. The two terms in the second line are $O_N \big( e^{-ca}\| \varepsilon \|_{H^1}^2 \big)$ and $O_N \Big( \frac{\| \varepsilon \|_{H^1}^2}{a} \Big)$, respectively. 
\;\\
We thus have
\begin{equation}
    \G(\varepsilon) \ge c \| \varepsilon \|_{H^1}^2- \frac{C_N}{a} \| \varepsilon \|_{H^1}^2
\end{equation}
and hence conclude the claim by taking $T_0(N) \gg1 $ large enough.
\end{proof}

\begin{Prop} \label{prop estimate on G(epsilon)}
Let us assume \eqref{eq bootstrap assumption} in Proposition \ref{prop bootstrap}. For $ N \geq 2$ and $T_0=T_0(N) \gg1$ large enough, there exists $C, C_N >0$ such that
 \begin{align}\label{G(epsilon) estimate}
     & |\G(\varepsilon(t))|   \leq \frac{C}{N} t^{- \frac{N}{2}} + C_N t^{- \frac{3N}{4}},
 \end{align}
for all $t \in [T_*,T_n]$.
\end{Prop}

\begin{proof}
(1) Using integration by parts, we have
\begin{equation} \begin{aligned}
    \frac{\d \G_1}{\dt}(\varepsilon) = &\ -2\mathrm{Im} \int i\partial_t \overline{\varepsilon} \left( \Delta \varepsilon- \phi_{|R|^2} \varepsilon- 2\phi_{\mathrm{Re} (\varepsilon \overline{R})} R- \mathcal{N}(\varepsilon) \right) \\
    &\ +4\mathrm{Re} \int \phi_{\mathrm{Re} (\varepsilon \overline{R})} \varepsilon \partial_t \overline{R}+ 2\int \phi_{\mathrm{Re} (\partial_t R \overline{R})} |\varepsilon|^2+ 2\mathrm{Re} \int \phi_{|\varepsilon|^2} \varepsilon \partial_t \overline{R}.
\end{aligned} \end{equation}
Here, by \eqref{eq equation of epsilon}, Proposition \ref{prop accuracy of approximate solution}, the control of $S_j$ via 
\begin{align} \label{est control S_j}
    |S_j(t,x)| \leq C Mod(t) e^{- c_N |x - \alpha_j|},\;\; x \in \R^4,\;\; t\geq T_0
\end{align}
for which we use Lemma \ref{lem properties of admissible functions-2} and $T_0(N) \gg1 $ large, and finally \eqref{some estimates} for $\mathcal{N}(\varepsilon)$, the first line is of order 
$$O \big( Mod \| \varepsilon \|_{H^1} \big)+ O_N \left( \frac{1}{a^{N+2}} \| \varepsilon \|_{H^1} \right).$$
For the second line, we use the fact that choosing $T_0 = T_0(N) \gg1$ large enough, we have 
\begin{align}
    & |V_j^{(N)}(y_j)| \leq C e^{- c|y_j|} + C_N a^{-1}e^{- c_N|y_j|},\;\;|R_{g,j}^{(N)}(x)| \leq C e^{- c|x - \alpha_j|} + C_N a^{-1} e^{- c|x - \alpha_j|}, \\[3pt]
    & |M_j^{(N)}| + |B_j^{(N)}| + |D_j^{(N)}| \leq  C a^{-2} + C_N a^{-3}.
\end{align}
Then recalling \eqref{param-est-hyp}, \eqref{param-est-par}  and taking the derivative of the modulation as in \eqref{eq time deriv}, we calculate using $R_j = g_j V_j$
\begin{equation} \begin{aligned}
    \partial_t R_j &= \lambda_j^2 \left( - \lambda_j \dot{\alpha}_j \cdot \nabla V_j +  i(\dot{\gamma}_j+ \dot{\beta}_j \cdot \alpha_j + \dot{\mu}_j |\alpha_j|^2) V_j(y_j) \right) e^{i\gamma_j+ i\beta \cdot x + \mu_j |x|^2}\\
    &\;\;\;+ O_N \left( \frac{1}{a^3} \right) e^{-c_N|x-\alpha_j|} \\[3pt]
    &= -2(\beta_j + 2\mu_j \alpha_j)  \cdot \nabla R_j+ i\big( \lambda_j^2 + |\beta_j|^2 + 4 \mu_j \beta_j \cdot \alpha_j + 4 \mu_j^2 |\alpha_j|^2 \big) R_j\\
    &\;\;\;\;+ \left( O_N \Big(\frac{1}{a^3}\Big)+ O(Mod) \right) e^{-c_N|x-\alpha_j|}\\
     &= -2\beta_j \cdot \nabla R_j+ i\big( \lambda_j^2 + |\beta_j|^2 \big) R_j + \left( O_N \Big(\frac{1}{a^3}\Big) +  O\Big(Mod + \frac{C}{t}\Big) \right) e^{-c_N|x-\alpha_j|}.
\end{aligned} \end{equation}

Hence combining the latter with \eqref{eq modulation estimate} in Proposition \ref{prop modulation estimate}, we infer
\begin{align} \label{eq estimate of G_1} 
    \frac{\d \G_1}{\dt}(\varepsilon) = \sum_{j=1}^m \Bigg( 4\Big( & \lambda_j^2 + |\beta_j|^2 \Big) \int \phi_{\mathrm{Re}( \varepsilon \overline{R})} \mathrm{Im} (\varepsilon \overline{R_j}) - 8\int \phi_{\mathrm{Re} (\varepsilon \overline{R})} \mathrm{Re} (\varepsilon \beta_j \cdot \nabla \overline{R_j})\\ \nonumber
    &- 4\int \phi_{\mathrm{Re} (\beta_j \cdot \nabla R_j \overline{R_j})} |\varepsilon|^2 \Bigg)+ O \left( \frac{\| \varepsilon \|_{H^1}^2} {t} \right)+ O_N \left( \frac{\| \varepsilon \|_{H^1}}{a^{N+2}}+ \| \varepsilon \|_{H^1}^3 \right) ,
\end{align} 
where we use Hardy-Littlewood-Sobolev's estimate and in particular, considering $R_j, R_k$ for $j \neq k$, the distance estimates of \cite[Lemma 2.3]{Wu} which equally hold in $d =4$ dimensions.\\[4pt]
(2) (For the hyperbolic case) By \eqref{param-est-hyp} and \eqref{eq cutoff}, we have
\begin{equation}
    \frac{\d \G_2}{\dt}(\varepsilon)= \sum_{j=1}^m 2\Big( \lambda_j^2+ |\beta_j|^2 \Big) \int \varphi_j \mathrm{Im} (i\partial_t \varepsilon \overline{\varepsilon})+ O\left( \frac{\| \varepsilon \|_{H^1}^2}{t} \right). 
\end{equation}
Thus using \eqref{eq equation of epsilon}, \eqref{some estimates}, the control of $S_j$ in \eqref{est control S_j} and \eqref{param-est-hyp} we conclude
\begin{equation} \begin{aligned}
    \frac{\d \G_2}{\dt}(\varepsilon) &= \sum_{j=1}^m 2\Big( \lambda_j^2+ |\beta_j|^2 \Big) \int \varphi_j \mathrm{Im} (2\phi_{\mathrm{Re}(\varepsilon \overline{R})} R \overline{\varepsilon})+ O\left( \frac{\| \varepsilon \|_{H^1}^2}{t} \right) \\
    &\quad + O_N \left( \frac{\| \varepsilon \|_{H^1}}{a^{N+1}}+ \| \varepsilon \|_{H^1}^3 \right) + O( Mod \| \varepsilon \|_{H^1}).
\end{aligned} \end{equation}
Therefore using the modulation estimate \eqref{eq modulation estimate} and the localization property of $\varphi_j$ in \eqref{est local varphi}, we further infer 
\begin{equation} \label{eq estimate of G_2} \begin{aligned}
    \frac{\d \G_2}{\dt}(\varepsilon) &= \sum_{j=1}^m- 4\Big( \lambda_j^2 + |\beta_j|^2 \Big) \int \phi_{\mathrm{Re}(\varepsilon \overline{R})} \mathrm{Im} (\varepsilon \overline{R_j})+ O\left( \frac{\| \varepsilon \|_{H^1}^2}{t} \right) \\
    &\quad + O_N \left( \frac{\| \varepsilon \|_{H^1}}{a^{N+2}}+ \| \varepsilon \|_{H^1}^3 \right).
\end{aligned} \end{equation} 
\;\\
(3) (For the hyperbolic case) Similarly, we can compute
\begin{equation} \label{eq estimate of G_3} \begin{aligned}
    \frac{\d \G_3}{\dt} &= \sum_{j=1}^m -4\beta_j \int \varphi_j \mathrm{Re}(i\partial_t \varepsilon \nabla \overline{\varepsilon})+ O\left( \frac{\| \varepsilon \|_{H^1}^2}{t} \right) \\
    &= \sum_{j=1}^m -4\beta_j \int \varphi_j \mathrm{Re} \left( 2\phi_{\mathrm{Re} (\varepsilon \overline{R})} R \nabla \overline{\varepsilon}+ \phi_{|R|^2} \varepsilon \nabla \overline{\varepsilon} \right)+ O\left( \frac{\| \varepsilon \|_{H^1}^2}{t} \right) \\
    &\quad + O_N \left( \frac{\| \varepsilon \|_{H^1}}{a^{N+2}}+ Mod \| \varepsilon \|_{H^1}+ \| \varepsilon \|_{H^1}^3 \right) \\
    &= \sum_{j=1}^m \left( 8\int \phi_{\mathrm{Re} (\varepsilon \overline{R})} \mathrm{Re}(\varepsilon \beta_j \varphi_j \cdot \nabla \overline{R})+ 4\int \phi_{\mathrm{Re} (\beta_j \varphi_j \cdot \nabla R \overline{R})} |\varepsilon|^2 \right) \\
    &\quad +O\left( \frac{\| \varepsilon \|_{H^1}^2}{t} \right)+ O_N \left( \frac{\| \varepsilon \|_{H^1}}{a^{N+2}}+ Mod \| \varepsilon \|_{H^1}+ \| \varepsilon \|_{H^1}^3 \right) \\
    &= \sum_{j=1}^m \left( 8\int \phi_{\mathrm{Re} (\varepsilon \overline{R})} \mathrm{Re}(\varepsilon \beta_j \cdot \nabla \overline{R_j})+ 4\int \phi_{\mathrm{Re} (\beta_j \cdot \nabla R_j \overline{R_j})} |\varepsilon|^2 \right) \\
    &\quad +O\left( \frac{\| \varepsilon \|_{H^1}^2}{t} \right)+ O_N \left( \frac{\| \varepsilon \|_{H^1}}{a^{N+2}}+ \frac{\| \varepsilon \|_{H^1}^2} {a^2}+ \| \varepsilon \|_{H^1}^3 \right).
\end{aligned} \end{equation}

Combining \eqref{eq estimate of G_1}, \eqref{eq estimate of G_2}, \eqref{eq estimate of G_3} together, we deduce
\begin{equation}
    \left| \frac{\d}{\dt} (\G_1+ \G_2+ \G_3)(\varepsilon(t)) \right| \le \frac{C\| \varepsilon \|_{H^1}^2}{t}+ C_N \left( \frac{\| \varepsilon \|_{H^1}}{a^{N+2}}+ \| \varepsilon \|_{H^1}^3 \right).
\end{equation}
Then integrating and using \eqref{eq bootstrap assumption}, we have 
\begin{align} \label{eq estimate of G_123}
    |\G_1(\varepsilon(t))+ \G_2(\varepsilon(t))+ \G_3(\varepsilon(t))| \le& \int_t^{T_n} C \tau^{- \frac{N}{2} -1} + C_N \tau^{-\frac{3N}{4} - 1}\; \d \tau\\ \nonumber
    \le& \frac{C}{N} t^{- \frac{N}{2}} + C_N t^{-\frac{3N}{4}}.
\end{align}

(4) We have, using \eqref{eq equation of epsilon} and \eqref{eq nonlin epsilon},
\begin{align}
    \frac{\d}{\d t} \Big(\int |x|^2 |\varepsilon |^2 \Big) =&\;\; 4 \mathrm{Im} \Big(\int x \nabla \varepsilon\overline{\varepsilon} \Big) + 4 \mathrm{Im} \Big(\int |x|^2 \phi_{\mathrm{Re}(\varepsilon \overline{R})} R \overline{\varepsilon}\Big)+ 2 \mathrm{Im} \Big(\int |x|^2 \Psi \overline{\varepsilon} \Big) ,\\
    &\;\;+ 2 \mathrm{Im} \Big(\int |x|^2 \phi_{|\varepsilon|^2} R \overline{\varepsilon} \Big) + 2\sum_{j =1}^m  \mathrm{Im} \Big(\int |x|^2 S_j e^{i(\gamma_j + \beta_j \cdot x + \mu_j |x|^2)} \overline{\varepsilon} \Big).
\end{align}
Clearly we have 
 \begin{align}
 &\left| \mathrm{Im} \Big(\int x \nabla \varepsilon\overline{\varepsilon} \Big)\right| \leq  \| x \varepsilon\|_{L^2} \| \varepsilon \|_{H^1} \leq t^{-\frac{N}{2}+2},\\
&\left | \mathrm{Im} \Big(\int |x|^2 \Psi \overline{\varepsilon} \Big)\right | \leq \|x \varepsilon \|_{L^2_x} \| \Psi x \|_{L^2_x} \leq \frac{C}{t^{\frac{N}{2}}} \|x \varepsilon \|_{L^2_x} \leq C t^{-\frac{3N}{4}+2},
\end{align}  
where in the second estimate we use Proposition \ref{prop accuracy of approximate solution}.
For the second and the fourth terms we infer from Sobolev's and Hardy-Littlewood-Sobolev's inequality
\begin{align} \label{2}
&\left |  \mathrm{Im} \Big(\int |x|^2 \phi_{\mathrm{Re}(\varepsilon \overline{R})} R \overline{\varepsilon}\Big)\right |\\
\leq&\; \Bigg(\left \| \int \frac{x-y}{|x-y|^2} \mathrm{Re}(\varepsilon \bar{R})(y) \d y\right \|_{L^8_x}  + \left \| \int |x - y|^{-2} \mathrm{Re}(\varepsilon \bar{R})(y) y \d y\right \|_{L^8_x}\Bigg)  \|x R \|_{L^{\frac{8}{5}}_x}  \| \varepsilon \|_{H^1} \\
\leq&\;\; C \|\varepsilon \|_{H^1}^2 \Big(\| R\|_{L^{\frac{8}{5}}}  + \| xR\|_{L^{8}}\Big)\| xR\|_{L^{\frac{8}{5}}}  \leq C t^2 \|\varepsilon \|_{H^1}^2 \le C t^{-\frac{N}{2}+2},
\end{align}
and 
\begin{align}
\left | \mathrm{Im} \Big(\int |x|^2 \phi_{|\varepsilon|^2} R \overline{\varepsilon} \Big)\right |\leq&\; \left \| \int |x - y|^{-2} |\varepsilon|^2(y)  \d y\right \|_{L^4_x} \| x R \|_{L^4_x} \| x \varepsilon \|_{L^2_x}\\
\leq&\;\; C t \|\varepsilon \|_{L^4}^2  \| x \varepsilon \|_{L^2_x} \leq C t^{- \frac{3N}{4}+3}.
\end{align}  
Finally for the last integral we have from \eqref{eq definition of S_j^N} and \eqref{eq modulation estimate} in Proposition \ref{prop modulation estimate} 
 \begin{align}
\left |\mathrm{Im} \Big(\int |x|^2 S_j e^{i(\gamma_j + \beta_j \cdot x + \mu_j |x|^2)} \overline{\varepsilon} \Big)\right | \leq&  \sum_{j = 1}^m \| |x|S_j \|_{L^2_x} \| x \varepsilon \|_{L^2_x}\\
\leq&\;\; C t  \|x \varepsilon \|_{L^2}\Big( t^{-1} \| \varepsilon \|_{H^1} + C_N( t^{- \frac{N}{2} -1}  + \|\varepsilon \|_{H^1}^2) \Big)\\
\leq&\;\; Ct^{- \frac{N}{2} +2} + C_N t^{- \frac{3N}{4} +3}.
\end{align}  
Then collecting these we get
\begin{equation} 
\left| \frac{\d}{\d t} \Big(\int |x|^2 |\varepsilon |^2 \Big) \right| \leq  C t^{- \frac{N}{2} +2} + C_N t^{- \frac{3N}{4} +3},
\end{equation}
and thus
\begin{equation} \label{eq estimate for x epsilon}
    \int |x|^2 |\varepsilon |^2 \leq  \frac{C}{N} t^{- \frac{N}{2} +3 } + C_N t^{- \frac{3N}{4} +4}.
\end{equation}
Therefore, using $ |\mu(t)| \lesssim t^{-2}$, we infer
\begin{align} \label{eq estimate of G_4}
   | \G_4(\varepsilon(t))|  \leq \frac{C}{N} t^{- \frac{N}{2}} + C_N t^{- \frac{3N}{4}}.
\end{align}

(5) By Cauchy-Schwarz and $ |\mu(t)| \lesssim t^{-2}$, 
\begin{equation} \label{eq estimate of G_5}
    |\G_5(\varepsilon(t))| \leq C t^{-2}\|x \varepsilon \|_{L^2_x} \| \varepsilon\|_{H^1} \leq t^{-2} \frac{C}{N} t^{- \frac{N}{4}+ 2}  t^{- \frac{N}{4}} \leq  \frac{C}{N} t^{- \frac{N}{2}}.
\end{equation}
\;\\
Combining \eqref{eq estimate of G_123}, \eqref{eq estimate of G_4} and \eqref{eq estimate of G_5} together, the proof for the hyperbolic case is finished.\\[4pt]

In the parabolic case, we additionally need to show
\begin{equation} \label{eq G_2 cutoff error}
    \sum_{j=1}^m \Big( \lambda_j^2+ |\beta_j|^2 \Big) \int \Big( \partial_t \varphi_j |\varepsilon|^2+ 2\nabla \varphi_j \mathrm{Im}(\nabla \varepsilon \overline{\varepsilon}) \Big)= O\left( \frac{\| \varepsilon \|_{H^1}^2}{t} \right)
\end{equation}
and
\begin{equation} \label{eq G_3 cutoff error} \begin{aligned}
    \sum_{j=1}^m \beta_j \int \bigg( \nabla \varphi_j \Big( 2|\nabla \varepsilon|^2+ 2\phi_{\mathrm{Re} (\varepsilon \overline{R})} \mathrm{Re} (& \varepsilon \overline{R}) + \phi_{|R|^2} |\varepsilon|^2 \Big) \\
    &+ \partial_t \varphi_j \mathrm{Im} (\nabla \varepsilon \overline{\varepsilon}) \bigg) =  O\left( \frac{\| \varepsilon \|_{H^1}^2}{t} \right).  
\end{aligned} \end{equation}
We know $|\partial_t \varphi_j|+ |\nabla \varphi_j|= O(t^{-1/2})$. The second estimate then holds because $|\beta_j| \le t^{-1/2}$. Further we need the assumption that all $\lambda_j's$ are equal, which then in particular implies the first estimate since $\sum \varphi_j=1$. 
\end{proof}

\begin{proof}[Proof of \eqref{eq bootstrap conclusion}] It remains to check the first line. Note that \eqref{eq estimate for x epsilon} implies the bound for $\| x\varepsilon \|_{L^2}$. Then combining \eqref{G(epsilon) estimate} in Proposition \ref{prop estimate on G(epsilon)} and the coercivity of $\G(\varepsilon)$ in Proposition \ref{prop coercivity}, we take $ N $ and $T_0 = T_0(N)$ subsequently large to conclude the claim.
\end{proof}

\appendix
\section{Auxiliary details for the non-degeneracy}\label{appendix:Nondegeneracy}

\begin{Lemma}
    The coefficients $\Gamma_{\el}(t)$ defined in \eqref{eqn:GammaCoefficients} form a sequence of positive and decreasing functions in the sense that
    \[ \Gamma_{\el}(t) > \Gamma_{\el+1}(t) > 0, \quad \forall \el \geq 0 . \]
\end{Lemma}

\begin{proof}
    Using Rodrigues' formula for Legendre Polynomials,
    \begin{equation} \begin{aligned}
        \Gamma_n(t) &= \frac{2n+1}{2(n+1)^2} \int_{-1}^1 \frac{1}{1+t^2-2t\eta} \frac{1}{2^n n!} \frac{\d^n}{\d \eta^n} (\eta^2-1)^n \d \eta \\
        &= \frac{(-1)^n (2n+1)}{2^{n+1} (n+1)^2 n!} \int_{-1}^1 \left( \frac{\d^n}{\d \eta^n} \frac{1}{1+t^2-2t\eta} \right) (\eta^2-1)^n \d \eta \\
        &= \frac{2n+1}{2 (n+1)^2} \cdot \frac{1}{t} \int_{-1}^1 \frac{(1-\eta^2)^n}{2^{n+1}(\frac{1+t^2}{2t}- \eta)^{n+1}} \d \eta.
    \end{aligned} \end{equation}
    Notice that the function inside the last integral is nonnegative, and
    \[ 0 \leq \frac{1-\eta^2}{2 \left( \frac{1+t^2}{2t}-\eta \right)} \leq 1, \]
    and since the coefficient
    \[ \frac{2n+1}{2(n+1)^2} \] 
    describes a decreasing function of $n$, we conclude that $0 < \Gamma_{\el+1}(t) < \Gamma_{\el}(t)$.
\end{proof}

We complete this appendix by providing a proof of \eqref{eqn:HeatKernel}.

\begin{proof}[Proof of \eqref{eqn:HeatKernel}]
    We begin by decomposing $e^{a t u}$ in the basis of Legendre polynomials, where $a > 0$, $t>0$ and $u \in [-1,1]$. This is,
    \[ e^{atu} = \sum_{\el = 0}^{\infty} C_{\el, a}(t) P_{\el}(u), \quad C_{\el,a}(t) = \int_{-1}^1 e^{atu} P_{\el}(u) \frac{2\el+1}{2} \d u .
    \]
    We may compute, using Rodrigues' formula, to find
    \begin{align*}
        C_{\el, a}(t) &= \frac{2\el+1}{2^{\el}\el!} (at)^{\el} \int_{-1}^1 e^{atu}(1-u^2)^{\el} \d u = \frac{2\el+1}{2^{\el}\el!} (at)^{\el} \int_{-1}^1 \cosh(iatu)(1-u^2)^{\el} \d u \\
        &= \frac{2\el+1}{2^{\el}\el!} (at)^{\el} \frac{\pi^{1/2} \Gamma(\el+1)}{2\left(\frac{1}{2}iat\right)^{\el+1/2}} J_{\el+1/2}(iat) = \frac{2\el+1}{2}i_{\el}(at),
    \end{align*}
    where $J_{n}$ is the Bessel function of the first kind of order $n$, and $i_{n}(z)$ is the modified spherical Bessel function of the first kind of order $n$. Setting $u= \frac{x}{\abs{x}} \cdot \frac{y}{\abs{y}}$ and $t=\abs{x}\abs{y}$ for $x,y \in \R^4$, we find, by using the addition theorem for spherical harmonics, that
    \[ e^{a (x \cdot y)} = \sum_{\el=0}^{\infty} \frac{2\el+1}{2} \cdot \frac{\omega_3}{(\el+1)^2} i_{\el}(a \abs{x}\abs{y}) \sum_{j=1}^{(\el+1)^2} S_{\el j} \left( \frac{x}{\abs{x}} \right) S_{\el j} \left( \frac{y}{\abs{y}} \right) .\]
    We conclude by combining this with the expression for the heat kernel of the Laplacian $\Delta$ on $\R^4$,
    \begin{equation}
        e^{t\Delta}(x,y) = \frac{1}{(4\pi t)^2} e^{-\frac{\abs{x-y}^2}{4t}} .
    \end{equation}
\end{proof}

\section{Proof of Propositions \ref{hyp} and \ref{parab}} \label{sec:furtherdetailsmbod}
Let $ m   \in \Z_{\geq 2} $ and $ \alpha = (\alpha_1, \alpha_2, \dots, \alpha_m) \in \R^{ 4m} $ be the position configuration of $ m $ bodies with center $ \alpha_j \in \R^4$  and masses $m_j \in \R_+$ in euclidean space.   We define our \emph{$m$-body law} to be the Newtonian system 
	
\begin{align} \label{m-bd}
	\left\{ \begin{aligned}
		\dot{\alpha}_j(t) &=  \beta_j(t),\; 
		\displaystyle \\
        \dot{\beta}_j(t) &= -\sum_{ k \neq j } m_k \frac{\alpha_j - \alpha_k}{|\alpha_j - \alpha_k|^4},
	\end{aligned} \right. \quad 1 \leq  j \leq m.
\end{align}
    
	The system \eqref{m-bd} has a  first integral
	\[
	H(\alpha(t), \beta(t))= \frac{1}{2} \sum_{j =1}^m m_j |\beta_j(t)|^2 - \f12 \sum_{j < k} \frac{ m_j m_k}{|\alpha_j - \alpha_k|^2},
	\] 
	i.e. $ H = K  - U$ with \emph{kinetic} and \emph{potential energy}
	\begin{align*}
		K(\beta) = \frac{1}{2} \sum_{j=1}^m  m_j |\beta_j(t)|^2,\;\; \;U(\alpha) =  \f12 \sum_{j < k} \frac{m_j m_k}{|\alpha_j - \alpha_k|^2}.
	\end{align*}
\;\\
The purpose of this section is to give proofs of Proposition \ref{hyp} and \ref{parab} in Section \ref{sec:n-body-intro}, which of course equally apply to the above $m$-body problem \eqref{m-bd} with general masses. 
    
\begin{proof}[Proof of Proposition \ref{hyp}] \

(1) Since $(\alpha, \beta)$ is a hyperbolic solution, we have $|\dot{\beta}(t)| \lesssim t^{-3}$, so $\lim_{t \to +\infty} \beta(t)$ exists. Let the limit be $\frac{v}{2}$. Then $|\beta(t)- \frac{v}{2}| \lesssim t^{-2}$, which implies $|\frac{\d}{\dt} (\alpha(t)- vt)| \lesssim t^{-2}$. We then deduce that $\lim_{t \to +\infty} (\alpha(t)- vt)$ exists. Let the limit be $x$. Then we have
\begin{equation}
    \alpha(t)= x+ vt + o(1),\;\;\beta(t)=  \frac{v}{2}+ o(1),\;\;\text{as}\;\; t \to \infty.
\end{equation}
If we assume $\alpha(t) \in \mathcal{X}$, then we have $x,v \in \mathcal{X}$. We further have $v \in \mathcal{Y}$ by hyperbolicity.
    
(2) We use a fixed point argument similar to the proof of Proposition \ref{prop hyperbolic trajectory}. Let
\begin{equation}
	X= \bigg\{ (\alpha, \beta) \in C^1 \big( [T_0, \infty), \R^{4m} \times \R^{4m} \big) \ \big| \ \big\| (\alpha- vt, 2\beta - v) \big\| \le 1 \bigg\},
\end{equation}
where we set
\begin{equation}
	\|  (\alpha, \beta) \| = \sup_{t \ge T_0} \sum_{j=1}^m \big( t^{\frac{1}{3}} |\alpha_j(t)|+ t^{\frac{5}{3}}|\beta_j(t)| \big),\;\; T_0 \gg1.
\end{equation}
Define the map $\Gamma: X \to X,\; \Gamma = (\Gamma^{(1)}, \Gamma^{(2)})$ by
	\begin{align}
		&(\Gamma^{(1)} \begin{psmallmatrix}
			\alpha\\[1pt] \beta
		\end{psmallmatrix})_j(t) := v_jt- \int_t^\infty (2\beta_j(\tau)-  v_j) \;d \tau,\\
		&(\Gamma^{(2)}\begin{psmallmatrix}
			\alpha\\[1pt] \beta
		\end{psmallmatrix})_j(t) := \frac{v_j}{2}+ \| Q \|_{L^2}^2 \int_t^\infty \sum_{k \neq j} \frac{\alpha_{jk}(\tau)}{|\alpha_{jk}(\tau)|^4} \;d \tau.
	\end{align}
We can directly check that when $T_0$ is large enough, $\Gamma$ is a contraction map from $X$ to $X$. Then the fixed point is the desired solution.
\end{proof}

\begin{Rem}
Part (2) also follows in a general context, i.e., Theorem 1.4 and Corollary 1.5 in \cite{Liu-Yan-Zhou}. The expansion in Proposition \ref{hyp} is then calculated from the M\o{}ller transform \footnote{viewed as a classical analogue of the wave operators of quantum Hamiltonians}  and symplectic intertwining
\begin{align*}
	\Omega =  \lim_{ t \to \infty} \Phi_{-t} \circ \Phi_{0,t},\;\;\; \Omega \circ  \Phi_{0,t} = \Phi_t \circ \Omega,
\end{align*}
where $ \Phi_t$ is the flow map and $ \Phi_{0,t}$  the free Galilean flow on a suitable space of escape orbits, see  Theorem 4.2 in \cite[Section 4.1]{FKM}.    A simple modification then implies Proposition \ref{hyp}. 
\end{Rem}

\begin{proof}[Proof of Proposition \ref{parab}] \

Let us consider homothetic solutions. First from the conditional minimality of $b$, using the method of Lagrange multiplier, there exists $c \in \R$ such that 
\begin{equation} \label{eq vector equation}
    c b_j = \|Q\|_{L^2}^2 \sum_{k \neq j} \frac{b_j- b_k}{|b_j- b_k|^4}, \quad \forall 1 \le j \le m.
\end{equation}
Multiplying both sides of \eqref{eq vector equation} by $b_j$ and summing over $j$, we deduce that $ c = 2 U(b)>0$.

We make the ansatz $\alpha(t) = f(t) b$ and $\beta(t)= \f12 \dot{\alpha}(t)$. Then \eqref{m-body} holds if and only if
\begin{equation}
    f''(t) =  -\frac{2c}{f^3(t)}.
\end{equation}
We see that for any $\eta \ge 0$, $f(t)= \big( \sqrt{8c} \ t+ \eta \big)^{\f12}$ solves the equation. Thus, all $(\alpha(t), \beta(t))$ given in Proposition \ref{parab} are solutions to \eqref{m-body}.
\end{proof}

	\vspace{3cm}
	
   
   

	\bibliographystyle{alpha}

	\vspace{1cm}


\end{document}